\documentclass[11pt]{article}

\usepackage[utf8]{inputenc}
\usepackage[T1]{fontenc}
\usepackage[USenglish]{babel}
\usepackage[a4paper,margin=32mm,headheight=14pt]{geometry}
\usepackage{lmodern}
\usepackage{microtype}
\usepackage{csquotes}

\usepackage{amsmath,amssymb,amscd}
\usepackage{amsthm}
\usepackage{mathtools}
\usepackage{mathrsfs}
\usepackage{enumitem}
\usepackage{aliascnt}
\usepackage{xcolor}
\definecolor{baseC}{HTML}{1B7F79}
\definecolor{excC}{HTML}{D9731A}
\definecolor{pullC}{HTML}{6A3FB5}
\definecolor{special}{HTML}{1F5FA8}
\definecolor{nonspecial}{HTML}{C0392B}
\definecolor{curve}{HTML}{1E8449}
\definecolor{trivial}{HTML}{5D6D7E}
\definecolor{panel}{HTML}{F4F6F7}
\usepackage{fancyhdr}
\usepackage{tikz}
\usepackage{tikz-cd}
\graphicspath{{images/}}
\usepackage{booktabs}
\usepackage{float}
\usetikzlibrary{arrows.meta,positioning,calc,bending,fit,backgrounds}
\usepackage{xurl}
\usepackage[
  colorlinks=true,
  linkcolor=blue!45!black,
  citecolor=green!35!black,
  urlcolor=blue!55!black,
  pdfauthor={Alex Junior Gomez Saltachin},
  pdftitle={Derived categories of log del Pezzo surfaces with 1/3(1,1) singularities},
  pdfsubject={Algebraic geometry},
  pdfkeywords={Log del Pezzo surfaces, canonical stacks, exceptional collections, semiorthogonal decompositions, quotient singularities, Brauer groups, special McKay correspondence}
]{hyperref}
\usepackage[capitalize,nameinlink]{cleveref}

\theoremstyle{plain}
\newtheorem{theorem}{Theorem}[section]
\newaliascnt{proposition}{theorem}
\newtheorem{proposition}[proposition]{Proposition}
\aliascntresetthe{proposition}
\newaliascnt{lemma}{theorem}
\newtheorem{lemma}[lemma]{Lemma}
\aliascntresetthe{lemma}
\newaliascnt{corollary}{theorem}
\newtheorem{corollary}[corollary]{Corollary}
\aliascntresetthe{corollary}

\theoremstyle{definition}
\newaliascnt{definition}{theorem}
\newtheorem{definition}[definition]{Definition}
\aliascntresetthe{definition}
\newaliascnt{example}{theorem}
\newtheorem{example}[example]{Example}
\aliascntresetthe{example}
\theoremstyle{remark}
\newaliascnt{remark}{theorem}
\newtheorem{remark}[remark]{Remark}
\aliascntresetthe{remark}

\crefname{theorem}{Theorem}{Theorems}
\crefname{proposition}{Proposition}{Propositions}
\crefname{lemma}{Lemma}{Lemmas}
\crefname{corollary}{Corollary}{Corollaries}
\crefname{definition}{Definition}{Definitions}
\crefname{example}{Example}{Examples}
\crefname{remark}{Remark}{Remarks}

\newcommand{\A}{\mathbb A}
\newcommand{\Pp}{\mathbb P}
\newcommand{\F}{\mathbb F}
\newcommand{\C}{\mathbb C}
\newcommand{\Z}{\mathbb Z}
\newcommand{\Db}{D^b}
\newcommand{\coh}{\operatorname{coh}}
\newcommand{\Sing}{\operatorname{Sing}}
\newcommand{\Perf}{\operatorname{Perf}}
\newcommand{\Pic}{\operatorname{Pic}}
\newcommand{\rk}{\operatorname{rk}}
\newcommand{\Hom}{\operatorname{Hom}}
\newcommand{\Ext}{\operatorname{Ext}}
\newcommand{\Bl}{\operatorname{Bl}}

\newcommand{\Cl}{\mathrm{Cl}}
\newcommand{\Br}{\mathrm{Br}}
\newcommand{\Spec}{\mathrm{Spec}}

\newcommand{\RHom}{\operatorname{RHom}}
\newcommand{\XX}{\mathcal X}

\newcommand{\muthree}{\boldsymbol{\mu}_3}
\newcommand{\OO}{\mathcal O}

\usepackage[
  backend=biber,
  style=numeric,
  sorting=nyt,
  giveninits=true,
  doi=true,
  url=false,
  isbn=false
]{biblatex}
\AtBeginBibliography{\small}

\begin{document}

\title{\bfseries Derived categories of log del Pezzo surfaces with \texorpdfstring{$\frac13(1,1)$}{1/3(1,1)} singularities}
\author{Alex Junior Gomez Saltachin\\[0.5em]
\small Departamento de Matem\'atica, Pontif\'icia Universidade Cat\'olica do Rio de Janeiro\\
\small Rua Marqu\^es de S\~ao Vicente, 225, Rio de Janeiro, RJ 22451-900, Brazil\\
\small \href{mailto:alex.gomez@pucp.edu.pe}{alex.gomez@pucp.edu.pe}}
\date{}

\maketitle

\begin{abstract}
Let $X$ be a log del Pezzo surface with only $\frac13(1,1)$ singularities. We compare the derived categories of its minimal resolution, canonical stack, and singular surface. For every family in the Corti--Heuberger classification, Orlov's blow-up formula and the Ishii--Ueda correspondence give explicit full exceptional collections on the canonical stacks, of length $12-K_X^2+\frac{4k}{3}$ when $X$ has $k$ singular points. In the one-point index-one cascade, the canonical stacks are blow-ups of the weighted projective stack $\mathcal P(1,1,3)$ at distinct smooth points and admit explicit full strong exceptional collections of line bundles. For quasismooth hypersurfaces $X_{10}\subset\mathbb P(1,2,3,5)$, we construct a full strong exceptional collection of 13 objects with an explicitly computed Euler matrix. The exceptional lattice determines the Brauer group: $\operatorname{Br}(X)=0$ for $k\le5$, while $\operatorname{Br}(X)\cong\mathbb Z/3$ for $k=6$. We construct semiorthogonal decompositions of $D^b(\operatorname{coh}X)$ for the one- and three-point cascades and of a Brauer-twisted derived category for the six-point cascade.
\end{abstract}



\tableofcontents

\section{Introduction}

Let $X$ be a log del Pezzo surface with quotient singularities. Write
$\widetilde X$ for its minimal resolution and $\mathcal X$ for its
canonical stack, the smooth Deligne--Mumford stack whose coarse moduli
space is $X$. These three spaces give rise to the derived categories
\[
D^b(\operatorname{coh}\widetilde X),\qquad
D^b(\operatorname{coh}\mathcal X),\qquad \text{and}\qquad
D^b(\operatorname{coh}X).
\]
We trace how exceptional collections on the resolution pass to the
canonical stack, how each singular point contributes to the collection,
and when the resulting decompositions descend to $X$.

For surfaces with only $\tfrac13(1,1)$ singularities, the geometry is
organized by the six $\mathbb Q$-Gorenstein deformation cascades of
Corti and Heuberger~\cite[Theorems~3, 4, and~6, Corollary~8]{Corti2016}.
The index $k$ records the number of singular points. We write $X_{k,d}$
for the index-one deformation family with $k$ singular points and
$K_X^2=d$, and $Z_k$ for its chosen rigid starting surface. For an
individual member $X\in X_{k,d}$, set $s=K_{Z_k}^2-d$, the number of
smooth-point blow-ups in the cascade.

The cascade picture also places these surfaces in the mirror-symmetric
classification program relating Fano polygons, mutation, periods, and
$\mathbb Q$-Gorenstein deformation classes
\cite{AkhtarEtAl2016,KasprzykNillPrince2017,OnetoPetracci2018}.
In this setting the canonical stack retains the quotient-singularity
data and serves as the categorical $B$-model
\cite{AkhtarEtAl2016,OnetoPetracci2018}.

Our categorical construction follows the same progression from
resolution to stack to singular surface. First, Orlov's projective-bundle
and blow-up formulas~\cite{Orlov1993} give full exceptional collections
on the rational surface $\widetilde X$, which can be chosen to consist
of line bundles
(\cref{prop:log-del-pezzo-rational,cor:resolution-fec}). Next, the
special McKay correspondence of Ishii--Ueda embeds
$D^b(\operatorname{coh}\widetilde X)$ into
$D^b(\operatorname{coh}\mathcal X)$ and adds local exceptional objects
at the stacky points~\cite{Ishii2015}. Finally, Karmazyn--Kuznetsov--Shinder
(KKS) prove a descent theorem that carries decompositions compatible with the
contraction to the singular surface
\cite[Theorems~2.12 and~3.16]{KKS2020}. For the stronger result in the
one-point index-one cascade, we also use compatibility of canonical stacks with
blow-ups at smooth points and the consecutive full strong
weighted-projective collections of Auroux--Katzarkov--Orlov
\cite[Theorem~2.12]{AurouxKatzarkovOrlov2008}.

Using these tools, we construct collections along the cascades, compute
the Brauer obstruction from the exceptional lattice, and obtain
semiorthogonal decompositions on the singular surfaces whenever the
required subcategories are available. We then examine the degree-$10$
hypersurface family in detail.

To state the main theorem, let
\[
\mathcal A_k=
\left\langle
\mathcal O_{\widetilde Z_k},A^{(k)}_2,\ldots,A^{(k)}_{N_k}
\right\rangle
\]
be a full exceptional collection of line bundles on $\widetilde Z_k$,
of length $N_k$. Appendix~\ref{app:explicit-bases} constructs a marked
blow-up presentation (\cref{def:marked-blowup-presentation}) of each
$\widetilde Z_k$ and the collection $\mathcal A_k$. Only the existence
and fullness of these collections are needed below.

We write $E_i$ for exceptional curves over quotient singularities and
$D_j$ for exceptional curves of blow-ups at smooth points on the minimal
resolution. Write $\mathcal D_j$ for the corresponding exceptional
divisors on canonical stacks and $\mathcal E_{p_i,j}$ for the
residual-gerbe objects, which are supported at the stacky points. For a
$\tfrac13(1,1)$-point, we abbreviate $\mathcal E_{p_i}:=
\mathcal E_{p_i,2}$.

\begin{theorem}[Exceptional collections along the cascades]
\label{thm:explicit-cascades-main}
Let $X\in X_{k,d}$, with singular points $p_1,\dots,p_k$, and let
$X\to Z_k$ be the blow-up of $s=K_{Z_k}^2-d$ distinct smooth points.
Let $\gamma\colon\widetilde X\to\widetilde Z_k$ be the induced map of
minimal resolutions, with exceptional curves $D_1,\dots,D_s$. Then
\[
\left\langle
\mathcal O_{\widetilde X},
\mathcal O_{\widetilde X}(D_1),\ldots,
\mathcal O_{\widetilde X}(D_s),
\gamma^*A^{(k)}_2,\ldots,
\gamma^*A^{(k)}_{N_k}
\right\rangle
\]
is a full exceptional collection of line bundles in
$D^b(\operatorname{coh}\widetilde X)$.
Preceding its image under the special McKay embedding $\Phi$ by the
residual-gerbe objects $\mathcal E_{p_1},\dots,\mathcal E_{p_k}$ gives a
full exceptional collection
\[
\left\langle
\mathcal E_{p_1},\ldots,\mathcal E_{p_k},
\Phi(\mathcal O_{\widetilde X}),
\Phi(\mathcal O_{\widetilde X}(D_1)),\ldots,
\Phi(\mathcal O_{\widetilde X}(D_s)),
\Phi(\gamma^*A^{(k)}_2),\ldots,
\Phi(\gamma^*A^{(k)}_{N_k})
\right\rangle
\]
in
$D^b(\operatorname{coh}\mathcal X)$ of length
\[
N_k+s+k=12-d+\frac{4k}{3}.
\]
\end{theorem}

For the index-one cascades, applying Orlov's blow-up formula to collections
of line bundles extends $\mathscr A_k$ through the $s$ blow-ups, and the
special McKay embedding
adds one object for each singular point
(\cref{cor:explicit-cascade-collections}). The resolution collection
consists of line bundles, while its image under $\Phi$ consists in general
of derived objects on $\mathcal X$
(\cref{rem:ishii-ueda-derived-images}). Thus the resulting Ishii--Ueda
collection is not asserted to be strong.

Bright's exact sequence identifies the Brauer group with the cokernel of
the intersection map for the exceptional lattice \cite{Bright2013}. An
elementary lattice calculation then identifies this group with the
Pontryagin dual of $\operatorname{Cl}(X)_{\mathrm{tors}}$
(\cref{cor:brauer-cokernel,lem:brauer-glue-duality}):
\[
\operatorname{Br}(X)\cong
\operatorname{Hom}\!\left(\operatorname{Cl}(X)_{\mathrm{tors}},
\mathbb Q/\mathbb Z\right).
\]
The descent theorem forces $\operatorname{Cl}(X)$ to be torsion-free
whenever the resolution admits an untwisted adherent decomposition
\cite[Corollary~4.8]{KKS2020}, so torsion in the class group obstructs
such a decomposition. For the $X_{k,d}$-families, invariance under smooth
blow-ups reduces the calculation to the six chosen starting surfaces $Z_k$
(\cref{lem:brauer-cascade-invariance}).

\begin{theorem}[The Brauer obstruction]
\label{thm:brauer-cascades-intro}
Let $X$ belong to one of the six Corti--Heuberger cascades or to one of
the additional higher-index families $B_{1,16/3}$ and $B_{2,8/3}$. Let
$k$ denote its number of $\frac13(1,1)$-singularities. Then
\[
\operatorname{Br}(X)\cong
\begin{cases}
0, & 1\le k\le 5,\\
\mathbb Z/3, & k=6.
\end{cases}
\]
\end{theorem}

The orders agree with the defect computed in
\cite[Section~2.1, Proposition~18(b), Lemma~28]{Corti2016}. The additional
information is the lattice-theoretic description in
Appendix~\ref{app:brauer}. \Cref{thm:brauer-cascades-intro} is proved in
\cref{app:brauer-classification}. In the six-point case, the
nonzero torsion class
is represented by the class
\[
\frac13(E_1-E_2+E_3-E_4+E_5-E_6),
\]
where $E_1,\ldots,E_6$ are the exceptional curves of
$\widetilde Z_6$. This class supplies the Brauer twist used in the
descent theorem below.

Write
\[
K(3,1)=\mathbb C[z_1,z_2]/(z_1,z_2)^2
\]
for the Kalck--Karmazyn algebra of a $\frac13(1,1)$-point
\cite{KalckKarmazyn2017}.

For a Brauer class $\alpha\in\operatorname{Br}(X)$, choose an Azumaya
algebra $\mathcal A$ representing $\alpha$, and write
$\coh(X,\alpha)$ for the category of coherent right $\mathcal A$-modules.
The category $D^b(\coh(X,\alpha))$ is independent of the chosen
representative up to Morita equivalence. This Brauer class $\alpha$ is
distinct from the integer twist parameter $b$ used for adherence below.

\begin{theorem}[Descent along the cascades]
\label{thm:descent-cascades-intro}
Let $X\in X_{k,d}$, and let $s$ be as in
\cref{thm:explicit-cascades-main}.
\begin{enumerate}[label=(\roman*)]
\item If $k=1$ or $k=3$, there is a semiorthogonal decomposition
\[
D^b(\operatorname{coh}X)
=
\left\langle
\underbrace{
D^b(K(3,1)\text{-mod}),\ldots,D^b(K(3,1)\text{-mod})
}_{k\text{ components}},
F_1,\ldots,F_{s+2}
\right\rangle,
\]
where the $F_i$ are exceptional objects.
\item If $k=6$, let $\alpha$ be a generator of
$\operatorname{Br}(X)\cong\mathbb Z/3$. Then
\[
D^b(\operatorname{coh}(X,\alpha))
=
\left\langle
\underbrace{
D^b(K(3,1)\text{-mod}),\ldots,D^b(K(3,1)\text{-mod})
}_{6\text{ components}},
F_1,\ldots,F_s
\right\rangle,
\]
where the $F_i$ are exceptional objects in the twisted derived category.
\end{enumerate}
\end{theorem}

The one-point case is \cref{cor:one-point-kks-sod}, the three-point case
is \cref{rem:three-point-toric-descent}, and the six-point twisted case is
\cref{cor:six-point-twisted-descent}. For $k\in\{2,4,5\}$, the Brauer
group vanishes, but this necessary condition does not itself produce
simultaneous adherent subcategories. The corresponding construction
remains open, as discussed at the end of Section~\ref{sec:cascade}.

The family $X_{1,1/3}$ occurs at the end of the one-point index-one
cascade and admits a general-member realization as quasismooth
hypersurfaces of degree $10$ in $\mathbb P(1,2,3,5)$. These surfaces
occur in the Johnson--Koll\'ar
classification of anticanonically embedded quasismooth log del Pezzo
hypersurfaces in weighted projective $3$-space
\cite{JohnsonKollar2001} and in the Reid--Suzuki cascade of projections
\cite{ReidSuzuki2004}.

Earlier explicit categorical constructions address different weighted
families. Elagin constructed full exceptional collections for stacky
hypersurfaces of degree $4n-2$ in
$\mathbb P(1,2,2n-1,4n-3)$ \cite{Elagin2007}. Gugiatti and Rota
constructed explicit full exceptional collections for the
Johnson--Koll\'ar series \cite{GugiattiRota2023}, whose members have the
form 
\[
X_{8k+4}\subset\mathbb P(2,2k+1,2k+1,4k+1).
\] 

We single out the degree-$10$ family because a general member of
$X_{1,1/3}$ has a quasismooth hypersurface model in
$\mathbb P(1,2,3,5)$, as established by Corti--Heuberger and
Reid--Suzuki~\cite{Corti2016,ReidSuzuki2004}. The corresponding blow-up
presentation provides the geometric input for the explicit categorical
constructions below. For every member $X\in X_{1,d}$ of the one-point
cascade, the surface and its canonical stack are obtained by blowing up
$\Pp(1,1,3)$ and its canonical stack, respectively:
\[
X\simeq\Bl_{p_1,\ldots,p_s}\Pp(1,1,3),
\qquad
\mathcal X\simeq
\Bl_{p_1,\ldots,p_s}\mathcal P(1,1,3),
\qquad
s=\frac{25}{3}-d,
\]
with distinct smooth centers. The consecutive collection of
Auroux--Katzarkov--Orlov together with the exceptional divisors gives an
explicit full strong exceptional collection of line bundles on $\mathcal X$
(\cref{thm:one-point-stack-strong}). For $X_{10}$, one has $d=1/3$
and hence $s=8$. The next theorem records this collection explicitly and
relates it to the hypersurface presentation.

\begin{theorem}[The hypersurface \(X_{10}\)]\label{thm:X10-main}
Let $X_{10}\subset\mathbb P(1,2,3,5)$ be a quasismooth hypersurface of
degree $10$, with canonical stack $\mathcal X_{10}$ and minimal
resolution $\widetilde X_{10}$. Then $X_{10}$ has a unique singular point
$p$, of type $\tfrac13(1,1)$, and $\widetilde X_{10}$ is the blow-up
$\sigma\colon\widetilde X_{10}\to\mathbb F_3$ at eight distinct points off
$C_0$, with pairwise disjoint exceptional curves $D_1,\dots,D_8$.
\begin{enumerate}[label=(\roman*)]
\item Let \(\mathcal P(1,1,3)\) be the canonical stack of
\(\mathbb P(1,1,3)\), and let
\(\beta:\mathcal X_{10}\to\mathcal P(1,1,3)\) be the blow-down.
Let \(\mathcal D_1,\ldots,\mathcal D_8\) be its exceptional divisors, and
put \(H=\beta^*\OO_{\mathcal P(1,1,3)}(1)\). Then
\(D^b(\operatorname{coh}\mathcal X_{10})\) admits the full strong
exceptional collection of line bundles
\[
\left\langle
\OO_{\mathcal X_{10}},
\OO_{\mathcal X_{10}}(\mathcal D_1),\ldots,
\OO_{\mathcal X_{10}}(\mathcal D_8),
\OO_{\mathcal X_{10}}(H),
\OO_{\mathcal X_{10}}(2H),
\OO_{\mathcal X_{10}}(3H),
\OO_{\mathcal X_{10}}(4H)
\right\rangle
\]
of length $13$.

\item Let $\eta\colon X_{10}\to\mathbb P(1,1,3)$ be the blow-down in the
presentation as the blow-up of $\mathbb P(1,1,3)$ at eight points, and write
$\nu\colon\widetilde X_{10}\to X_{10}$ for the resolution morphism. Then
\[
D^b(\operatorname{coh}X_{10})
=
\left\langle
D^b(K(3,1)\text{-mod}),
\mathcal O_{X_{10}},
\eta^*\mathcal O_{\mathbb P(1,1,3)}(3),
\mathcal O_{\nu(D_1)},\ldots,
\mathcal O_{\nu(D_8)}
\right\rangle.
\]
\end{enumerate}
\end{theorem}

\Cref{prop:X10-quasismooth} gives the unique $\tfrac13(1,1)$ singularity,
while \cite[Corollary~8(1)]{Corti2016} yields the presentation of
$\widetilde X_{10}$ as the blow-up of $\mathbb F_3$ at eight distinct
points off $C_0$. Part~(i) is the specialization of
\cref{thm:one-point-stack-strong} to \(d=1/3\) and \(s=8\), and its
Euler matrix is the specialization \(S_8\) of
\cref{prop:one-point-strong-euler}.
\Cref{thm:descent-cascades-intro} gives the decomposition in part~(ii),
while \cref{thm:X10-kks-sod} identifies the exceptional objects explicitly.
The Ishii--Ueda exceptional collection on \(\mathcal X_{10}\) is already
the \(k=1\) specialization of \cref{thm:explicit-cascades-main}.

For rational surfaces with arbitrary baskets of
$\tfrac1{n_i}(1,1)$ singularities, the Euler matrix has block
upper-triangular form. Its diagonal blocks are the local singularity
blocks and the matrix on the resolution. For a collection of line bundles on
the resolution, the off-diagonal blocks are determined by intersections
with the exceptional curves (\cref{prop:mixed-basket-euler}). These are
the Euler matrices of the Ishii--Ueda collections. For \(X_{10}\), the
Ishii--Ueda collection and the direct full strong collection share ten of
their thirteen objects, while their remaining three-object complements
give a concrete comparison of their exceptional bases.

Sections~\ref{sec:cascade-foundations}--\ref{sec:further-directions}
develop the cascades and their derived categories, with particular attention
to $X_{10}$ and Euler matrices. The appendices treat the Brauer obstructions
and collections on the rigid bases, and complete the proof of
\cref{thm:explicit-cascades-main}.

\subsection*{Notation and conventions}

All varieties and stacks are defined over \(\mathbb C\).
If \(\mathcal T\) is a triangulated category and \(E,F\in\mathcal T\),
we use
\[
\Ext^k(E,F):=\Hom_{\mathcal T}(E,F[k]).
\]
Throughout, a semiorthogonal decomposition
\(\langle\mathcal A_1,\ldots,\mathcal A_N\rangle\) is ordered so that
\(\operatorname{Hom}(\mathcal A_j,\mathcal A_i[t])=0\) for every
\(t\in\mathbb Z\) whenever \(j>i\). Derived pullback, pushforward, and
internal Hom are denoted by \(Lf^*\), \(Rf_*\), and
\(R\mathcal Hom\), respectively.
We write $\mathbb F_n$ for the $n$-th Hirzebruch surface, $C_0$ for its
negative section, and $f$ for its fiber class. The notation $\Pp(E)$
denotes the Grothendieck projectivization parametrizing one-dimensional
quotients of $E$. If $p:\Pp(E)\to B$ is the projection, then
$p^*E\twoheadrightarrow\OO_{\Pp(E)}(1)$ is the universal quotient, and
$\OO_{\Pp(E)}(1)$ is the tautological quotient line bundle. We set
$\OO_{\Pp(E)}(-1)=\OO_{\Pp(E)}(1)^\vee$. We write
$\nu:\widetilde X\to X$ for the minimal resolution. In the general
discussion, $p_1,\ldots,p_k$ denote the singular points of $X$, with
$p_i$ of type $\tfrac1{n_i}(1,1)$. When there is a single singular
point, we write $n$ for its index.

\section{Corti--Heuberger cascades}
\label{sec:cascade-foundations}

The Corti--Heuberger classification organizes these surfaces into six
\(\mathbb Q\)-Gorenstein cascades, indexed by the number \(k\) of
\(\frac13(1,1)\)-points. Each cascade starts from a fixed surface \(Z_k\),
and its other members are obtained by blow-ups at smooth points. Rigidity
explains why the starting surfaces are fixed, while marked blow-up
presentations record the centers and incidences needed for the later
geometric and categorical constructions.

\begin{definition}[\(\mathbb Q\)-Gorenstein deformations and rigidity]
\label{def:qg-deformations}
Let \(X\) be a normal projective surface with quotient singularities.
A \emph{\(\mathbb Q\)-Gorenstein deformation} of \(X\) is a flat deformation
that locally lifts to equivariant
deformations of the index-one, or canonical, covers of the quotient
singularities, namely the local cyclic covers on which the canonical
divisor becomes Cartier.
This deformation framework originates in the work of
Koll\'ar--Shepherd-Barron \cite{KollarShepherdBarron1988}.
For projective surfaces we use proper families, and \(\mathbb Q\)-Gorenstein deformation
equivalence means being connected by a chain of such families over
connected bases \cite[Basic Concepts, Definition~2 and proof of Lemma~6]{AkhtarEtAl2016}.
A surface is \emph{\(\mathbb Q\)-Gorenstein rigid} if every sufficiently small
\(\mathbb Q\)-Gorenstein deformation is trivial.
\end{definition}

The degree is also locally constant in $\mathbb Q$-Gorenstein families
\cite[Section~1.1.1]{Corti2016}. Remark~7 of Corti--Heuberger states that,
with the higher-index families $B_{1,16/3}$ and $B_{2,8/3}$ excepted, the
surfaces in their Theorem~6 are $\mathbb Q$-Gorenstein rigid. Thus the
starting surfaces $Z_k$ are fixed up to isomorphism. For $2\leq k\leq6$,
each is the unique isomorphism class in its index-one deformation family
\cite[Theorem~6 and Remark~7]{Corti2016}. This global rigidity is distinct
from local rigidity: a singularity of type $\frac13(1,1)$ has no nontrivial
local $\mathbb Q$-Gorenstein deformations, but local rigidity alone does
not imply rigidity of a surface containing it. Variation along a cascade
instead comes from configurations of smooth blow-up centers, with floating
$(-1)$-curves describing the birational steps.

\begin{definition}[Floating \((-1)\)-curves and cascades]
\label{def:floating-curves-cascades}
Let \(X\) be a log del Pezzo surface with only \(\frac13(1,1)\)-points.
A \emph{floating \((-1)\)-curve} on \(X\) is a smooth rational curve
contained in its smooth locus with self-intersection \(-1\)
\cite[Definition~5]{Corti2016}. A \emph{cascade} connects the deformation
families by blow-ups at smooth points, or equivalently by reversing
contractions of floating \((-1)\)-curves.
\end{definition}

A blow-up at a smooth point lowers \(K^2\) by one and leaves the number
and analytic types of the quotient singularities unchanged. Thus the drop
in degree counts the blow-ups in a cascade.

By \cite[Corollary~8(1)--(6)]{Corti2016}, every index-one member
\(X\in X_{k,d}\) is obtained from \(Z_k\) by
\(s=K_{Z_k}^2-K_X^2\) blow-ups at distinct smooth points. Only
configurations for which the resulting surface remains log del Pezzo are
admissible, so its anticanonical divisor must remain ample. The case
\(s=0\) is included when \(X\simeq Z_k\) for \(2\leq k\leq6\). For
\(k=1\), every member has \(s\geq1\), since \(Z_1\) is not an index-one
member.

The classification also provides explicit projective models for the 26
index-one families, including weighted complete intersections and
degeneracy loci \cite[Theorem~4 and Table~2]{Corti2016}. Corti and
Heuberger explain that these constructions realize general members of the
corresponding families, but are not asserted to describe every individual
member \cite[Section~1.1.2, footnote~2]{Corti2016}. By contrast, the
blow-up presentations just described apply to every member, starting from
its fixed surface $Z_k$. This universal cascade description underlies the
exceptional collections and descent constructions below.

The exceptional collections constructed in Appendix~\ref{app:explicit-bases}
and the intersection calculations below depend on the order of the blow-ups
and each center's incidence with earlier exceptional curves. A marked presentation records this information.

\begin{definition}[Marked blow-up presentations]
\label{def:marked-blowup-presentation}
A \emph{marked blow-up presentation} of a smooth surface \(Y\) is a
sequence of point blow-ups
\[
Y_0\xleftarrow{\beta_1}Y_1\xleftarrow{\beta_2}\cdots
\xleftarrow{\beta_m}Y_m\simeq Y,
\]
with fixed identifications of the starting and final surfaces and, at
each stage, a specified point \(q_i\in Y_{i-1}\), its incidences with
the strict transforms of previously introduced exceptional curves, and
a label for the new exceptional divisor \(D_i\subset Y_i\).
For a marked blow-up presentation, the centers may be infinitely
near points of \(Y_0\). When \(Y\) is a
resolution of a surface with quotient singularities, the marking also
specifies the curves contracted to those singularities. A \emph{marked
cascade presentation} specifies distinct smooth points on the fixed
surface \(Z_k\), their order, and the identification of their blow-up
with the member \(X\).
\end{definition}

These are distinct constructions: a marked presentation of the minimal
resolution \(\widetilde Z_k\) from \(Y_0\) may involve infinitely near centers,
whereas cascade blow-ups are centered at smooth points of \(Z_k\).
The cascade centers are distinct by ampleness. Indeed, in any sequence of
blow-ups at smooth points ending in a log del Pezzo surface \(X\), an
exceptional curve has anticanonical degree one when created. If \(h\geq1\)
later centers lie on its successive strict transforms, its final strict
transform \(D\), which remains in the smooth locus, satisfies
\begin{equation}\label{eq:cascade-distinct}
(-K_X)\cdot D=1-h\leq0,
\end{equation}
contrary to ampleness. No later cascade center can therefore lie on an earlier exceptional
curve, so the exceptional curves are pairwise disjoint.
Figure~\ref{fig:cascade-centers-ampleness} illustrates this obstruction to
repeated centers.

\begin{figure}[H]
\centering
\includegraphics[width=0.9\textwidth]{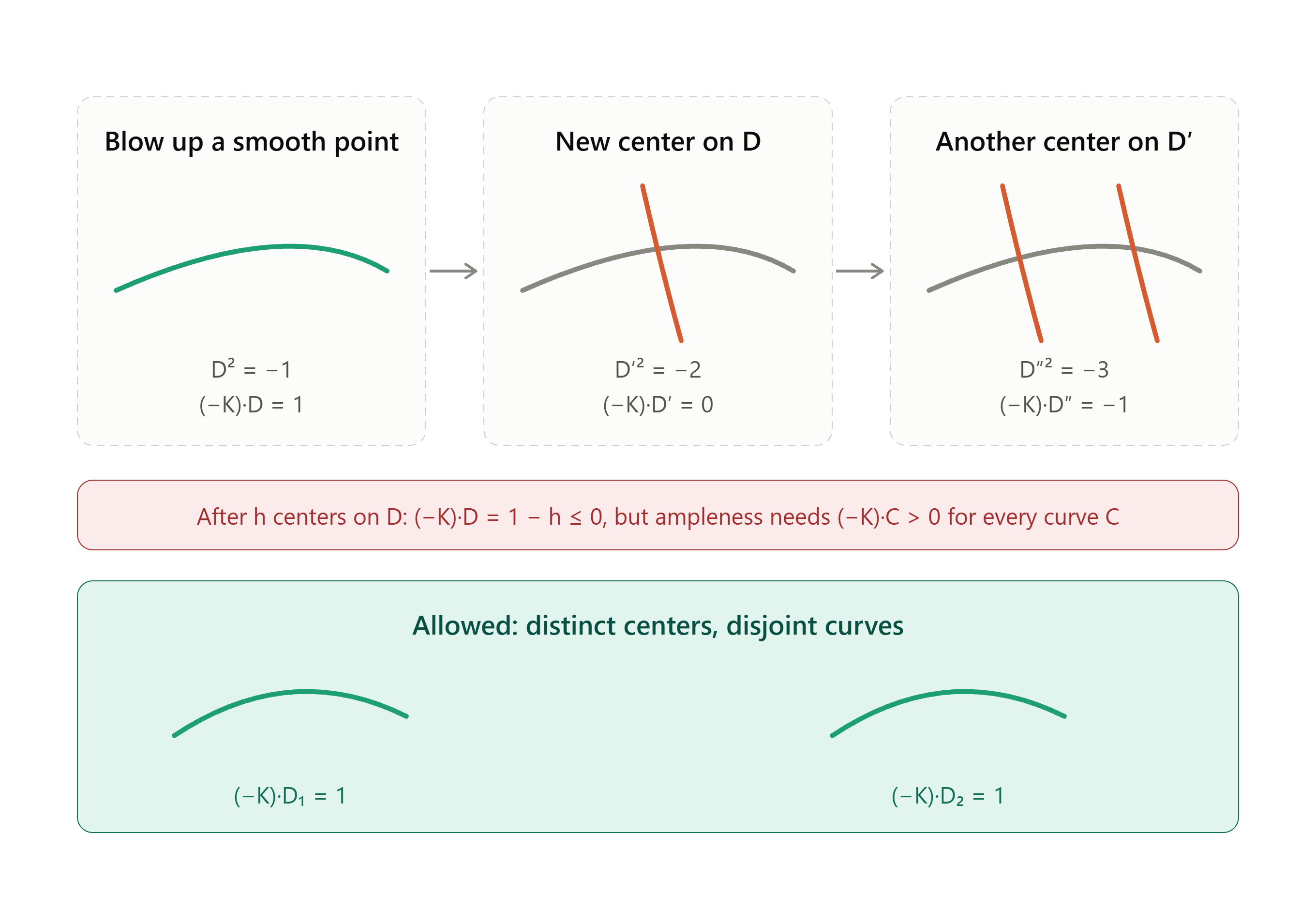}
\caption{Ampleness excludes repeated cascade centers, so the exceptional curves are pairwise disjoint.}
\label{fig:cascade-centers-ampleness}
\end{figure}

For a cascade member $X$, let $\beta:X\to Z_k$ be the contraction to its
rigid starting surface, and let $\gamma:\widetilde X\to\widetilde Z_k$ be
the induced morphism of minimal resolutions. Figure~\ref{fig:cascade-resolution-square}
places them in a commutative square. The vertical maps resolve the quotient
singularities, and the horizontal maps contract the exceptional curves from
the cascade blow-ups.

\begin{figure}[H]
\centering
\begin{tikzcd}[column sep=large,row sep=large]
\widetilde X \arrow[r,"\gamma"] \arrow[d,"\nu_X"'] &
\widetilde Z_k \arrow[d,"\nu_k"] \\
X \arrow[r,"\beta"'] & Z_k
\end{tikzcd}
\caption{The cascade contraction and its lift to the minimal resolutions.}
\label{fig:cascade-resolution-square}
\end{figure}

Table~\ref{tab:rigid-base-invariants} collects the six rigid starting
surfaces, their numerical data, and the marked blow-up presentations of
their minimal resolutions \cite[Theorem~6(1)--(6) and Remark~7]{Corti2016}.
The first is $Z_1=S_{1,25/3}=\Pp(1,1,3)$, of Fano index five, and is not
itself an index-one member. For $2\leq k\leq6$, each $Z_k$ represents the
unique isomorphism class in its corresponding index-one family. The table
also records each family label. The $Y_0$ and blow-up-count columns
describe the presentations of $\widetilde Z_k$: they give the underlying
smooth surface and the number of point blow-ups used to obtain the minimal
resolution, not the number of blow-ups in a cascade. The final columns give
the lengths $N_k$ and $\ell_k$ of full exceptional collections on the
resolution and canonical stack. Appendix~\ref{app:explicit-bases} gives the
marked presentations underlying the table.
\begin{table}[H]
\centering
\small
\renewcommand{\arraystretch}{1.2}
\begin{tabular}{@{}c c c c c c c@{}}
\(k\)&\(Z_k\)&\(K_{Z_k}^2\)&\(Y_0\)&\shortstack{Number of\\blow-ups}&\(N_k\)&\(\ell_k\)\\ \hline
1&\(Z_1=S_{1,25/3}=\Pp(1,1,3)\)&\(25/3\)&\(\F_3\)&0&4&5\\
2&\(Z_2\in X_{2,17/3}\)&\(17/3\)&\(\F_3\)&3&7&9\\
3&\(Z_3\in X_{3,5}\)&5&\(\F_3\)&4&8&11\\
4&\(Z_4\in X_{4,7/3}\)&\(7/3\)&\(\F_2\)&7&11&15\\
5&\(Z_5\in X_{5,5/3}\)&\(5/3\)&\(\Pp^2\)&9&12&17\\
6&\(Z_6\in X_{6,2}\)&2&\(\F_0\)&8&12&18
\end{tabular}
\caption{Rigid cascade bases, marked blow-up presentations, and collection lengths.}
\label{tab:rigid-base-invariants}
\end{table}

The last two columns can also be read off from the geometry of the
resolutions. Each singularity of type $\frac13(1,1)$ resolves to a single
$(-3)$-curve with discrepancy $-\frac13$, so
\[
K_{\widetilde Z_k}^2=K_{Z_k}^2-\frac{k}{3}.
\]
Since $\widetilde Z_k$ is a smooth rational surface, its Grothendieck group
has rank
\[
N_k=\rk K_0(\widetilde Z_k)
    =12-K_{\widetilde Z_k}^2
    =12-K_{Z_k}^2+\frac{k}{3}.
\]
Adding one exceptional object for each singular point gives
\[
\ell_k=N_k+k=12-K_{Z_k}^2+\frac{4k}{3}.
\]
Thus each smooth-point blow-up increases the collection length by one,
while the local contributions remain unchanged. These numerical changes
are realized categorically in Sections~\ref{sec:rationality} and
\ref{sec:canonical-stack}.

Miura's classification~\cite{Miura2019} studies del Pezzo surfaces with
\(\frac13(1,1)\)- and \(\frac14(1,1)\)-singularities and no floating
\((-1)\)-curves. His approach uses directed sequences of extremal
contractions to reduce the classification to suitable minimal surfaces,
and some contractions may involve the singular locus. The Corti--Heuberger
cascades, by contrast, retain the singularity basket at every step.

Cavey and Prince~\cite{CAVEY2020} place the one-point cascade in a broader
framework for del Pezzo surfaces with a single \(\frac1r(1,1)\)-singularity,
for arbitrary positive integers \(r\), relating projective models, toric
degenerations, and root systems in the Picard lattice. For \(r=3\), the
degree-\(\frac13\) family \(X_{1,1/3}\) has an \(E_8\) root system in the
orthogonal complement of the canonical class. Figure~\ref{fig:corti-heuberger-degree-chart}
places this family at the end of the first cascade in a degree-wise overview
of all six. The family is studied further in Section~\ref{sec:cascade}.

In this diagram, each point denotes a \(\mathbb Q\)-Gorenstein deformation
family, not an individual surface. Different configurations of smooth
blow-up centers can give distinct members of the same family. Green points
mark families containing toric representatives, whose explicit geometry
and intersection data are used in the Brauer-group computations of
Appendix~\ref{app:brauer}. The categorical constructions in
Sections~\ref{sec:rationality} and~\ref{sec:canonical-stack} apply to the
cascade families themselves, not only to their toric representatives. The
Brauer-group labels beneath the diagram preview
Theorem~\ref{thm:brauer-cascades-intro}.

\begin{figure}[H]
\centering
\includegraphics[width=\textwidth]{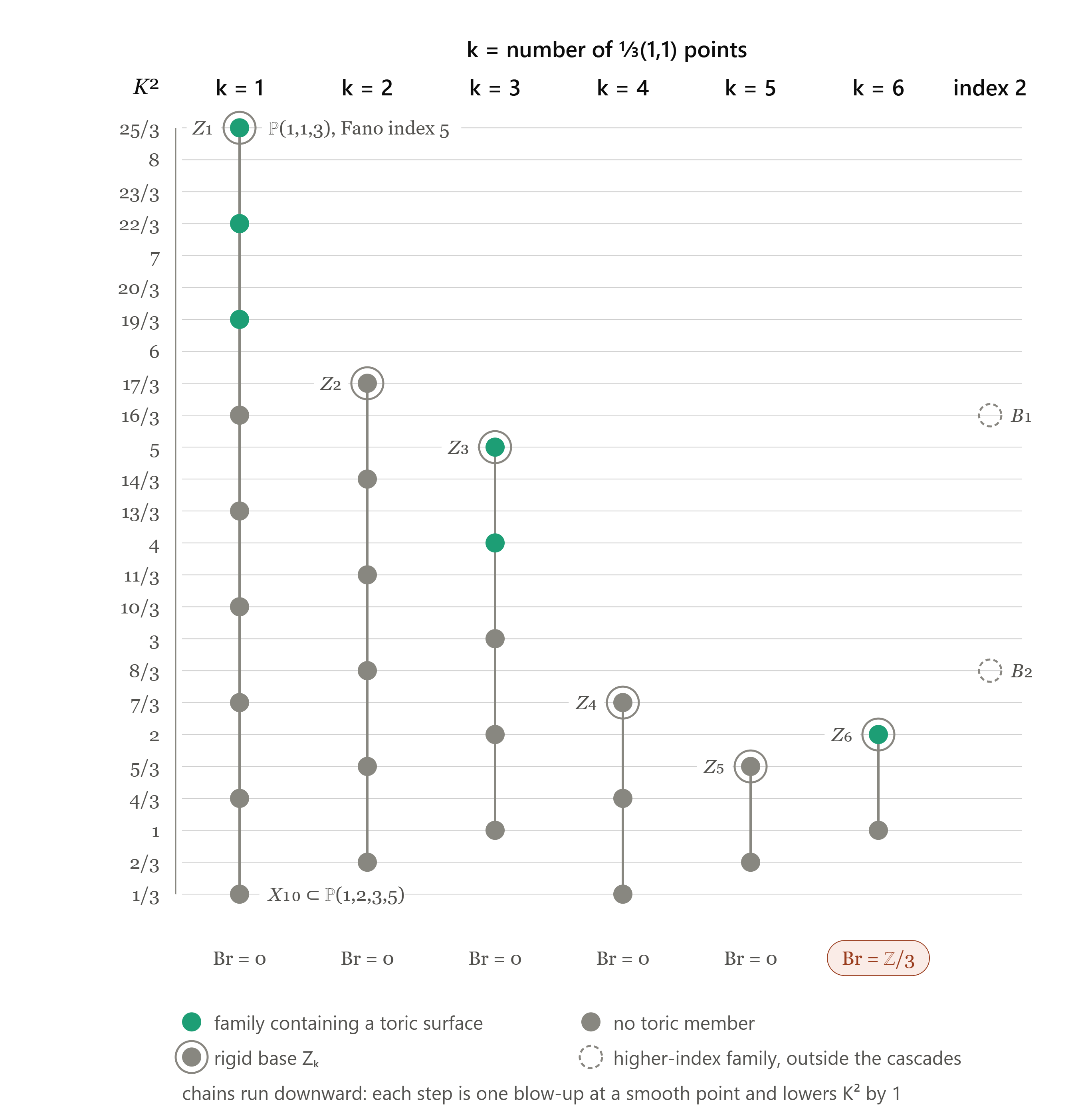}
\caption{Each downward step is a smooth-point blow-up and lowers $K^2$ by one. Green points indicate families containing a toric surface, outlined points mark the rigid bases, and the higher-index families $B_1$ and $B_2$ are shown for comparison.}
\label{fig:corti-heuberger-degree-chart}
\end{figure}

Sections~\ref{sec:rationality} and~\ref{sec:canonical-stack} construct the
exceptional collections on the minimal resolutions and canonical stacks.
Appendix~\ref{app:explicit-bases} gives the explicit marked presentations
of the rigid bases used in these constructions.

\section{Resolutions and exceptional collections}
\label{sec:rationality}
\label{sec:orlov}

This section proves rationality of the smooth minimal resolution
\(\widetilde X\) and constructs full exceptional collections on it.
Beilinson's collections on projective spaces and the standard
projective-bundle collections on Hirzebruch surfaces provide the starting
points. For each point blow-up, the semiorthogonal decomposition is given
by Orlov's formula in \cref{thm:orlov-blow-up}. When the original surface
has a full exceptional collection of line bundles beginning with its
structure sheaf, \cref{lem:point-blow-up-collection} produces one of the
same kind on the blow-up. Iterating this construction after
rationality is established yields \cref{cor:resolution-fec}.

The relevant background on log del Pezzo surfaces can be found in
\cite{AlexeevNikulin}, and the rationality statement is also recalled in
\cite[Remark~4.3]{GugiattiRota2023}. Our treatment of exceptional
collections on smooth rational surfaces follows \cite{HillePerling2011}.

\begin{definition}
A \emph{log del Pezzo surface} is a normal projective surface \(X\) with
Kawamata log terminal (klt) singularities such that \(-K_X\) is ample.
\end{definition}

For normal surfaces, klt singularities are precisely quotient singularities \cite[Proposition~4.18]{Kollar1998}. They are also rational \cite[Theorem~5.22]{Kollar1998}. If \(\nu:\widetilde X\to X\) is the minimal resolution, then \(\nu_*\OO_{\widetilde X}=\OO_X\) and \(R^i \nu_*\OO_{\widetilde X}=0\) for all \(i>0\). The Leray spectral sequence gives a canonical isomorphism
\[
H^1(\widetilde X,\OO_{\widetilde X})
\cong
H^1(X,\OO_X).
\]

We recall the rationality argument.

\begin{proposition}\label{prop:log-del-pezzo-rational}
Let \(X\) be a log del Pezzo surface, and let \(\nu:\widetilde X\to X\) be its minimal resolution. The surface \(\widetilde X\) is rational. In particular, \(X\) is rational.
\end{proposition}

\begin{proof}
Since \(X\) is klt and \(-K_X\) is ample, Kawamata--Viehweg vanishing for
klt varieties \cite[Theorem~2.70]{Kollar1998} gives
\(H^1(X,\OO_X)=0\). The rationality of the singularities and the Leray
isomorphism above then give
\(H^1(\widetilde X,\OO_{\widetilde X})=0\).

It remains to show that \(H^0(\widetilde X,\OO_{\widetilde X}(2K_{\widetilde X}))=0\). Suppose, to the contrary, that there exists an effective divisor \(D\sim 2K_{\widetilde X}\). The pullback \(-\nu^*K_X\) is nef because \(-K_X\) is ample, so \((-\nu^*K_X)\cdot D\geq 0\). On the other hand, write the discrepancy formula as
\[
K_{\widetilde X}=\nu^*K_X+\sum_i a_iE_i.
\]
The projection formula gives \(\nu^*K_X\cdot E_i=0\) for every
\(\nu\)-exceptional curve \(E_i\) and \((\nu^*K_X)^2=K_X^2\).
Intersecting the discrepancy formula with \(-\nu^*K_X\) therefore gives
\[
(-\nu^*K_X)\cdot D
=
(-\nu^*K_X)\cdot 2K_{\widetilde X}
=
-2K_X^2<0.
\]
This contradiction implies \(H^0(\widetilde X,\OO_{\widetilde X}(2K_{\widetilde X}))=0\). By Castelnuovo's rationality criterion, \(\widetilde X\) is rational. Since \(X\) is birational to \(\widetilde X\), the surface \(X\) is rational as well.
\end{proof}

\begin{remark}[Historical note]
The rationality of log del Pezzo surfaces is classical. The theory of
open surfaces developed by Miyanishi and Tsunoda
\cite{MiyanishiTsunoda1984a,MiyanishiTsunoda1984b}, and Zhang's
classification \cite{Zhang1988}, concern logarithmic del Pezzo pairs on
smooth rational surfaces. Alexeev and Nikulin developed the theory of del
Pezzo surfaces with log-terminal singularities
\cite{AlexeevNikulin1988,AlexeevNikulin1989} and gave a Castelnuovo proof
of rationality \cite[Lemma~4.1]{AlexeevNikulin}, using plurigenus
vanishing from the effectivity of \(-K_X\). Gurjar and Zhang gave another
proof by reducing the resolution to \(\Pp^2\) or a ruled surface
\cite[Lemma~1.3]{GurjarZhang1994}. Keel and McKernan proved the stronger
statement that the smooth locus is rationally connected
\cite[Corollary~1.6]{KeelMcKernan1999}. Rationality in the surface case
also follows from the general rational connectedness results for log
\(\mathbb Q\)-Fano varieties \cite{Zhang2006,HaconMcKernan2007}.
\end{remark}

\begin{example}
The weighted projective plane \(\Pp(1,1,3)\) has a unique singularity
of type \(\frac13(1,1)\). Its minimal resolution is
\(\F_3=\Pp_{\Pp^1}(\OO\oplus\OO(3))\), whose negative section
\(C_0\), with \(C_0^2=-3\), is the exceptional locus
\cite[Example~3.17]{KKS2020}.
\end{example}

\begin{definition}
Let \(\mathcal T\) be a triangulated category. We call an object
\(G\in\mathcal T\) \emph{exceptional} when its endomorphism algebra is
\(\C\) and all its self-extension groups in nonzero degrees vanish. Thus
\(G\) satisfies
\[
\Ext^0(G,G)=\C
\qquad\text{and}\qquad
\Ext^k(G,G)=0
\]
for every \(k\neq0\).

An ordered sequence \(\langle G_1,\ldots,G_n\rangle\) forms an
\emph{exceptional collection} when every object \(G_i\) is exceptional
and every later object is right orthogonal to every earlier object. More
precisely, the sequence satisfies the vanishing
\[
\Ext^k(G_j,G_i)=0
\]
for every \(k\in\Z\) whenever \(j>i\).

We call the collection \emph{full} when its objects generate
\(\mathcal T\) as a thick triangulated subcategory. Thus fullness means
that the equality
\[
\operatorname{thick}(G_1\oplus\cdots\oplus G_n)=\mathcal T
\]
holds. We call the collection \emph{strong} when every Ext group between
its objects vanishes in every nonzero degree. Thus strength means that
the equalities
\[
\Ext^k(G_i,G_j)=0
\]
hold for every \(i,j\) and every \(k\neq0\).

\end{definition}

Beilinson's standard full exceptional collections \cite{Beilinson1978}
include
\[
\begin{aligned}
\Db(\coh\Pp^1)
&=\left\langle\OO_{\Pp^1},\OO_{\Pp^1}(1)\right\rangle,\\
\Db(\coh\Pp^2)
&=\left\langle
\OO_{\Pp^2},\OO_{\Pp^2}(1),\OO_{\Pp^2}(2)
\right\rangle.
\end{aligned}
\]
Since tensoring by a line bundle is an autoequivalence and preserves Ext
groups,
\[
\Ext^k(E\otimes L,F\otimes L)\simeq\Ext^k(E,F).
\]
It follows that
\(\langle\OO_{\Pp^1}(a),\OO_{\Pp^1}(a+1)\rangle\) is full exceptional
for every \(a\in\Z\).

Let
\[
p:\F_n=\Pp_{\Pp^1}
\bigl(\OO_{\Pp^1}\oplus\OO_{\Pp^1}(n)\bigr)
\longrightarrow\Pp^1
\]
be the ruling, with negative section \(C_0\) and fiber class \(f\). Thus
\(C_0^2=-n\), \(C_0\cdot f=1\), and \(f^2=0\), while
\[
p^*\OO_{\Pp^1}(a)\simeq\OO_{\F_n}(af).
\]
With our quotient convention, the universal quotient and its divisor
class are
\[
p^*\bigl(\OO_{\Pp^1}\oplus\OO_{\Pp^1}(n)\bigr)
\twoheadrightarrow\OO_{\F_n}(1),
\qquad
\OO_{\F_n}(1)\simeq\OO_{\F_n}(C_0+nf).
\]
This relative tautological quotient line bundle is distinct from the
pullback \(p^*\OO_{\Pp^1}(1)\simeq\OO_{\F_n}(f)\). Orlov's
projective-bundle formula \cite[Corollary 2.7]{Orlov1993} gives
\[
\Db(\coh\F_n)
=
\left\langle
p^*\Db(\coh\Pp^1),
p^*\Db(\coh\Pp^1)\otimes\OO_{\F_n}(1)
\right\rangle.
\]
Taking the standard Beilinson pair in the first component and its twist
by \(\OO_{\Pp^1}(a-n)\) in the second yields
\[
\begin{aligned}
\OO_{\F_n}(1)\otimes p^*\OO_{\Pp^1}(a-n)
&\simeq\OO_{\F_n}(C_0+af),\\
\OO_{\F_n}(1)\otimes p^*\OO_{\Pp^1}(a-n+1)
&\simeq\OO_{\F_n}(C_0+(a+1)f).
\end{aligned}
\]
Consequently,
\begin{equation}\label{eq:hirzebruch-family}
\Db(\coh\F_n)
=
\left\langle
\OO_{\F_n},
\OO_{\F_n}(f),
\OO_{\F_n}(C_0+af),
\OO_{\F_n}(C_0+(a+1)f)
\right\rangle,
\qquad a\in\Z.
\end{equation}
This is a full exceptional collection, and \(a=n\) recovers the standard
projective-bundle collection. We suppress subscripts below when the
ambient space is clear.

\begin{theorem}[Orlov's blow-up formula {\cite[Theorem~4.3 and Corollary~4.4]{Orlov1993}}]\label{thm:orlov-blow-up}
Let \(S\) be a smooth projective surface, and let \(\beta:\Bl_pS\to S\) be the blow-up of \(S\) at a point \(p\), with exceptional curve \(D\). There is a semiorthogonal decomposition
\[
\Db(\coh \Bl_pS)
=
\left\langle
\OO_D(-1),L\beta^*\Db(\coh S)
\right\rangle,
\]
where \(\OO_D(-1)\) is the degree-\(-1\) line bundle on
\(D\simeq\Pp^1\), regarded as an object of \(\Db(\coh \Bl_pS)\) via
pushforward along the inclusion \(D\hookrightarrow \Bl_pS\). In
particular, if \(\Db(\coh S)\) admits a full exceptional collection, the
category \(\Db(\coh \Bl_pS)\) also admits one.
\end{theorem}

For the collections of line bundles used below, this blow-up step is
the first-position case of the Hille--Perling standard augmentation
\cite[Definition~5.4 and Proposition~5.5]{HillePerling2011}. We record the
corresponding statement for full exceptional collections.

\begin{lemma}[Point blow-ups and standard augmentation]
\label{lem:point-blow-up-collection}
Let
\[
\beta:\widetilde S=\Bl_pS\to S
\]
be the blow-up of a smooth projective surface at a point, with exceptional
divisor \(D\). Suppose
\[
\langle\OO_S,F_1,\ldots,F_r\rangle
\]
is a full exceptional collection on \(S\). Then
\[
\left\langle
\OO_{\widetilde S},
\OO_{\widetilde S}(D),
L\beta^*F_1,\ldots,L\beta^*F_r
\right\rangle
\]
is a full exceptional collection in \(\Db(\coh\widetilde S)\). If all
\(F_i\) are line bundles, then the resulting collection consists entirely
of line bundles.
\end{lemma}

\begin{proof}
Since \(\OO_S\) is exceptional and
\(R\beta_*\OO_{\widetilde S}\simeq\OO_S\), the structure sheaves of
\(S\) and \(\widetilde S\) have cohomology \(\C\) in degree zero and no
higher cohomology. In the sequence
\[
0\longrightarrow\OO_{\widetilde S}(-D)
\longrightarrow\OO_{\widetilde S}
\longrightarrow\OO_D
\longrightarrow0,
\]
restriction is an isomorphism on global sections and \(\OO_D\) has no
higher cohomology. Hence
\(H^\bullet(\widetilde S,\OO_{\widetilde S}(-D))=0\), so
\(\RHom(\OO_{\widetilde S}(D),\OO_{\widetilde S})=0\).
The divisor sequence
\[
0\longrightarrow\OO_{\widetilde S}
\longrightarrow\OO_{\widetilde S}(D)
\longrightarrow\OO_D(-1)
\longrightarrow0
\]
gives \(\RHom(\OO_{\widetilde S},\OO_{\widetilde S}(D))\simeq\C\) in
degree zero. Also,
\[
\RHom(\OO_{\widetilde S}(D),\OO_{\widetilde S}(D))
\simeq R\Gamma(\widetilde S,\OO_{\widetilde S})\simeq\C,
\]
so \(\OO_{\widetilde S}(D)\) is exceptional. Since \(\OO_D(-1)\) is acyclic, pushing forward
the divisor sequence gives
\[
R\beta_*\OO_{\widetilde S}(D)\simeq\OO_S.
\]
Therefore, for each \(j\), adjunction and exceptionality of the original
collection give
\[
\begin{aligned}
\RHom(L\beta^*F_j,\OO_{\widetilde S})
&\simeq\RHom(F_j,\OO_S)\\
&=0,\\
\RHom(L\beta^*F_j,\OO_{\widetilde S}(D))
&\simeq\RHom(F_j,\OO_S)\\
&=0.
\end{aligned}
\]
Orlov's blow-up formula gives
\[
\Db(\coh\widetilde S)
=\left\langle\OO_D(-1),L\beta^*\Db(\coh S)\right\rangle.
\]
The divisor sequence shows that \(\OO_D(-1)\) and
\(\OO_{\widetilde S}\) generate the same subcategory as
\(\OO_{\widetilde S}\) and \(\OO_{\widetilde S}(D)\). Together with
fully faithful pullback and the preceding vanishings, this proves fullness
and exceptionality. Pullbacks of line bundles are line bundles, which
proves the final assertion.
\end{proof}

\begin{corollary}[{\cite[Theorem~5.6]{HillePerling2011}}]\label{cor:rational-surface-fec}
Every smooth projective rational surface admits a full exceptional
collection of line bundles.
\end{corollary}

\begin{proof}
Express the surface as an iterated point blow-up of \(\Pp^2\) or a
Hirzebruch surface \(\F_n\) \cite[Theorem~V.10]{Beauville1996}. For
\(\Pp^2\), start with Beilinson's collection
\[
\left\langle
\OO_{\Pp^2},\OO_{\Pp^2}(1),\OO_{\Pp^2}(2)
\right\rangle.
\]
For \(\F_n\), take \(a=n\) in \eqref{eq:hirzebruch-family}. In either
case the starting collection is a full exceptional collection of line
bundles beginning with the structure sheaf. At each point blow-up, apply
\cref{lem:point-blow-up-collection}. The new
collection again begins with the structure sheaf, so the construction can
be iterated. This also covers infinitely near centers, since the blow-ups
are performed one at a time.
\end{proof}

Together, \cref{prop:log-del-pezzo-rational,cor:rational-surface-fec} yield the following consequence.

\begin{corollary}\label{cor:resolution-fec}
Let \(X\) be a log del Pezzo surface, and let
\(\nu:\widetilde X\to X\) be its minimal resolution. The category
\(\Db(\coh \widetilde X)\) admits a full exceptional collection of line
bundles.
\end{corollary}

\section{Canonical stacks and exceptional collections}
\label{sec:canonical-stack}

The resolution-side input is \cref{cor:resolution-fec}. The special McKay
embedding places \(\Db(\coh\widetilde X)\) fully faithfully inside
\(\Db(\coh\XX)\), where residual-gerbe objects give the additional
exceptional terms at quotient singularities. This yields the collections
and length formulas in \cref{prop:stack-existence}. For cascade blow-ups,
compatibility of canonical stacks with smooth-point blow-ups and the stack
version of Orlov's formula describe how the construction extends.

\begin{definition}
Let \(X\) be a normal surface with quotient singularities. The \emph{canonical stack} of \(X\) is the smooth Deligne--Mumford stack \(\pi:\XX\to X\) whose coarse moduli space is \(X\), whose local models at quotient singularities are the corresponding quotient stacks, and whose structure morphism is an isomorphism over the smooth locus of \(X\).
Existence for quotient singularities follows from Vistoli's construction
\cite[Proposition~2.8 and its proof]{Vistoli1989}. Uniqueness up to
unique isomorphism over \(X\) follows from the universal property in
\cite[Theorem~4.6 and Corollary~4.8]{FantechiMannNironi2010}.
\end{definition}

Following \cite[Remark~4.9]{FantechiMannNironi2010}, in the next lemma
we write \(Y^{\mathrm{can}}\) for the canonical stack of a normal surface
\(Y\) with quotient singularities. We will use the compatibility of this
construction with blow-ups whose centers lie in the smooth locus.

\begin{lemma}[Canonical stacks and smooth blow-ups]
\label{lem:canonical-stack-smooth-blowup}
Let \(X\) be a normal surface with quotient singularities and let
\(Z\subset X_{\mathrm{sm}}\) be a smooth closed center. With the notation
\(Y^{\mathrm{can}}\) of
\cite[Remark~4.9]{FantechiMannNironi2010}, there is a canonical
isomorphism
\[
\bigl(\Bl_ZX\bigr)^{\mathrm{can}}
\simeq
\Bl_Z\bigl(X^{\mathrm{can}}\bigr),
\]
where \(Z\) is identified with its inverse image in
\(X^{\mathrm{can}}\).
In particular, this applies when \(Z\) is a finite union of distinct
smooth points.
\end{lemma}

\begin{proof}
The coarse-space morphism \(X^{\mathrm{can}}\to X\) is an isomorphism over
\(X_{\mathrm{sm}}\), so \(Z\) lifts canonically to a smooth closed
substack of \(X^{\mathrm{can}}\) with trivial stabilizer. Set
\(\mathcal Y=\Bl_Z(X^{\mathrm{can}})\). Near \(Z\), the stack
\(X^{\mathrm{can}}\) is a scheme and \(\mathcal Y\) is the ordinary
blow-up along a smooth center. Away from \(Z\), the blow-up morphism is an
isomorphism. Hence
\(\mathcal Y\) is a smooth Deligne--Mumford stack, no new stabilizers are
introduced, and the quotient-stack charts near the original
singularities of \(X\) are unchanged.

These local descriptions identify \(\Bl_ZX\) as the coarse moduli space
of \(\mathcal Y\) and show that the coarse-space morphism
\(\mathcal Y\to\Bl_ZX\) is an isomorphism over the smooth locus of
\(\Bl_ZX\). The uniqueness characterization of the canonical stack
\cite[Theorem~4.6 and Corollary~4.8]{FantechiMannNironi2010} then gives
\(\mathcal Y\simeq(\Bl_ZX)^{\mathrm{can}}\).
\end{proof}

The following is the algebraic-stack version of Orlov's blow-up formula,
due in this generality to Bergh--Schn\"urer
\cite{BerghSchnurer2020}. We state it directly in the bounded coherent
form used below.

\begin{proposition}[Orlov--Bergh--Schn\"urer blow-up formula]
\label{prop:stack-blow-up-formula}
Let \(\mathcal X\) be a noetherian algebraic stack, and let
\(\iota:\mathcal Z\hookrightarrow\mathcal X\) be a regular closed
immersion of constant codimension \(c\geq2\). Let
\[
\pi:\widetilde{\mathcal X}=\Bl_{\mathcal Z}\mathcal X\longrightarrow
\mathcal X
\]
be the blow-up, with exceptional divisor
\(\kappa:\mathcal D\hookrightarrow\widetilde{\mathcal X}\) and projection
\(\rho:\mathcal D\to\mathcal Z\). For each \(-c+1\leq i\leq-1\), define
\[
\begin{aligned}
\Phi_i:\Db(\coh\mathcal Z)&\longrightarrow
\Db(\coh\widetilde{\mathcal X}),
&F&\longmapsto
\kappa_*\bigl(\rho^*F\otimes\OO_{\mathcal D}(i)\bigr),\\
\Phi_0:\Db(\coh\mathcal X)&\longrightarrow
\Db(\coh\widetilde{\mathcal X}),
&F&\longmapsto L\pi^*F.
\end{aligned}
\]
The functors \(\Phi_{-c+1},\ldots,\Phi_{-1},\Phi_0\) are fully faithful,
and their essential images form the semiorthogonal decomposition
\[
\Db(\coh\widetilde{\mathcal X})
=
\left\langle
\Phi_{-c+1}\bigl(\Db(\coh\mathcal Z)\bigr),\ldots,
\Phi_{-1}\bigl(\Db(\coh\mathcal Z)\bigr),
\Phi_0\bigl(\Db(\coh\mathcal X)\bigr)
\right\rangle.
\]
\end{proposition}

\begin{proof}
Bergh--Schn\"urer first work with unbounded quasi-coherent complexes.
For \(-c+1\leq i\leq-1\), their convention
\(\OO_{\widetilde{\mathcal X}}(1)\simeq
\OO_{\widetilde{\mathcal X}}(-\mathcal D)\) gives
\[
\begin{aligned}
\Phi_i^{\mathrm{qc}}:
D_{\mathrm{qc}}(\mathcal Z)&\longrightarrow
D_{\mathrm{qc}}(\widetilde{\mathcal X}),
&F&\longmapsto
\OO_{\widetilde{\mathcal X}}(-i\mathcal D)\otimes\kappa_*\rho^*F,\\
\Phi_0^{\mathrm{qc}}:
D_{\mathrm{qc}}(\mathcal X)&\longrightarrow
D_{\mathrm{qc}}(\widetilde{\mathcal X}),
&F&\longmapsto L\pi^*F.
\end{aligned}
\]
The projection formula and
\(\kappa^*\OO_{\widetilde{\mathcal X}}(-i\mathcal D)\simeq\OO_{\mathcal D}(i)\)
rewrite \(\Phi_i^{\mathrm{qc}}(F)\) as
\(\kappa_*(\rho^*F\otimes\OO_{\mathcal D}(i))\), which is the form used in the
statement.
The blow-up theorem of Bergh--Schn\"urer gives full faithfulness and the
semiorthogonal decomposition on \(D_{\mathrm{qc}}\)
\cite[Theorem~6.9]{BerghSchnurer2020}. Its restriction to the locally
bounded pseudo-coherent categories is established in
\cite[Corollary~6.10]{BerghSchnurer2020}. For every
noetherian algebraic stack \(\mathcal Y\), Remark~2.6 identifies
\(D_{\mathrm{pc}}^{\mathrm{lb}}(\mathcal Y)=\Db(\coh\mathcal Y)\)
\cite[Remark~2.6]{BerghSchnurer2020}. Applying this identification to
\(\mathcal X\), \(\mathcal Z\), and \(\widetilde{\mathcal X}\)
identifies the restrictions of \(\Phi_i^{\mathrm{qc}}\) with the
bounded-coherent functors \(\Phi_i\) in the statement and gives the
claimed decomposition.
\end{proof}

The next corollary is the Deligne--Mumford stack analogue of the point-blow-up
construction in \cref{lem:point-blow-up-collection}.

\begin{corollary}[Point blow-ups of smooth Deligne--Mumford surfaces]
\label{cor:stack-point-blow-up-formula}
Let \(\mathcal S\) be a smooth noetherian Deligne--Mumford surface, let
\(p_1,\ldots,p_m\in\mathcal S\) be distinct points with trivial
stabilizer, and let
\[
\beta:\widetilde{\mathcal S}
=\Bl_{p_1,\ldots,p_m}\mathcal S\longrightarrow\mathcal S
\]
be the blow-up, with exceptional divisors
\(\mathcal D_1,\ldots,\mathcal D_m\simeq\Pp^1\). Then
\[
\Db(\coh\widetilde{\mathcal S})
=
\left\langle
\OO_{\mathcal D_1}(-1),\ldots,\OO_{\mathcal D_m}(-1),
L\beta^*\Db(\coh\mathcal S)
\right\rangle.
\]

Suppose moreover that \(\mathcal S\) admits a full exceptional collection
\[
\left\langle
\OO_{\mathcal S},F_1,\ldots,F_r
\right\rangle.
\]
Then
\[
\left\langle
\OO_{\widetilde{\mathcal S}},
\OO_{\widetilde{\mathcal S}}(\mathcal D_1),\ldots,
\OO_{\widetilde{\mathcal S}}(\mathcal D_m),
L\beta^*F_1,\ldots,L\beta^*F_r
\right\rangle
\]
is a full exceptional collection on \(\widetilde{\mathcal S}\). In
particular, if the \(F_j\) are line bundles, this is a full exceptional
collection of line bundles.
\end{corollary}

\begin{proof}
Apply \cref{prop:stack-blow-up-formula} to the regular immersion of
codimension two
\(\mathcal Z=p_1\sqcup\cdots\sqcup p_m\hookrightarrow\mathcal S\).
Since the centers have trivial stabilizer, \(\mathcal Z\) is a disjoint
union of copies of \(\Spec\C\), and the exceptional divisor over each
point is \(\Pp^1\). Thus the component
\(\Phi_{-1}\Db(\coh\mathcal Z)\) splits into mutually orthogonal
components generated by \(\OO_{\mathcal D_i}(-1)\), giving the stated
semiorthogonal decomposition.

A neighborhood of each center is a scheme because its stabilizer is
trivial. There the divisor sequence
\[
0\longrightarrow\OO_{\widetilde{\mathcal S}}
\longrightarrow\OO_{\widetilde{\mathcal S}}(\mathcal D_i)
\longrightarrow\OO_{\mathcal D_i}(-1)
\longrightarrow0
\]
identifies the replacement of \(\OO_{\mathcal D_i}(-1)\) by
\(\OO_{\widetilde{\mathcal S}}(\mathcal D_i)\) with the divisor-sequence
argument of \cref{lem:point-blow-up-collection}. It follows that the
resulting collection is full and exceptional. Pullback preserves line
bundles, giving the final assertion.
\end{proof}

We now describe the local representation theory that determines the
residual-gerbe terms in the special McKay decomposition. At a singular point
of type \(\frac1n(1,1)\), with \(n\geq2\), the local analytic model is
\(\A^2/\boldsymbol\mu_n\), and the canonical stack is locally
\([\A^2/\boldsymbol\mu_n]\). We use the ring-action convention
\(\rho_i(\zeta)=\zeta^i\), so the coordinate functions
\(x,y\in\C[x,y]\) both have weight \(\rho_1\). All character labels in
this local discussion follow this convention.

For a finite small subgroup \(G\subset GL_2(\mathbb C)\) \cite{Ishii2002}, let
\(V_{\mathrm{coord}}\) denote the two-dimensional coordinate representation
on \(\mathbb Cx\oplus\mathbb Cy\). We define the McKay quiver of \(G\) to
have vertex set \(\operatorname{Irr}(G)\). For irreducible representations
\(\rho_i,\rho_j\), the number of arrows \(\rho_i\to\rho_j\) equals
\[
\dim\operatorname{Hom}_G(\rho_j,\rho_i\otimes V_{\mathrm{coord}}).
\]
When \(G\subset SL_2(\mathbb C)\), the McKay quiver is the double quiver
of an affine ADE graph. Deleting its trivial vertex produces the dual graph
of the exceptional curves in the minimal resolution. The classical McKay
correspondence identifies these graphs \cite{McKay1980}. Kapranov--Vasserot
studied the derived correspondence for Kleinian surface singularities,
while Bridgeland--King--Reid established a derived equivalence in the
\(G\)-Hilbert scheme framework
\cite{KapranovVasserot2000,BridgelandKingReid2001}.

When \(G\not\subset SL_2(\mathbb C)\), some nontrivial irreducible
representations do not correspond to exceptional curves. For a
representation \(\rho\), we set
\(M_\rho=(\mathbb C[x,y]\otimes\rho^\vee)^G\), and we denote its associated
full sheaf on the minimal resolution \(Y\) by \(\widetilde M_\rho\).
Following Wunram, we call \(\rho\) special when
\[
H^1(Y,\widetilde M_\rho^\vee)=0
\]
\cite{Wunram1988}. The nontrivial special representations correspond
bijectively to the irreducible exceptional curves
\cite{Wunram1988,Riemenschneider2003}. We normalize this correspondence by
\[
c_1(\widetilde M_{\rho_t})\cdot E_s=\delta_{ts}.
\]

For the cyclic quotient \(\frac1n(1,q)\), where \(\gcd(n,q)=1\), let
\(\zeta\) act by \(\zeta\cdot(x,y)=(\zeta x,\zeta^q y)\). Under the
ring-action convention \(\rho_i(\zeta)=\zeta^i\), the coordinate
representation is
\[
V_{\mathrm{coord}}=\rho_1\oplus\rho_q.
\]
Thus the McKay quiver has arrows \(\rho_i\to\rho_{i+1}\) and
\(\rho_i\to\rho_{i+q}\), with indices modulo \(n\). Write the
Hirzebruch--Jung expansion as \(\frac nq=[a_1,\ldots,a_m]\)
\cite{Riemenschneider1974}. Define its \(I\)-series by
\[
i_0=n,\qquad i_1=q,\qquad i_{t+1}=a_ti_t-i_{t-1}.
\]
With the convention \(M_\ell=(\mathbb C[x,y]\otimes\rho_\ell^\vee)^{\boldsymbol\mu_n}\),
Wunram's cyclic correspondence identifies the special characters as
\(\rho_0,\rho_{i_1},\ldots,\rho_{i_m}\)
\cite{Wunram1988,Riemenschneider2003}. For each \(t\),
the character \(\rho_{i_t}\) corresponds to the exceptional curve \(E_t\),
which satisfies \(E_t^2=-a_t\) and
\(c_1(\widetilde M_{i_t})\cdot E_s=\delta_{ts}\). Consequently, the
remaining \(n-m-1\) irreducible characters are non-special.

\begin{lemma}\label{lem:non-special-cyclic}
Let \(X\) be a normal surface with singular points \(p_1,\ldots,p_k\), where \(p_i\) has
type \(\frac1{n_i}(1,1)\), with \(n_i\geq2\). At \(p_i\), the
special representations of \(\boldsymbol\mu_{n_i}\) are precisely
\(\rho_0\) and \(\rho_1\). The non-special representations are
\(\rho_2,\ldots,\rho_{n_i-1}\), so there are \(n_i-2\) of them.
For \(n_i=2\) this list is empty, while for \(n_i=3\) it consists of
\(\rho_2\) alone.
\end{lemma}

\begin{proof}
For \(\frac1{n_i}(1,1)\), the Hirzebruch--Jung expansion is
\(n_i/1=[n_i]\), and its \(I\)-series is
\(i_0=n_i\), \(i_1=1\), \(i_2=0\). The recalled cyclic correspondence
therefore gives the special characters \(\rho_0\) and \(\rho_1\), while
the remaining characters \(\rho_2,\ldots,\rho_{n_i-1}\) are non-special.
\end{proof}

\begin{definition}\label{def:local-cyclic-blocks}
Let \(X\) be a normal surface with singular points \(p_1,\ldots,p_k\) of types
\(\frac1{n_i}(1,1)\), with \(n_i\geq2\), and let \(\XX\) be its
canonical stack. For the residual gerbe
\(\iota_{p_i}:B\boldsymbol\mu_{n_i}\hookrightarrow\XX\),
define the \emph{local exceptional objects} associated with the non-special
characters identified in \cref{lem:non-special-cyclic} by
\[
\mathcal E_{p_i,j}:=
\iota_{p_i*}\bigl(\OO_{B\boldsymbol\mu_{n_i}}\otimes\rho_j\bigr),
\qquad 2\leq j\leq n_i-1,
\]
and their ordered exceptional collection
\[
\mathscr B_{p_i}:=
\langle \mathcal E_{p_i,2},\ldots,\mathcal E_{p_i,n_i-1}\rangle.
\]
The collection is empty when \(n_i=2\). Locally,
\(\mathcal E_{p_i,j}\) is represented by the equivariant module
\(\C[x,y]/(x,y)\otimes\rho_j\),
namely \(\OO_0\otimes\rho_j\) in the notation of Ishii--Ueda.
For a point \(p\) of type \(\frac13(1,1)\), we retain the abbreviation
\(\mathcal E_p:=\mathcal E_{p,2}\).
\end{definition}

Each listed local object is exceptional. Indeed, the equivariant Koszul
resolution at the origin gives
\[
\operatorname{Ext}^m_{[\A^2/\boldsymbol\mu_{n_i}]}
(\OO_0\otimes\rho_j,\OO_0\otimes\rho_j)
\cong
\bigl(\wedge^m V_{\mathrm{coord}}^\vee\bigr)^{\boldsymbol\mu_{n_i}}.
\]
Here \(V_{\mathrm{coord}}=\rho_1^{\oplus2}\), so these invariant spaces
are \(\C\) in degree zero and vanish in positive degrees for \(n_i\geq3\).
When \(n_i=2\), the collection is empty. Thus every object in
\(\mathscr B_{p_i}\) is exceptional. Ishii--Ueda's local decomposition
places these objects in increasing character order as an exceptional
collection \cite[Theorem~3.1]{Ishii2015}.

\begin{theorem}[Ishii--Ueda {\cite[Theorems~1.4 and~3.1]{Ishii2015}}]\label{thm:ishii-ueda-global}
Let \(X\) be a surface with at worst quotient singularities, let \(\XX\) be its canonical stack, and let \(Y\to X\) be the minimal resolution. There is a fully faithful functor
\[
\Phi:\Db(\coh Y)\to \Db(\coh \XX)
\]
and a semiorthogonal decomposition
\[
\Db(\coh \XX)
=
\left\langle
\mathcal E_1,\dots,\mathcal E_{N_{\mathrm{loc}}},\Phi(\Db(\coh Y))
\right\rangle,
\]
where \(\mathcal E_1,\dots,\mathcal E_{N_{\mathrm{loc}}}\) form an exceptional collection.

If the singularities are \(p_1,\ldots,p_k\) of types
\(\frac1{n_i}(1,1)\), with \(n_i\geq2\), the local description
\cite[Theorem~3.1]{Ishii2015} and
\cite[Example~2.6]{GugiattiRota2023} gives
\[
\Db(\coh\XX)=
\left\langle
\mathscr B_{p_1},\ldots,\mathscr B_{p_k},
\Phi\bigl(\Db(\coh Y)\bigr)
\right\rangle,
\]
with the ordered local collections of \cref{def:local-cyclic-blocks} and
\(N_{\mathrm{loc}}=\sum_{i=1}^k(n_i-2)\). Collections at distinct
points are mutually orthogonal and can be ordered arbitrarily. For each
point, the local collection is
\[
\mathscr B_{p_i}=
\langle\mathcal E_{p_i,2},\ldots,\mathcal E_{p_i,n_i-1}\rangle.
\]
When \(n_i=2\), it is empty, and when \(n_i=3\), it contains one
object. For \(n_i\geq4\), the increasing-character order matters in
general. Indeed, for \(2\leq j\leq n_i-2\), the equivariant Koszul
calculation gives
\[
\operatorname{Ext}^1(\mathcal E_{p_i,j},\mathcal E_{p_i,j+1})
\simeq\C^2,
\]
so the objects cannot simply be reordered arbitrarily within a local
collection.
\end{theorem}

Write \(\nu:Y\to X\) for the minimal resolution and
\(\pi:\mathcal X\to X\) for the coarse-moduli map. The Ishii--Ueda
functor
\[
\Phi:\Db(\coh Y)\to\Db(\coh\mathcal X)
\]
is the Fourier--Mukai transform with kernel supported on
\(Y\times_X\mathcal X\). By the projection formula, for
\(M\in\Perf(X)\) and \(F\in\Db(\coh Y)\), there is a natural
isomorphism
\begin{equation}\label{eq:ishii-ueda-projection-formula}
\Phi(F\otimes L\nu^*M)
\simeq
\Phi(F)\otimes L\pi^*M.
\end{equation}
In particular, this applies when \(M\) is a line bundle associated with
a Cartier divisor on \(X\), and \(\Phi(\OO_Y)\simeq\OO_{\mathcal X}\).
We use this compatibility below for divisor classes supported away from
the singular locus and to compare the two exceptional collections on
\(\mathcal X_{10}\).

\begin{example}[Crepant versus non-crepant McKay for $\mu_3$]\label{ex:mckay}
For the Du Val singularity $\tfrac13(1,2)$, the group lies in
$\operatorname{SL}_2(\mathbb C)$ and the minimal resolution is crepant.
The derived McKay correspondence gives an equivalence
\[
\Db(\coh Y)\simeq\Db(\coh[\A^2/\boldsymbol\mu_3]),
\]
as established by Kapranov--Vasserot and
Bridgeland--King--Reid~\cite{KapranovVasserot2000,BridgelandKingReid2001}.
For $\tfrac13(1,1)$, the group does not lie in
$\operatorname{SL}_2(\mathbb C)$, so the resolution is not crepant. The
Ishii--Ueda category then has a residual-gerbe contribution. The local
comparison is illustrated in Figure~\ref{fig:mckay}.
\end{example}

\begin{figure}[H]
\centering
\resizebox{\textwidth}{!}{%
\begin{tikzpicture}[
  >={Stealth[length=2.4mm]},
  every node/.style={font=\small},
  vertex/.style={circle,draw,thick,minimum size=7.5mm,inner sep=0pt,fill=white,font=\normalsize},
  sp/.style={vertex,draw=special,text=special},
  nsp/.style={vertex,draw=nonspecial,text=nonspecial,fill=nonspecial!8},
  tr/.style={vertex,draw=trivial,text=trivial},
  arr/.style={->,thick,draw=black!70},
  lab/.style={font=\footnotesize,inner sep=1pt,fill=white},
  cv/.style={circle,fill=curve,inner sep=0pt,minimum size=2.6mm},
  cvlab/.style={font=\footnotesize,text=curve},
  panelbox/.style={rounded corners=3pt,fill=panel,draw=black!15},
  title/.style={font=\bfseries\small,anchor=north},
  caption/.style={font=\footnotesize,text width=5.6cm,align=center}
]
\begin{scope}[local bounding box=L]
  \node[title,align=center] at (0,3.6) {Du Val point $\tfrac13(1,2)$\\[1pt]{\mdseries $\zeta\cdot(x,y)=(\zeta x,\zeta^{2}y)$,\quad $\boldsymbol\mu_3\subset \mathrm{SL}_2$}};
  \node[tr] (a0) at (0,2.2) {$\rho_0$};
  \node[sp] (a1) at (-1.5,0.3) {$\rho_1$};
  \node[sp] (a2) at (1.5,0.3) {$\rho_2$};
  \draw[arr] (a0) to[bend right=18] node[lab,pos=.5] {$x$} (a1);
  \draw[arr] (a1) to[bend right=18] node[lab,pos=.5] {$y$} (a0);
  \draw[arr] (a1) to[bend right=18] node[lab,pos=.5] {$x$} (a2);
  \draw[arr] (a2) to[bend right=18] node[lab,pos=.5] {$y$} (a1);
  \draw[arr] (a2) to[bend right=18] node[lab,pos=.5] {$x$} (a0);
  \draw[arr] (a0) to[bend right=18] node[lab,pos=.5] {$y$} (a2);
  \node[caption,text width=4.4cm] at (0,-0.8) {McKay quiver: double quiver of affine $\widetilde{A}_2$ graph};
  \node[cv,label={[cvlab]below:$E_1,\ E_1^2=-2$}] (e1) at (-1.3,-2.3) {};
  \node[cv,label={[cvlab]below:$E_2,\ E_2^2=-2$}] (e2) at (1.3,-2.3) {};
  \draw[curve,very thick] (e1) -- (e2);
  \draw[special,densely dashed,->] (a1.south west) to[bend right=40] (e1.west);
  \draw[special,densely dashed,->] (a2.south east) to[bend left=40] (e2.east);
  \node[caption,text width=6.4cm] at (0,-4.1)
    {all characters are special.\\
     $\Phi\colon \Db(\operatorname{coh}Y)\xrightarrow{\ \sim\ }\Db([\A^2/\boldsymbol\mu_3])$\\
     no residual-gerbe object};
\end{scope}
\begin{scope}[xshift=8.6cm,local bounding box=R]
  \node[title,align=center] at (0,3.6) {The point $\tfrac13(1,1)$\\[1pt]{\mdseries $\zeta\cdot(x,y)=(\zeta x,\zeta y)$,\quad $\boldsymbol\mu_3\not\subset \mathrm{SL}_2$}};
  \node[tr] (b0) at (0,2.2) {$\rho_0$};
  \node[sp] (b1) at (-1.5,0.3) {$\rho_1$};
  \node[nsp] (b2) at (1.5,0.3) {$\rho_2$};
  \draw[arr] (b0) to[bend right=28] node[lab,pos=.45] {$x$} (b1);
  \draw[arr] (b0) to[bend right=4] node[lab,pos=.55] {$y$} (b1);
  \draw[arr] (b1) to[bend right=28] node[lab,pos=.5] {$x$} (b2);
  \draw[arr] (b1) to[bend right=4] node[lab,pos=.5] {$y$} (b2);
  \draw[arr] (b2) to[bend right=28] node[lab,pos=.55] {$x$} (b0);
  \draw[arr] (b2) to[bend right=4] node[lab,pos=.45] {$y$} (b0);
  \node[caption,text width=4.4cm] at (0,-0.8) {McKay quiver: oriented $3$-cycle, each arrow has multiplicity two};
  \node[cv,label={[cvlab]below:$E,\ E^2=-3$}] (e) at (-1.5,-2.3) {};
  \draw[special,densely dashed,->] (b1.south west) to[bend right=35] (e.west);
  \node[nonspecial,font=\small,align=center] (gerbe) at (1.4,-2.45)
    {$\mathcal E_p=\mathcal O_0\otimes\rho_2$\\[-1pt]{\footnotesize residual-gerbe object}};
  \draw[nonspecial,densely dashed,->] (b2.south) to[bend left=10] (gerbe.north);
  \node[caption,text width=6.4cm] at (0,-4.1)
    {special: $\rho_0,\rho_1$ (\cref{lem:non-special-cyclic})\\ non-special: $\rho_2$\\
     $\Db([\A^2/\boldsymbol\mu_3])=\bigl\langle \mathcal E_p,\ \Phi\Db(\operatorname{coh}Y)\bigr\rangle$};
\end{scope}
\begin{scope}[on background layer]
  \node[panelbox,fit=(L),inner sep=8pt] {};
  \node[panelbox,fit=(R),inner sep=8pt] {};
\end{scope}
\end{tikzpicture}%
}
\caption{Crepant and non-crepant $\mu_3$-quotients.}
\label{fig:mckay}
\end{figure}

For rational quotient surfaces, combining
\cref{cor:rational-surface-fec} with the local decomposition of
\cref{thm:ishii-ueda-global} gives a full exceptional collection on the
canonical stack, as also noted in \cite[Remark~4.3]{GugiattiRota2023}.

\begin{proposition}[Exceptional collections on rational quotient surfaces]
\label{prop:stack-existence}
Let \(X\) be a normal projective rational surface with quotient
singularities, let \(\XX\) be its canonical stack, and let
\(\nu:\widetilde X\to X\) be its minimal resolution. Then
\(\Db(\coh\XX)\) admits a full exceptional collection. In particular,
this holds for every log del Pezzo surface.

More precisely, suppose that the singular points
\(p_1,\ldots,p_k\) are of types \(\frac1{n_i}(1,1)\), with
\(n_i\geq2\). With the local collections \(\mathscr B_{p_i}\) of
\cref{def:local-cyclic-blocks}, any full exceptional collection
\(\langle G_1,\ldots,G_N\rangle\) on \(\widetilde X\) gives the ordered
full exceptional collection
\[
\left\langle
\mathscr B_{p_1},\ldots,\mathscr B_{p_k},
\Phi(G_1),\ldots,\Phi(G_N)
\right\rangle
\]
on \(\XX\). The lengths of the resolution and stack collections are,
respectively,
\[
\begin{aligned}
N&=\rk K_0(\widetilde X)
=12-K_X^2+\sum_{i=1}^k\frac{(n_i-2)^2}{n_i},\\
\ell&=N+\sum_{i=1}^k(n_i-2)
=12-K_X^2+\sum_{i=1}^k
\frac{2(n_i-1)(n_i-2)}{n_i}.
\end{aligned}
\]
By \cref{cor:resolution-fec}, the collection
\(\langle G_1,\ldots,G_N\rangle\) may be chosen to consist of line bundles.
\end{proposition}

\begin{proof}
The resolution \(\widetilde X\) is smooth and rational, so
\cref{cor:rational-surface-fec} supplies a full exceptional collection
of line bundles.
Inserting it into the Ishii--Ueda decomposition of
\cref{thm:ishii-ueda-global} proves the general existence statement. The
log del Pezzo case follows from \cref{prop:log-del-pezzo-rational}.

Now assume that the singularities have the stated cyclic types, and write
\(E_i\subset\widetilde X\)
for the exceptional curve over \(p_i\). Thus \(E_i^2=-n_i\), and the
discrepancy formula is
\[
K_{\widetilde X}
=
\nu^*K_X-\sum_{i=1}^k\frac{n_i-2}{n_i}E_i.
\]
To verify the coefficient of \(E_i\), write it as \(a_i\) and
intersect with \(E_i\). Adjunction gives
\(K_{\widetilde X}\cdot E_i=-2-E_i^2=n_i-2\), whereas
\(\nu^*K_X\cdot E_i=0\). Hence
\(-n_i a_i=n_i-2\), giving \(a_i=-(n_i-2)/n_i\).
The curves \(E_i\) are disjoint, so
\[
K_{\widetilde X}^2
=K_X^2-\sum_{i=1}^k\frac{(n_i-2)^2}{n_i}.
\]
Noether's formula
\cite[Appendix~A, Example~4.1.2]{Hartshorne1977} gives
\(12\chi(\OO_{\widetilde X})=K_{\widetilde X}^2+c_2(\widetilde X)\).
For the smooth rational surface \(\widetilde X\)
(\cref{prop:log-del-pezzo-rational}), we have
\(\chi(\OO_{\widetilde X})=1\) and
\(c_2(\widetilde X)=2+\rho(\widetilde X)\), the latter being its
topological Euler characteristic, so
\(\rho(\widetilde X)=10-K_{\widetilde X}^2\). Hence
\[
N=\rk K_0(\widetilde X)=2+\rho(\widetilde X)
=12-K_{\widetilde X}^2
=12-K_X^2+\sum_{i=1}^k\frac{(n_i-2)^2}{n_i}.
\]
Inserting the resolution collection into
\cref{thm:ishii-ueda-global} gives the displayed collection, with
\(n_i-2\) additional objects at each point. Its length is
\[
\ell=N+\sum_{i=1}^k(n_i-2)
=12-K_X^2+\sum_{i=1}^k
\frac{2(n_i-1)(n_i-2)}{n_i}.\qedhere
\]
\end{proof}

\begin{remark}\label{rem:ishii-ueda-derived-images}
The images of line bundles on the resolution need not be line bundles or
even sheaves on the canonical stack. Locally at a cyclic quotient
singularity, the Ishii--Ueda embedding is described by the integral
transform associated with the universal (G)-cluster \cite{Ishii2002}
\[
\mathcal Z\subset Y\times\A^2,
\qquad
p:\mathcal Z\to Y,
\qquad
q:\mathcal Z\to\A^2,
\]
and the local functor has the form
\[
\phi(F)=Rq_*(Lp^*F),
\]
with its equivariant structure giving an object on
\([\A^2/G]\) \cite[Section~1]{Ishii2015}. This is the local cyclic
model for the global embedding \(\Phi\) of
\cref{thm:ishii-ueda-global}.

For a singularity of type \(\frac1n(1,1)\), let
\(\pi_Y:Y=\operatorname{Tot}_{\Pp^1}\OO_{\Pp^1}(-n)\to\Pp^1\) be its
local resolution, and set \(P_a:=\pi_Y^*\OO_{\Pp^1}(a)\) and
\(\mathcal L_a:=\OO_{\A^2}\otimes\rho_a\). The two special characters give
\[
\phi(P_0)\simeq\mathcal L_0,
\qquad
\phi(P_{-1})\simeq\mathcal L_1.
\]
Applying \(\phi\) to the Euler sequence gives
\[
\phi(P_1)
\simeq
\operatorname{coker}\bigl(
\mathcal L_1\longrightarrow\mathcal L_0^{\oplus2}
\bigr)
\simeq
(x,y)\otimes\rho_{n-1}.
\]
This is a rank-one torsion-free sheaf which is not locally free at the
origin. Thus the collections of \cref{prop:stack-existence} may be
chosen to consist of line bundles on the minimal resolution, while their
images under \(\Phi\) are derived objects on the canonical stack. The
full strong exceptional collection of
\cref{thm:one-point-stack-strong} is constructed directly on the
canonical stack and consists of line bundles there.
\end{remark}

For the singularities used in the Corti--Heuberger cascades, the local
collections consist of single exceptional objects. Setting \(n_i=3\) in \cref{prop:stack-existence}
gives the following specialization.

\begin{corollary}\label{cor:length}
Let \(X\) be a log del Pezzo surface whose singularities are
\(p_1,\ldots,p_k\), all of type \(\frac13(1,1)\). Let \(\XX\) be
its canonical stack, and let
\(\widetilde X\to X\) be its minimal resolution. Then
\(\Db(\coh\XX)\) admits a full exceptional collection of length
\(\ell=12-K_X^2+\frac{4k}{3}\). The resolution collection has length
\(N=\rk K_0(\widetilde X)=12-K_X^2+\frac k3\), and each local collection in
\cref{prop:stack-existence} is the singleton
\(\mathcal E_{p_i}:=\mathcal E_{p_i,2}\).
\end{corollary}

\section{Derived categories of the singular surface}
\label{sec:kks}

Sections 3 and 4 describe the derived categories of the minimal resolution
$\widetilde X$ and the canonical stack $\mathcal X$ by full exceptional
collections. The coarse moduli space differs from these smooth spaces
because, when $X$ is singular,
\[
\operatorname{Perf}(X)\subsetneq D^b(\operatorname{coh}X).
\]
For the isolated quotient singularities considered here, Orlov's
formal-locality theorem and the equivalences of Kalck--Karmazyn make the
idempotent completion of this quotient comparatively accessible. This local information does not
determine the full bounded coherent derived category. The stronger problem
of describing that category is addressed by the descent theory of
Karmazyn--Kuznetsov--Shinder.

\begin{definition}[Perfect complexes and the singularity category]
Let $X$ be a noetherian scheme. An object
\[
P\in D^b(\operatorname{coh}X)
\]
is perfect if it is locally quasi-isomorphic to a bounded complex of
finite-rank locally free sheaves. Write
\[
\operatorname{Perf}(X)\subset D^b(\operatorname{coh}X)
\]
for the full triangulated subcategory of perfect complexes. The
triangulated category of singularities is
\[
D_{\mathrm{sg}}(X)
:=D^b(\operatorname{coh}X)/\operatorname{Perf}(X).
\]
We denote its idempotent completion by
\[
\overline{D_{\mathrm{sg}}(X)}.
\]
A Verdier quotient need not be idempotent complete, so this completion is
part of our notation for the singularity category used below.
\end{definition}

The following standard observation explains the categorical difference
between the smooth resolution and the singular surface \(X\). It is
stated explicitly, for example, in Ballard--Duncan--McFaddin
\cite[Remark~2.5]{BallardDuncanMcFaddin2019}, referring to Orlov
\cite[Propositions~3.30--3.31]{Orlov2016}. We include a short argument
in the present setting for completeness.

\begin{proposition}\label{prop:full-exceptional-implies-smooth}
Let \(X\) be a proper scheme of finite type over \(\C\) with the
resolution property. If
\[
\Db(\coh X)=\langle E_1,\ldots,E_n\rangle
\]
is a full exceptional collection, then \(X\) is smooth over \(\C\).
\end{proposition}

\begin{proof}
Since \(X\) is proper with the resolution property and \(\C\) is a
perfect field, Lunts--Schnürer \cite[Theorem~4.5]{LuntsSchnurer2016}
show that the dg enhancement of \(\Db(\coh X)\) is smooth. The full
exceptional collection expresses it as an iterated gluing of copies of
\(\Perf(\C)\). By Orlov's gluing criterion
\cite[Theorem~3.25]{Orlov2016}, the relevant gluing bimodules are
perfect. A perfect bimodule between proper dg categories has
finite-dimensional total cohomology on every pair of objects. Orlov's
properness criterion \cite[Proposition~3.22]{Orlov2016} therefore
implies that gluing two proper factors along such a bimodule is proper.
Induction on the length of the exceptional collection shows that the
full dg category is proper.

For every closed point \(x\),
\[
\bigoplus_{i\in\mathbb Z}
\operatorname{Ext}^i_X(\C(x),\C(x))
\]
is finite-dimensional. Hence
\(\operatorname{Ext}^i_X(\C(x),\C(x))=0\) for \(i\gg0\). These are
the Ext groups of the residue field over \(\OO_{X,x}\), so the residue
field has finite projective dimension over \(\OO_{X,x}\). By the
Auslander--Buchsbaum--Serre criterion, \(\OO_{X,x}\) is regular. Every
point is a generization of a closed point, and localizations of regular
local rings are regular. Thus \(X\) is regular and hence smooth over
\(\C\).
\end{proof}

In particular, under the hypotheses of the proposition,
\[
\operatorname{Perf}(X)=D^b(\operatorname{coh}X).
\]

\begin{corollary}\label{cor:singular-log-del-pezzo-no-full-ec}
A singular log del Pezzo surface \(X\) admits no full exceptional
collection in \(\Db(\coh X)\).
\end{corollary}

\begin{proof}
Projectivity implies both properness and the resolution property, so the
proposition applies and forces \(X\) to be smooth over \(\C\), a
contradiction.
\end{proof}

For isolated singularities, the idempotent-completed singularity
category is determined by the formal local rings:

\begin{proposition}\label{prop:singularity-category-local}
Let $X$ be a projective surface over $\C$ with isolated singularities
\[
\operatorname{Sing}(X)=\{p_1,\ldots,p_k\}.
\]
There is an equivalence
\[
\overline{D_{\mathrm{sg}}(X)}
\simeq
\bigoplus_{i=1}^k
\overline{D_{\mathrm{sg}}\!\left(\widehat{\mathcal O}_{X,p_i}\right)}.
\]
\end{proposition}

\begin{proof}
The formal completion of $X$ along its singular locus is
$\coprod_i\operatorname{Spf}\widehat{\mathcal O}_{X,p_i}$. Since $X$ is
of finite type over $\C$, its local rings are excellent, so each
completed local ring is regular away from its closed point. Thus this is
also the formal completion of
$Y:=\coprod_i\operatorname{Spec}\widehat{\mathcal O}_{X,p_i}$ along its
singular locus. Both schemes satisfy Orlov's condition (ELF), meaning
that they are separated, noetherian of finite Krull dimension, and have
enough locally free sheaves. Indeed, $X$ is projective and $Y$ is a
finite disjoint union of affine noetherian schemes of finite dimension.
Projective and affine schemes have enough locally free sheaves, so these
hypotheses hold for both. Orlov's formal-completion theorem
\cite[Theorem~2.10]{OrlovFormal2011} gives an equivalence after
idempotent completion. The singularity category of $Y$ is the direct
sum of those of its components, and idempotent completion preserves
finite direct sums. This gives the claimed equivalence.
\end{proof}

For $n\geq2$, define the corresponding local algebra by
\[
K(n,1):=\C[z_1,\ldots,z_{n-1}]/(z_1,\ldots,z_{n-1})^2.
\]

\begin{corollary}[Cyclic quotient singularities]
\label{cor:cyclic-singularity-categories}
Let $X$ be a projective surface over $\C$ with isolated singular points
$p_1,\ldots,p_k$. Suppose that the point $p_i$ has type
$\frac1{n_i}(1,1)$ for each $i$. Then the idempotent-completed
singularity category of $X$ is equivalent to
\[
\overline{D_{\mathrm{sg}}(X)}
\simeq
\bigoplus_{i=1}^k
\overline{D_{\mathrm{sg}}(K(n_i,1))}.
\]
If every singular point $p_i$ has type $\frac13(1,1)$, then the
idempotent-completed singularity category of $X$ is equivalent to
\[
\overline{D_{\mathrm{sg}}(X)}
\simeq
\bigoplus_{i=1}^k
\overline{D_{\mathrm{sg}}(K(3,1))}.
\]
These equivalences determine the idempotent-completed singularity
category for every Corti--Heuberger cascade, including the cascades with
two, four, or five singular points, without requiring a global KKS
semiorthogonal decomposition.
\end{corollary}

\begin{proof}
By \cref{prop:singularity-category-local},
\[
\overline{D_{\mathrm{sg}}(X)}
\simeq
\bigoplus_{i=1}^k
\overline{D_{\mathrm{sg}}\!\left(\widehat{\mathcal O}_{X,p_i}\right)}.
\]
For each $i$, the completed local ring is
\[
\widehat{\mathcal O}_{X,p_i}
\simeq \C[[x,y]]^{\frac1{n_i}(1,1)}.
\]
The Kalck--Karmazyn equivalence~\cite{KalckKarmazyn2017} induces an
equivalence after idempotent completion:
\[
\overline{D_{\mathrm{sg}}\!\left(\C[[x,y]]^{\frac1{n_i}(1,1)}\right)}
\simeq \overline{D_{\mathrm{sg}}(K(n_i,1))}.
\]
Substitution in the preceding decomposition proves the claim.
\end{proof}

This decomposition of the idempotent completion of the singularity
category into local contributions does not determine
\(D^b(\operatorname{coh}X)\) itself. To address the stronger problem of
realizing the local singular contributions as actual semiorthogonal
components of the global category, we use the descent theory of
Karmazyn--Kuznetsov--Shinder. Their method starts with a semiorthogonal
decomposition of the minimal resolution compatible with the contraction.
Under the stronger adherence condition, components assigned to exceptional
curves descend to derived categories of finite-dimensional local algebras.
We first recall KKS compatibility with the contraction.

\begin{definition}[{KKS compatibility, \cite[Definition~2.7]{KKS2020}}]\label{def:kks-compatibility}
Let \(\tau:\widetilde X\to X\) be a resolution of a normal projective surface
with rational singularities,
viewed as a birational contraction, and let
\[
\Db(\coh\widetilde X)
=
\left\langle
\widetilde{\mathcal A}_1,\dots,\widetilde{\mathcal A}_t
\right\rangle
\]
be a semiorthogonal decomposition. The decomposition is \emph{compatible
with the contraction \(\tau\)} if, for every irreducible curve \(E\)
contracted by \(\tau\), the object \(\OO_E(-1)\) belongs to one of the
components \(\widetilde{\mathcal A}_i\).
\end{definition}

\begin{theorem}[{KKS descent, \cite[Theorem~2.12(i)]{KKS2020}}]
\label{thm:kks-descent}
In the setting of \cref{def:kks-compatibility}, a compatible decomposition
descends uniquely to
\[
\Db(\coh X)=\langle\mathcal A_1,\ldots,\mathcal A_t\rangle,
\qquad
\mathcal A_i=R\tau_*(\widetilde{\mathcal A}_i).
\]
Let \(\mathcal D_i\) be the union of the exceptional curves assigned to
\(\widetilde{\mathcal A}_i\). The restriction of \(R\tau_*\) induces
an equivalence with the Verdier quotient
\[
\mathcal A_i\simeq
\widetilde{\mathcal A}_i/
\langle\OO_E(-1)\mid E\subset\mathcal D_i\rangle.
\]
Here \(E\) runs over the irreducible components of \(\mathcal D_i\).
In particular, pushforward is an equivalence on a component with no assigned curves.
\end{theorem}

The theorem's rational-singularity hypothesis holds here by the klt
property, as recalled in \cref{sec:rationality}.

To identify the descended components, we need the stronger condition of
adherence. When the exceptional locus consists of a single curve, adherence
has the following concrete description, with untwisted adherence given by
$b=0$.

\begin{definition}[{Twisted adherence to a single exceptional curve, \cite[Definition~3.6]{KKS2020}}]
\label{def:kks-adherence}
Let \(S\) be a smooth projective surface with
\(H^\bullet(S,\OO_S)\cong\C\). Let \(E\subset S\) be a smooth rational curve
with self-intersection \(E^2=-n\), where \(n\geq2\). We call a triangulated
subcategory \(\mathcal A\subset\Db(\coh S)\) \emph{$b$-twisted adherent} to
\(E\), for an integer \(b\), if it admits a presentation
\[
\mathcal A
=
\left\langle
L_0,\OO_E(b-1)
\right\rangle
\]
for a line bundle \(L_0\) on \(S\) whose intersection number with \(E\) equals
\(n+b-1\). The case \(b=0\) is called untwisted adherence.
\end{definition}

By \cite[Lemma~3.5, Definition~3.6, and Example~3.7]{KKS2020}, the
two presentations are equivalent when \(b=0\). For any \(b\), the same
divisor-sequence argument gives
\[
\mathcal A=\langle L_0,L_1\rangle,
\qquad
L_1=L_0(E),
\qquad
L_0\cdot E=n+b-1.
\]
The one-point decompositions below use \(b=0\),
while nonzero twists are needed for the Brauer-twisted six-point
construction of \cref{cor:six-point-twisted-descent}, whose Brauer class
is computed in Appendix~\ref{app:brauer-six}.

The following lemma proves preservation of compatibility and adherence
under blow-up for all the contracted curves simultaneously.

\begin{lemma}
\label{lem:kks-blowups}
Let \(S\) be a smooth projective surface and let
\(C_1,\ldots,C_s\subset S\) be pairwise disjoint smooth rational
curves with \(C_i^2=-n_i\), where \(n_i\geq2\). Let
\(\beta:S'\to S\) be the blow-up of a point outside
\(\bigcup_i C_i\), and write \(C_i'\) for the strict transforms.
If a semiorthogonal decomposition of \(\Db(\coh S)\) contains
components \(\mathcal A_i\) such that
\[
\OO_{C_i}(-1)\in\mathcal A_i,
\]
then \(\OO_{C_i'}(-1)\in L\beta^*\mathcal A_i\) for every \(i\).
More generally, if \(H^\bullet(S,\OO_S)\cong\C\) and subcategories
\(\mathcal A_i\) have \(b_i\)-twisted adherent presentations
\[
\mathcal A_i=\langle L_i,L_i(C_i)\rangle,
\qquad L_i\cdot C_i=n_i+b_i-1,
\]
then every \(L\beta^*\mathcal A_i\) is \(b_i\)-twisted adherent to
\(C_i'\) and contains \(\OO_{C_i'}(b_i-1)\).
In particular, untwisted adherence is preserved when all \(b_i=0\).
These conclusions hold after any finite sequence of point blow-ups
whose centers avoid the successive strict transforms of all the
\(C_i\), even when a center is a point infinitely near to an earlier
blow-up center.
\end{lemma}

\begin{proof}
The blow-up center lies outside the pairwise disjoint curves \(C_i\), so
\(\beta\) is an isomorphism in a neighborhood of each curve. Therefore the
strict transform \(C_i'\) is identified with \(C_i\) near its support, and for
every integer \(b_i\) we have
\[
L\beta^*\OO_{C_i}(b_i-1)\simeq\OO_{C_i'}(b_i-1),
\qquad
\beta^*C_i=C_i'.
\]
Set \(L_i'=\beta^*L_i\). The pullback preserves the line-bundle
presentation and the intersection number:
\[
\beta^*(L_i(C_i))=L_i'(C_i'),
\qquad
L_i'\cdot C_i'=L_i\cdot C_i=n_i+b_i-1.
\]
The divisor sequence
\[
0\longrightarrow L_i'
\longrightarrow L_i'(C_i')
\longrightarrow\OO_{C_i'}(b_i-1)
\longrightarrow0
\]
then shows that \(L\beta^*\mathcal A_i\) is generated by \(L_i'\) and
\(\OO_{C_i'}(b_i-1)\). Thus this component remains \(b_i\)-twisted adherent
to \(C_i'\). Orlov's formula makes derived pullback fully faithful, so the
pulled-back components remain semiorthogonal. A point blow-up preserves
\(H^\bullet(S,\OO_S)\cong\C\), and hence the hypotheses for the next
blow-up remain valid.

Repeating this argument proves the claims for any finite sequence of
point blow-ups. It also applies when a center is infinitely near to an
earlier center, because the hypothesis requires each center to avoid the
successive strict transforms of every \(C_i\). Finally, setting \(b_i=0\) in
the first pullback identity shows that \(\OO_{C_i}(-1)\) pulls back to
\(\OO_{C_i'}(-1)\). This proves compatibility independently of adherence.
\end{proof}

For several singularities, however, one must construct adherent subcategories
simultaneously for the different exceptional curves, and a global
obstruction appears through the divisor class group.

\begin{theorem}[{KKS obstruction, \cite[Corollary~4.8]{KKS2020}}]
\label{thm:kks-obstruction}
Let \(X\) be a normal projective rational surface with cyclic quotient
singularities, and let \(\tau:Y\to X\) be its minimal resolution with
exceptional divisor \(\mathcal D\). Suppose there is a semiorthogonal decomposition
\[
\Db(\coh Y)=\langle\widetilde{\mathcal A}_1,\ldots,
\widetilde{\mathcal A}_t\rangle,
\qquad \mathcal D=\bigsqcup_{i=1}^t\mathcal D_i,
\]
such that every nonempty \(\mathcal D_i\) is a connected component of \(\mathcal D\)
and every \(\widetilde{\mathcal A}_i\) is untwisted adherent to
\(\mathcal D_i\) in the sense of \cite[Definition~3.6]{KKS2020}.
For \(\mathcal D_i=\varnothing\), adherence means that
\(\widetilde{\mathcal A}_i\) is generated by one line bundle,
using the convention of \cite[Remark~3.8]{KKS2020}.
Then \(\Cl(X)\) is torsion-free.
\end{theorem}

\begin{proof}
KKS Corollary~3.18 gives a decomposition into derived categories of
local finite-dimensional algebras and explicitly allows empty chains.
For an empty chain the algebra is \(\C\). Additivity makes \(G_0(X)\) free,
and \(G_0(X)_{\mathrm{tors}}\cong\Cl(X)_{\mathrm{tors}}\).
This is the proof of \cite[Corollary~4.8]{KKS2020}, using
\cite[Corollary~3.18, Lemma~4.1, and Proposition~4.4]{KKS2020}.
\end{proof}

Bright's exact sequence translates this obstruction into a Brauer-group
calculation. We use the cohomological Brauer group
\(\Br(X)=H^2_{\mathrm{\acute et}}(X,\mathbb G_m)_{\mathrm{tors}}\). In
the present projective setting, its torsion classes are represented by Azumaya
algebras by Gabber's theorem, as used in
\cite[proof of Theorem~6]{Bright2013}. This is also the
Brauer group entering the obstruction in \cite[Corollary~4.8]{KKS2020}.

\begin{lemma}[Bright's exact sequence {\cite[Proposition~1 and Section~3]{Bright2013}}]\label{rem:brauer-obstruction}
Let \(X\) be a normal projective surface with rational singularities and
minimal resolution \(\tau:Y\to X\), where \(Y\) is a rational surface. Let
\(\Lambda_\tau\subset \Pic(Y)\) denote the subgroup generated by the exceptional
curves. The intersection pairing induces a homomorphism
\[
\theta:\Pic(Y)\longrightarrow \Lambda_\tau^\vee,
\qquad
D\longmapsto \bigl(E\mapsto D\cdot E\bigr),
\]
and there is an exact sequence
\[
0
\longrightarrow
\Pic(X)
\longrightarrow
\Pic(Y)
\xrightarrow{\theta}
\Lambda_\tau^\vee
\longrightarrow
\Br(X)
\longrightarrow
\Br(Y).
\]
Since \(Y\) is smooth, projective, and rational, birational invariance of
the Brauer group gives
\(\Br(Y)\simeq\Br(\Pp^2)=0\)
\cite[Corollary~6.2.11 and Theorem~6.1.3]{ColliotTheleneSkorobogatov2021}.
\end{lemma}

\begin{corollary}\label{cor:brauer-cokernel}
Let \(X\) be a normal projective surface with quotient singularities and
minimal resolution \(\tau:Y\to X\), where \(Y\) is a rational surface. For the intersection map of
\cref{rem:brauer-obstruction}, there is an isomorphism
\[
\Br(X)\cong\operatorname{coker}
\bigl(\theta:\Pic(Y)\longrightarrow\Lambda_\tau^\vee\bigr).
\]
In particular, \(\Br(X)=0\) if and only if there are divisor classes
\(L_i\in\Pic(Y)\) with \(L_i\cdot E_j=\delta_{ij}\) for all exceptional
curves \(E_j\).
\end{corollary}

\begin{proof}
Quotient singularities are rational, so \cref{rem:brauer-obstruction}
applies, and its final term vanishes. Hence
\(\Br(X)\cong\operatorname{coker}(\theta)\). Therefore \(\Br(X)=0\)
exactly when \(\theta\) is surjective, equivalently when every dual basis
vector is represented by a divisor class \(L_i\), that is,
\(L_i\cdot E_j=\delta_{ij}\).
\end{proof}

All log del Pezzo surfaces in this paper satisfy the rationality hypothesis
by \cref{prop:log-del-pezzo-rational}. In this setting, the vanishing in
\cref{cor:brauer-cokernel} is equivalent to the KKS \emph{torsion-free}
condition: the Weil divisor class group \(\Cl(X)\) has no torsion
\cite[Definition~4.7]{KKS2020}. It is necessary for an untwisted adherent
decomposition by \cite[Corollary~4.8]{KKS2020}.

Appendix~\ref{app:brauer} computes the Brauer obstruction throughout
the six Corti--Heuberger cascades.

\section{The one-point cascade and the surface \texorpdfstring{$X_{10}$}{X10}}
\label{sec:cascade}

Two constructions play distinct roles in this section. Auroux--Katzarkov--Orlov
\cite[Theorem~2.12]{AurouxKatzarkovOrlov2008} provide full strong exceptional
collections on weighted projective stacks. The canonical-stack compatibility
of \cref{lem:canonical-stack-smooth-blowup} and the point-blow-up collection
of \cref{cor:stack-point-blow-up-formula}, which follows from the
Orlov--Bergh--Schn\"urer blow-up formula, transport these collections to the
canonical stacks in \cref{thm:one-point-stack-strong}. Separately,
Karmazyn--Kuznetsov--Shinder
(KKS) give in \cite[Example~3.17]{KKS2020} a semiorthogonal decomposition of the
singular weighted projective plane, which we use as the starting point for
descent to the singular surfaces.

For a single \(\frac1n(1,1)\)-point, \cref{thm:kks-descent} descends a
decomposition on \(\F_n\) compatible with contraction of its negative
section. The adherence condition of \cref{def:kks-adherence}, together with
KKS's identification \cite[Theorem~3.16]{KKS2020}, determines the component
corresponding to the contracted curve. By \cref{lem:kks-blowups}, compatibility
and adherence persist under successive point blow-ups whose centers avoid the
strict transforms of the contracted curves, including infinitely near centers.
The Corti--Heuberger application below uses \(n=3\). We begin by recalling the
KKS decomposition of \(\Pp(1,1,n)\) in our notation.

\begin{proposition}[KKS weighted-projective-plane decomposition]
\label{prop:kks-weighted-plane}
Let \(n\geq2\), and let \(\tau:\F_n\to\Pp(1,1,n)\) contract the
negative section \(C_0\), with \(C_0^2=-n\). Writing \(f\) for the
fiber class, the full exceptional collection
\[
\langle\OO(-C_0-f),\OO(-f),\OO,\OO(C_0+nf)\rangle
\]
begins with generators of an untwisted adherent subcategory
\(\widetilde{\mathcal A}_0=
\langle\OO(-C_0-f),\OO(-f)\rangle\) over \(C_0\).
It descends to
\[
\Db(\coh\Pp(1,1,n))=
\left\langle\Db(K(n,1)\text{-mod}),\OO,\OO(n)\right\rangle,
\]
where the Kalck--Karmazyn algebra is
\[
K(n,1)\cong\C[z_1,\ldots,z_{n-1}]/(z_1,\ldots,z_{n-1})^2.
\]
\end{proposition}

\begin{proof}
The collection is the projective-bundle collection used in
\cite[Example~3.17]{KKS2020}, with their exceptional section and fiber
class denoted here by \(C_0\) and \(f\). Indeed, taking the full
exceptional collection in \eqref{eq:hirzebruch-family} with \(a=n+1\)
and tensoring by \(\OO(-f)\) gives
\(\langle\OO(-f),\OO,\OO(C_0+nf),
\OO(C_0+(n+1)f)\rangle\).

The canonical bundle is
\(\omega_{\F_n}=\OO(-2C_0-(n+2)f)\), so the Serre functor sends
\(\OO(C_0+(n+1)f)\) to \(\OO(-C_0-f)[2]\). Moving this Serre
transform cyclically to the front produces the full exceptional
collection
\[
\langle\OO(-C_0-f)[2],\OO(-f),\OO,\OO(C_0+nf)\rangle.
\]
An individual shift preserves semiorthogonality and does not change the
generated triangulated subcategory. We may therefore replace the first
object by \(\OO(-C_0-f)\).

The divisor sequence
\[
0\longrightarrow\OO(-C_0-f)\longrightarrow\OO(-f)
\longrightarrow\OO_{C_0}(-1)\longrightarrow0
\]
shows that \(\OO_{C_0}(-1)\) belongs to
\(\widetilde{\mathcal A}_0\), where \(\OO_{C_0}(-1)\) denotes the
degree-\(-1\) line bundle on \(C_0\simeq\Pp^1\). Thus the decomposition
is compatible with the contraction. Moreover,
\((-C_0-f)\cdot C_0=n-1\), which proves untwisted adherence.

The surface is rational and has rational singularities, so
\cite[Theorem~2.12(i)]{KKS2020} descends the compatible decomposition.
Untwisted adherence allows \cite[Theorem~3.16]{KKS2020} to identify its
first component with \(\Db(K(n,1)\text{-mod})\). The other two
components descend equivalently to \(\OO\) and \(\OO(n)\), because
\(\tau^*\OO(n)=\OO(C_0+nf)\). Finally,
\cite[Example~3.14(2)]{KKS2020} gives the stated presentation of
\(K(n,1)\).
\end{proof}

We now track this decomposition through successive blow-ups whose centers
avoid the successive strict transforms of \(C_0\). The centers may be
infinitely near. The derived pullback and its compatibility with contraction
are controlled by \cref{lem:kks-blowups}. Figure~\ref{fig:two-blowups-kks}
illustrates the case of two successive blow-ups.

\begin{figure}[H]
\centering
\includegraphics[width=0.82\textwidth]{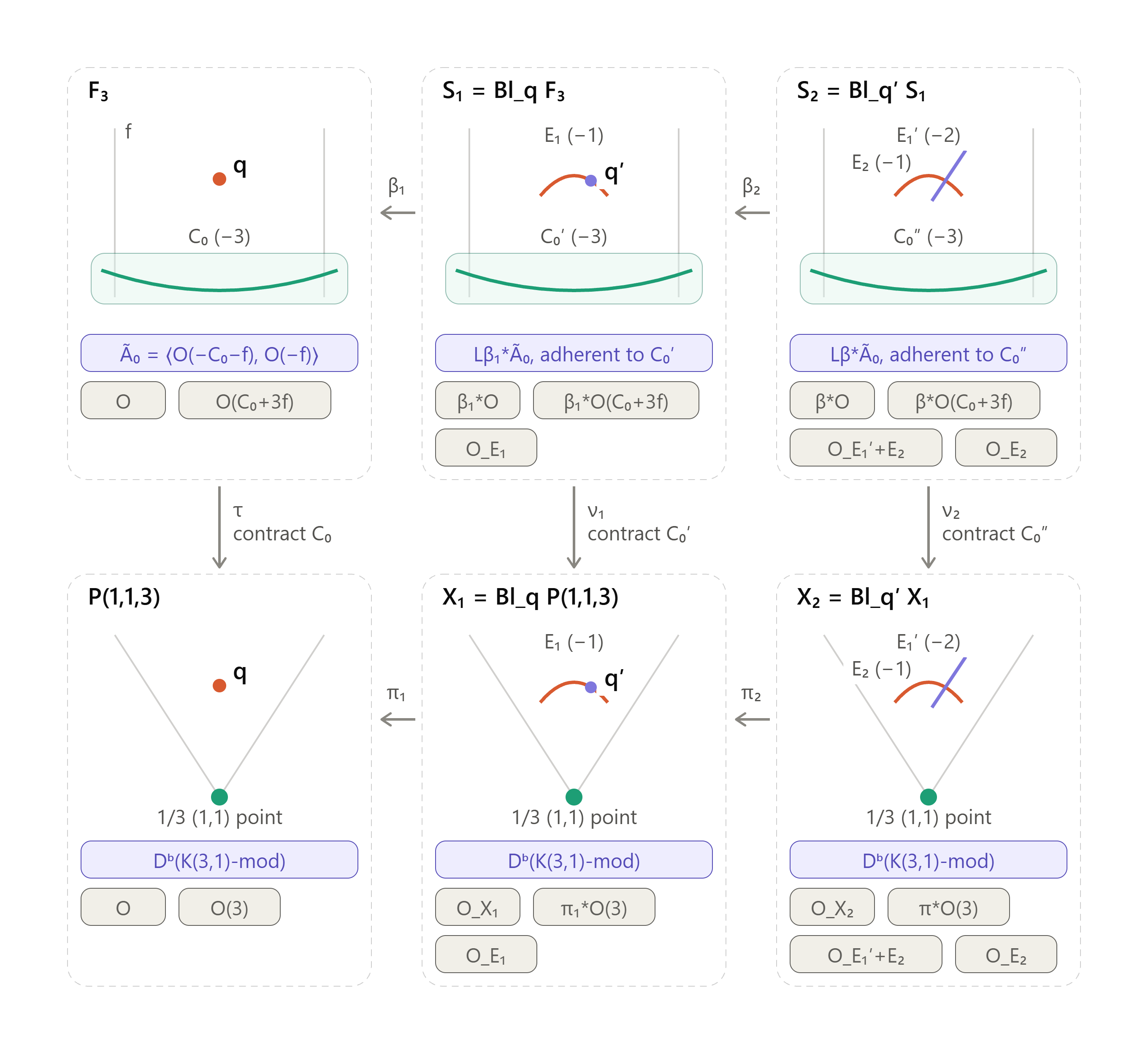}
\caption{Two successive blow-ups in the one-point case, with the second
center \(q'\) infinitely near \(q\). Both centers avoid the successive
strict transforms of \(C_0\), so \cref{lem:kks-blowups} preserves the
adherent component throughout. After the second blow-up,
\(\beta_2^*E_1=E_1'+E_2\) and
\(L\beta_2^*\mathcal O_{E_1}\simeq\mathcal O_{E_1'+E_2}\), while the new
Orlov component is generated by \(\mathcal O_{E_2}\). The lower row records
the corresponding contractions of the strict transform of \(C_0\) and the
induced KKS decompositions. This is the infinitely-near configuration
allowed in \cref{cor:cyclic-one-point-kks-sod}.}
\label{fig:two-blowups-kks}
\end{figure}

\begin{corollary}[One-point descent under blow-ups]
\label{cor:cyclic-one-point-kks-sod}
Let \(X\) be a normal projective surface with a unique singularity
of type \(\frac1n(1,1)\), where \(n\geq2\). Suppose its minimal
resolution \(\nu:S\to X\) is obtained from \(\F_n\) by \(m\)
point blow-ups whose centers avoid the negative section \(C_0\)
and its successive strict transforms, and that the exceptional locus
of \(\nu\) is the strict transform of \(C_0\).
Then
\[
\Db(\coh X)=\left\langle
\Db(K(n,1)\text{-mod}),F_1,\ldots,F_{m+2}
\right\rangle,
\]
where all \(F_i\) are exceptional. The centers may be infinitely near,
provided they lie over the smooth locus and avoid the contracted curve.
\end{corollary}

\begin{proof}
We pull back the three components of the decomposition on \(\F_n\)
in \cref{prop:kks-weighted-plane} through the successive blow-ups.
At each blow-up \(\beta\) with exceptional curve \(D\),
Orlov's decomposition, written with the exceptional object on the right, is
\[
\Db(\coh S')=\langle L\beta^*\Db(\coh S),\OO_D\rangle.
\]
By \cref{lem:kks-blowups}, the first component remains untwisted adherent
to the strict transform of \(C_0\). The other two initial components
and the \(m\) new components are generated by exceptional objects. The
new exceptional objects are understood via derived pullback from their
creation stages, as in \cref{lem:kks-blowups}. Full faithfulness preserves
exceptionality. The decomposition
is compatible with \(\nu\), whose only exceptional curve is assigned
to the first component. By \cref{thm:kks-descent}, all other components descend
equivalently, and \cite[Theorem~3.16]{KKS2020} identifies the first
with \(\Db(K(n,1)\text{-mod})\).
\end{proof}

The same argument also applies to the one-point Cavey--Prince cascades starting
from \(\Pp(1,1,n)\) whenever their minimal resolutions admit blow-up
presentations satisfying the hypotheses of
\cref{cor:cyclic-one-point-kks-sod}. For the one-point Corti--Heuberger
cascade, set
\[
P=\Pp(1,1,3)=\operatorname{Proj}\C[u,v,\xi],
\]
with
\(\deg u=\deg v=1\) and \(\deg\xi=3\). For \(X\in X_{1,d}\), \cite[Corollary~8(1)]{Corti2016} gives
\[
X\simeq\Bl_{p_1,\ldots,p_s}P,
\qquad
s=\frac{25}{3}-d,
\]
where \(p_1,\ldots,p_s\) are distinct smooth points. For \(n=3\),
\cref{prop:kks-weighted-plane} gives the KKS decomposition of \(P\).
Write \(\mathcal P=\mathcal P(1,1,3)\) for its canonical stack.

By \cref{lem:canonical-stack-smooth-blowup}, the canonical stack \(\XX\)
of \(X\) is
\[
\XX\simeq
\Bl_{p_1,\ldots,p_s}\mathcal P.
\]
Let
\[
\beta:\XX\longrightarrow\mathcal P
\]
be the blow-down, with exceptional divisors
\(\mathcal D_1,\ldots,\mathcal D_s\). Set
\[
H:=\beta^*\OO_{\mathcal P}(1)
\]
for the pullback line-bundle class. This notation distinguishes the
\(\mathcal D_i\) from the resolution curve \(E\), the resolution
blow-up curves \(D_i\), and the residual-gerbe object \(\mathcal E_p\).

Auroux--Katzarkov--Orlov \cite[Theorem~2.12]{AurouxKatzarkovOrlov2008}
give the full strong exceptional collection
\[
\left\langle
\OO_{\mathcal P},\OO_{\mathcal P}(1),\OO_{\mathcal P}(2),
\OO_{\mathcal P}(3),\OO_{\mathcal P}(4)
\right\rangle.
\]
Applying \cref{cor:stack-point-blow-up-formula} to this collection and
\(\beta\) gives a full exceptional collection on \(\XX\), strengthening
the general Ishii--Ueda construction in this one-point case. The following
theorem shows that the augmented collection remains strong.

\begin{theorem}[Full strong exceptional collections in the one-point cascade]
\label{thm:one-point-stack-strong}
Let \(X\in X_{1,d}\) be a member of the one-point
Corti--Heuberger cascade, and let \(\XX\) be its canonical stack. With
the notation above, set
\[
\begin{aligned}
\mathcal C_s=\bigl\langle
&\OO_{\XX},
\OO_{\XX}(\mathcal D_1),\ldots,
\OO_{\XX}(\mathcal D_s),\\
&\OO_{\XX}(H),\OO_{\XX}(2H),
\OO_{\XX}(3H),\OO_{\XX}(4H)
\bigr\rangle
\end{aligned}
\]
where \(s=\frac{25}{3}-d\). Then \(\mathcal C_s\) is a full strong
exceptional collection of line bundles. The line bundles
\(\OO_{\XX}(\mathcal D_i)\) are mutually orthogonal, so their internal
ordering is immaterial. Its length is
\[
s+5
=
\frac{25}{3}-d+5
=
12-d+\frac43,
\]
in agreement with the length formula in
\cref{thm:explicit-cascades-main} for \(k=1\).
\end{theorem}

\begin{proof}
The blow-up centers lie in the nonstacky locus of \(\mathcal P\).
Therefore \cref{cor:stack-point-blow-up-formula}, applied to the
Auroux--Katzarkov--Orlov collection, proves that \(\mathcal C_s\) is full
and exceptional. We only need to prove that it is strong. Since
\(\beta\) is a sequence of blow-ups at smooth points, it satisfies
\[
R\beta_*\OO_{\XX}\simeq\OO_{\mathcal P}.
\]
The projection formula then yields
\[
\RHom_{\XX}(\beta^*F,\beta^*G)
\simeq
\RHom_{\mathcal P}(F,G).
\]
The Auroux--Katzarkov--Orlov collection is strong. The preceding
isomorphism therefore shows that the pullback block
\[
\left\langle
\OO_{\XX},
\OO_{\XX}(H),
\OO_{\XX}(2H),
\OO_{\XX}(3H),
\OO_{\XX}(4H)
\right\rangle
\]
is strong. This strongness gives
\[
H^q\bigl(\XX,\OO_{\XX}(kH)\bigr)=0
\qquad
(q>0,\ 0\leq k\leq4).
\]
Reverse semiorthogonality in the weighted-projective collection also
gives
\[
H^q\bigl(\XX,\OO_{\XX}(-kH)\bigr)=0
\qquad
(q\geq0,\ 1\leq k\leq4).
\]

For each \(i\), the exceptional divisor \(\mathcal D_i\) determines the
exact sequence
\[
0\longrightarrow\OO_{\XX}
\longrightarrow\OO_{\XX}(\mathcal D_i)
\longrightarrow\OO_{\mathcal D_i}(-1)
\longrightarrow0,
\]
which compares \(\OO_{\XX}\) with the line bundle contributed by the
\(i\)-th blow-up. The curve \(\mathcal D_i\) is isomorphic to \(\Pp^1\),
so the sheaf \(\OO_{\mathcal D_i}(-1)\) is acyclic. The long exact
sequence, together with the exceptionality of \(\OO_{\XX}\), therefore
gives
\[
H^0\bigl(\XX,\OO_{\XX}(\mathcal D_i)\bigr)=\C,
\qquad
H^q\bigl(\XX,\OO_{\XX}(\mathcal D_i)\bigr)=0
\quad(q>0).
\]
These equalities give the required higher Ext vanishings from
\(\OO_{\XX}\) to \(\OO_{\XX}(\mathcal D_i)\). The exact sequence
\[
0\longrightarrow\OO_{\XX}(-\mathcal D_i)
\longrightarrow\OO_{\XX}
\longrightarrow\OO_{\mathcal D_i}
\longrightarrow0
\]
controls the Ext groups in the opposite direction. Its map on global
sections identifies constants with constants, while the higher
cohomology of the other two terms vanishes. The long exact sequence
therefore proves that \(\OO_{\XX}(-\mathcal D_i)\) is acyclic, which
means
\[
H^\bullet\bigl(\XX,\OO_{\XX}(-\mathcal D_i)\bigr)=0.
\]
Finally, the equality
\[
\Ext^q\bigl(\OO_{\XX}(\mathcal D_i),
\OO_{\XX}(\mathcal D_i)\bigr)
=H^q\bigl(\XX,\OO_{\XX}\bigr)
\]
shows that \(\OO_{\XX}(\mathcal D_i)\) is exceptional. We have now
obtained the required Ext vanishings between this line bundle and
\(\OO_{\XX}\) in both directions.

If \(i\neq j\), the distinctness of the cascade centers makes the
exceptional divisors disjoint. Their intersection therefore satisfies
\[
\mathcal D_i\cap\mathcal D_j=\varnothing.
\]
Restriction to \(\mathcal D_j\) gives the exact sequence
\[
0\longrightarrow\OO_{\XX}(-\mathcal D_i)
\longrightarrow\OO_{\XX}(\mathcal D_j-\mathcal D_i)
\longrightarrow\OO_{\mathcal D_j}(-1)
\longrightarrow0.
\]
The preceding paragraph proves that the first term is acyclic. The
isomorphism \(\mathcal D_j\simeq\Pp^1\) proves that the third term is
also acyclic. Consequently, the middle term is acyclic, which means
\[
H^\bullet\bigl(\XX,
\OO_{\XX}(\mathcal D_j-\mathcal D_i)\bigr)=0.
\]
The symmetric sequence obtained by exchanging \(i\) and \(j\) gives the
reverse vanishing. Hence the line bundles
\[
\OO_{\XX}(\mathcal D_1),\ldots,
\OO_{\XX}(\mathcal D_s)
\]
form a mutually orthogonal collection.

It remains to compare \(\OO_{\XX}(\mathcal D_i)\) with the line bundles
\(\OO_{\XX}(kH)\), where \(1\leq k\leq4\). The divisor \(kH\)
restricts trivially to \(\mathcal D_i\), so it produces the exact
sequence
\[
0\longrightarrow\OO_{\XX}(kH-\mathcal D_i)
\longrightarrow\OO_{\XX}(kH)
\longrightarrow\OO_{\mathcal D_i}
\longrightarrow0.
\]
The identifications
\[
H^0\bigl(\XX,\OO_{\XX}(kH)\bigr)
\simeq
H^0\bigl(\mathcal P,\OO_{\mathcal P}(k)\bigr),
\qquad
H^0\bigl(\mathcal D_i,\OO_{\mathcal D_i}\bigr)
\simeq
\OO_{\mathcal P}(k)|_{p_i}
\]
show that the induced map on global sections is evaluation at \(p_i\).
The point \(p_i\) lies in the smooth locus of \(P\), so at
least one of the weight-one coordinates is nonzero there. Its \(k\)-th
power defines a section of \(\OO_{\mathcal P}(k)\) that does not vanish
at \(p_i\). Therefore the evaluation map is surjective. The
higher-cohomology vanishing for \(\OO_{\XX}(kH)\) and the long exact
sequence then imply
\[
H^q\bigl(\XX,\OO_{\XX}(kH-\mathcal D_i)\bigr)=0
\qquad
(q>0).
\]
The exact sequence
\[
0\longrightarrow\OO_{\XX}(-kH)
\longrightarrow\OO_{\XX}(\mathcal D_i-kH)
\longrightarrow\OO_{\mathcal D_i}(-1)
\longrightarrow0
\]
controls the Ext groups in the opposite direction. Reverse
semiorthogonality in the weighted-projective block makes the first outer
term acyclic. The isomorphism \(\mathcal D_i\simeq\Pp^1\) makes the other
outer term acyclic. Hence the middle term is acyclic, which means
\[
H^\bullet\bigl(\XX,
\OO_{\XX}(\mathcal D_i-kH)\bigr)=0.
\]
The preceding calculations cover every pair of objects in
\(\mathcal C_s\) and prove that all their higher Ext groups vanish.
Therefore the collection \(\mathcal C_s\) is strong.
\end{proof}

\begin{corollary}\label{cor:one-point-kks-sod}
Let \(X\in X_{1,d}\). Choose a marked cascade
presentation by blowing up \(s\) distinct smooth points of \(P\), so that
\(s=25/3-d\). Then
\[
\Db(\coh X)
=
\left\langle
\Db(K(3,1)\text{-mod}),F_1,\dots,F_{s+2}
\right\rangle,
\]
where \(F_1,\dots,F_{s+2}\) are exceptional objects.
\end{corollary}

\begin{proof}
The marked cascade presentation induces the marked blow-up presentation required
by \cref{cor:cyclic-one-point-kks-sod}, with \(n=3\). Its centers
are distinct on \(\F_3\) and avoid the negative section by
\eqref{eq:cascade-distinct}. Apply that corollary.
\end{proof}

We now specialize the constructions of the one-point cascade to
quasismooth degree-10 hypersurfaces in \(\mathbb P(1,2,3,5)\). Their
blow-up presentations allow us to describe the exceptional collections
and semiorthogonal decompositions explicitly.
Consider a hypersurface of degree \(10\)
\(X_{10}\subset\Pp(1,2,3,5)\), writing \(x,y,z,w\) for
coordinates of weights \(1,2,3,5\). The singularities of a general member
are standard \cite{JohnsonKollar2001,IanoFletcher2000}. We verify them for
an arbitrary quasismooth member.

\begin{lemma}\label{prop:X10-quasismooth}
A general hypersurface of degree \(10\) in \(\Pp(1,2,3,5)\) is
quasismooth. Every quasismooth member avoids the coordinate
points of weights \(2\) and \(5\) and has a unique singular point
\(p=[0:0:1:0]\), of type \(\frac13(1,1)\).
\end{lemma}

\begin{proof}
To see that quasismooth members exist, consider
\[
f=xz^3+y^5+w^2+x^{10}.
\]
This polynomial has weighted degree \(10\), and its partial derivatives
vanish simultaneously only at the origin. Hence its affine cone is smooth
away from the origin.

Quasismoothness is an open condition in the parameter space of degree-10
hypersurfaces, so this example proves quasismoothness for a general member.

For any quasismooth member, the coefficients of
\(y^5\) and \(w^2\) are nonzero: these are the only monomials of degree 10 whose first derivatives can be nonzero on the respective
punctured coordinate axes. Otherwise the corresponding axis would lie in
the cone and be singular. Thus the hypersurface avoids the ambient
singular points of weights \(2\) and \(5\).
Every degree-10 hypersurface contains \(p=[0:0:1:0]\), since no pure
power of \(z\), which has weight three, has weighted degree ten. At \(p\),
the quotient chart has weights \((1,2,5)\equiv(1,2,2)\pmod3\).
Quasismoothness forces a nonzero coefficient of \(xz^3\), the only
monomial giving a linear term there. Eliminating \(x\) leaves tangent
weights \((2,2)\), equivalent to \((1,1)\). These are all the ambient
singular points, so quasismoothness implies smoothness elsewhere.
\end{proof}

For every member satisfying \cref{prop:X10-quasismooth}, weighted
adjunction and the intersection formula
of \cite{Dolgachev1982} give
\[
\omega_{X_{10}}^\vee=\OO_{X_{10}}(1),
\qquad
K_{X_{10}}^2=\frac13.
\]
The anticanonical class is ample, so \(X_{10}\) is a log del Pezzo surface.
The families of higher Fano index in \cite[Theorem~3]{Corti2016} have different
numerical invariants, so the Fano index is one.
Thus \(X_{10}\) lies in the index-one family \(X_{1,1/3}\).

By \cite[Corollary~8(1)]{Corti2016}, every such \(X_{10}\) is
the blow-up of \(P\) at eight smooth points, distinct by
\eqref{eq:cascade-distinct}. Its minimal resolution is therefore
the blow-up of \(\F_3\) at eight distinct points off \(C_0\). Conversely,
the Reid--Suzuki anticanonical construction
\cite[Section~1 and Table~1.1]{ReidSuzuki2004} realizes the blow-up of
\(P\) at eight general smooth points as a degree-$10$ hypersurface in
\(\Pp(1,2,3,5)\). This is the general-member model of
\(X_{1,1/3}\) recorded by Corti--Heuberger
\cite[Theorem~4, Table~2, Section~1.1.2, footnote~2]{Corti2016}. We use
the blow-up presentation to describe the minimal resolution and identify
the exceptional objects in the semiorthogonal decomposition below.

For any \(X_{10}\) satisfying the hypotheses of \cref{thm:X10-main}, fix
a marked blow-up presentation and let
\[
\sigma:\widetilde X_{10}\longrightarrow\F_3
\]
be the blow-up morphism at eight distinct points off \(C_0\). Write
\(D_1,\dots,D_8\) for its pairwise disjoint exceptional curves and
\(f\) for the fiber class on \(\F_3\). Let
\[
\nu:\widetilde X_{10}\longrightarrow X_{10}
\]
be the minimal resolution, and let \(E\) be the strict transform of \(C_0\).

Applying \cref{cor:one-point-kks-sod} to this eight-point presentation gives
a decomposition with ten exceptional objects. The following theorem identifies
those objects explicitly.

\begin{theorem}
\label{thm:X10-kks-sod}
Every \(X_{10}\) satisfying the hypotheses of \cref{thm:X10-main}
admits a semiorthogonal decomposition
\[
\Db(\coh X_{10})
=
\left\langle
\Db(K(3,1)\text{-mod}),F_1,\dots,F_{10}
\right\rangle,
\]
where \(\eta:X_{10}\to P\) is the blow-down of the chosen
presentation, \(\nu\) is the minimal resolution, and
\[
\begin{aligned}
F_1&=\OO_{X_{10}},&F_2&=\eta^*\OO_P(3),\\
F_{i+2}&=\OO_{\nu(D_i)}&& (1\leq i\leq8)
\end{aligned}
\]
are exceptional objects.
\end{theorem}

\begin{proof}
\Cref{cor:cyclic-one-point-kks-sod}, with \(n=3\) and \(m=8\), gives
this decomposition using the presentation as the blow-up of \(P\) at eight
points, which is valid for every quasismooth \(X_{10}\). Since the centers are distinct
and \(\nu\) is an isomorphism near every \(D_i\), the objects
\(\OO_{D_i}\) in Orlov's decomposition descend to \(\OO_{\nu(D_i)}\). If
\(\tau:\F_3\to P\) contracts \(C_0\), then
\(\tau\sigma=\eta\nu\) and
\(\tau^*\OO_P(3)=\OO_{\F_3}(C_0+3f)\).
The projection formula and \(R\nu_*\OO=\OO_{X_{10}}\) therefore
identify the two descended line bundles with \(F_1,F_2\).
\end{proof}

\begin{proof}[Proof of \cref{thm:X10-main}]
The singularity description is \cref{prop:X10-quasismooth}, and
Corti--Heuberger's result \cite[Corollary~8(1)]{Corti2016} gives the blow-up
presentation of the minimal resolution. Part~(i) is the specialization of
\cref{thm:one-point-stack-strong} to \(d=1/3\) and \(s=8\), with Euler
matrix given by \cref{prop:one-point-strong-euler}. Part~(ii) is
\cref{thm:X10-kks-sod}.
\end{proof}

For \(k\in\{2,4,5\}\), the Brauer obstruction vanishes, and the
idempotent-completed singularity category is known by
\cref{thm:brauer-cascades-intro,cor:cyclic-singularity-categories}.
The remaining problem is to construct, for a surface \(X\) in the
corresponding Corti--Heuberger cascade, a semiorthogonal decomposition on
its minimal resolution \(\nu:\widetilde X\to X\) with \(k\) simultaneous
untwisted adherent components, one for each exceptional \((-3)\)-curve.
Such a decomposition would allow KKS descent to produce \(k\) components
equivalent to \(D^b(K(3,1)\text{-mod})\) inside
\(D^b(\operatorname{coh}X)\), together with an exceptional remainder.

\section{Euler matrices and comparison of exceptional collections}
\label{sec:further-directions}

\Cref{thm:one-point-stack-strong} gives a full strong exceptional collection
directly on the canonical stack in the one-point cascade. In general,
\cref{prop:stack-existence} gives the Ishii--Ueda collection by applying the
embedding to a full exceptional collection on the minimal resolution. We
compare the Euler matrices of these two collections and compute the general
Ishii--Ueda matrix by separating the local singularity blocks from the Euler
matrix of the resolution collection. Since the resolution
collection may be chosen to consist of line bundles by
\cref{cor:resolution-fec}, the mixed blocks are determined by intersection
numbers with the exceptional curves. In the applications below, \(R\) is
the Euler matrix of the collection on the minimal resolution, while \(S\)
is the Euler matrix of its Ishii--Ueda collection on the canonical stack.
We write \(S_s\) for the Euler matrix of the direct full strong collection
in \cref{prop:one-point-strong-euler}.

For objects $A,B$ in these derived categories, write
\[
\chi(A,B)=\sum_q(-1)^q\dim_{\C}\Ext^q(A,B).
\]
The Euler matrix of an ordered exceptional collection
$(E_1,\ldots,E_m)$ is the matrix
$(\chi(E_i,E_j))_{1\leq i,j\leq m}$, with rows and columns in the
displayed order.

The Ishii--Ueda semiorthogonal decomposition is ordered as follows
\[
\Db(\coh\XX)=
\left\langle
\mathscr B_{p_1},\ldots,\mathscr B_{p_k},
\Phi\bigl(\Db(\coh\widetilde X)\bigr)
\right\rangle,
\]
so semiorthogonality gives
\(\RHom(\Phi A,\mathcal E_{p_i,j})=0\), which accounts for the
lower-left zero blocks in the Euler matrices. To compute the opposite
pairings \(\RHom(\mathcal E_{p_i,j},\Phi A)\), we use the left adjoint
\(\Psi\) of \(\Phi\):
\[
\RHom(\mathcal E_{p_i,j},\Phi A)
\simeq
\RHom(\Psi(\mathcal E_{p_i,j}),A).
\]
In the projective setting, \(\Phi\) has both adjoints. We now fix the
local conventions needed to compute the left adjoint.

With the convention \(\rho_i(\zeta)=\zeta^i\) and \(x,y\) of weight
\(\rho_1\), the equivariant Koszul resolution of
\(\OO_0\otimes\rho_i\) has free terms \(\OO\otimes\rho_i\),
\((\OO\otimes\rho_{i+1})^{\oplus2}\), and
\(\OO\otimes\rho_{i+2}\) in degrees \(0,-1,-2\), respectively.

On \(\mathcal U_3=[\A^2/\muthree]\), the form \(dx\wedge dy\) has
character \(\rho_2\), so
\(\omega_{\mathcal U_3}\simeq\OO_{\A^2}\otimes\rho_2\), whereas the
anticanonical bundle has character \(\rho_1\)
\cite[Section~3, equation~(17)]{GugiattiRota2023}. The local character of
\(\OO_{\mathcal X_{10}}(1)\), the anticanonical bundle in our hypersurface
example, is \(\rho_1\), not \(\rho_2\). This distinction fixes the character convention in the calculation below
and in the Euler pairings that follow.

\begin{lemma}[Local left adjoint for \(\frac1n(1,1)\)]
\label{lem:cyclic-left-adjoint}
Let \(n\geq3\), let
\(Y=\operatorname{Tot}_{\Pp^1}\OO_{\Pp^1}(-n)\), with projection
\(p:Y\to\Pp^1\) and zero section \(E\), and let \(\psi\) be the left
adjoint of the local Ishii--Ueda embedding \(\phi\). Then
\begin{equation}\label{eq:cyclic-left-adjoint}
\psi(\OO_0\otimes\rho_i)\simeq
\begin{cases}
0,&0\leq i\leq n-3,\\
\OO_E(-n)[1],&i=n-2,\\
\OO_E(-n+1),&i=n-1.
\end{cases}
\end{equation}
Here \(\OO_E(t)\) has degree \(t\) on \(E\simeq\Pp^1\).
\end{lemma}

\begin{proof}
Set \(U=[\A^2/\mu_n]\), \(P_a=p^*\OO_{\Pp^1}(a)\), and
\(\mathcal L_a=\OO_{\A^2}\otimes\rho_a\). Write
\[
\Db(\coh U)=
\left\langle\mathcal B_p,\phi(\Db(\coh Y))\right\rangle,
\qquad
\mathcal B_p=\phi(\Db(\coh Y))^\perp,
\]
and
\[
\mathcal B_p=
\left\langle
\OO_0\otimes\rho_2,\ldots,\OO_0\otimes\rho_{n-1}
\right\rangle.
\]
For compactly supported objects, local Serre duality gives
\(S_U(-)=(-)\otimes\omega_U[2]\). With our character convention,
\(\omega_U\simeq\mathcal L_2\), so
\[
{}^\perp\phi(\Db(\coh Y))=S_U^{-1}(\mathcal B_p).
\]
Tensoring the generators of \(\mathcal B_p\) by \(\omega_U^{-1}\)
shifts their character indices by \(-2\) modulo \(n\), giving the
characters \(\rho_0,\ldots,\rho_{n-3}\). The cohomological shift
\([-2]\) does not affect membership in this orthogonal subcategory.
Since \(\psi\) is left adjoint to \(\phi\), this proves
\(\psi(\OO_0\otimes\rho_i)=0\) for \(0\leq i\leq n-3\). This
vanishing range agrees with
\cite[Example~3.5, equation~(26)]{GugiattiRota2023}.

For \(j=n-2,n-1\), apply \(\psi\) to the equivariant Koszul
resolution
\[
0\longrightarrow\mathcal L_{j+2}
\xrightarrow{(-y,x)}\mathcal L_{j+1}^{\oplus2}
\xrightarrow{(x,y)}\mathcal L_j
\longrightarrow\OO_0\otimes\rho_j\longrightarrow0.
\]
Let \(s_0,s_1\) be the homogeneous coordinates on \(\Pp^1\),
pulled back to \(Y\), and let
\(t\in H^0(Y,P_{-n})\) be the section cutting out \(E\).
Write \(u=(-s_1,s_0)^{\mathsf T}\) and \(v=(s_0,s_1)\).
Reducing character indices modulo \(n\) before replacing
\(\mathcal L_a\) by \(P_{-a}\), the resulting complexes in degrees
\(-2,-1,0\) are
\[
\begin{array}{ll}
j=n-2:
 & [P_0\xrightarrow{tu}P_{1-n}^{\oplus2}
       \xrightarrow{v}P_{2-n}],\\[3pt]
j=n-1:
 & [P_{-1}\xrightarrow{u}P_0^{\oplus2}
       \xrightarrow{tv}P_{1-n}].
\end{array}
\]
Here \(t\) records multiplication across the character boundary from
\(n-1\) to \(0\), and has degree \(-n\). In the first complex,
\(\ker v=P_{-n}\), and the only cohomology is
\(P_{-n}/tP_0=\OO_E(-n)\) in degree \(-1\). In the second,
\(\operatorname{coker}u=P_1\), and the induced map
\(P_1\xrightarrow{t}P_{1-n}\) is injective with cokernel
\(\OO_E(1-n)\) in degree zero. Thus
\[
\psi(\OO_0\otimes\rho_{n-2})\simeq\OO_E(-n)[1],
\qquad
\psi(\OO_0\otimes\rho_{n-1})\simeq\OO_E(1-n),
\]
completing \eqref{eq:cyclic-left-adjoint}.
\end{proof}

\begin{proposition}[Euler matrices of the strong collections in the
one-point cascade]
\label{prop:one-point-strong-euler}
Let \(X\in X_{1,d}\), let \(\XX\) be its canonical stack, and put
\(s=\frac{25}{3}-d\). For the collection \(\mathcal C_s\) of
\cref{thm:one-point-stack-strong}, set
\[
\begin{gathered}
v=\begin{pmatrix}2&3&5&7\end{pmatrix},
\qquad
w=\begin{pmatrix}1&2&4&6\end{pmatrix},\\
A=
\begin{pmatrix}
1&2&3&5\\
0&1&2&3\\
0&0&1&2\\
0&0&0&1
\end{pmatrix}.
\end{gathered}
\]
Then the Euler matrix of \(\mathcal C_s\) is
\[
\boxed{
S_s
=
\begin{pmatrix}
1&\mathbf 1_{1\times s}&v\\
0_{s\times1}&I_s&\mathbf 1_{s\times1}w\\
0_{4\times1}&0_{4\times s}&A
\end{pmatrix}
}.
\]
\end{proposition}

\begin{proof}
The global sections of a line bundle on
\(\mathcal P(1,1,3)\) are the corresponding graded piece of
\(\C[x,y,z]\), where \(\deg x=\deg y=1\) and \(\deg z=3\). The monomial
bases in degrees \(1,2,3,4\) give
\[
h^0(\OO(1))=2,\qquad
h^0(\OO(2))=3,\qquad
h^0(\OO(3))=5,\qquad
h^0(\OO(4))=7.
\]
Indeed, the degree-three piece consists of the four cubic monomials in
\(x,y\) together with \(z\), while the degree-four piece consists of the
five quartic monomials in \(x,y\) together with \(xz,yz\). This gives
the row \(v=(2,3,5,7)\).

The divisor sequence for \(\mathcal D_i\) and the acyclicity of
\(\OO_{\mathcal D_i}(-1)\) give
\[
\chi\bigl(\OO_{\XX},\OO_{\XX}(\mathcal D_i)\bigr)=1.
\]
For \(i\neq j\), the exceptional divisors are disjoint, and the
orthogonality proved in \cref{thm:one-point-stack-strong} gives
\[
\chi\bigl(\OO_{\XX}(\mathcal D_i),
\OO_{\XX}(\mathcal D_j)\bigr)=0.
\]
Thus the exceptional-divisor block is \(I_s\).

For \(1\leq k\leq4\), sections of
\(\OO_{\XX}(kH-\mathcal D_i)\) are precisely sections of
\(\OO_{\mathcal P(1,1,3)}(k)\) vanishing at \(p_i\). Evaluation at
\(p_i\) is surjective because \(p_i\) is smooth, as in the proof of
\cref{thm:one-point-stack-strong}. Since the collection is strong,
\[
\begin{aligned}
\chi\bigl(\OO_{\XX}(\mathcal D_i),\OO_{\XX}(kH)\bigr)
&=
h^0\bigl(\OO_{\XX}(kH-\mathcal D_i)\bigr)\\
&=
h^0\bigl(\OO_{\mathcal P(1,1,3)}(k)\bigr)-1.
\end{aligned}
\]
Hence every exceptional-divisor row against the final four objects is
\(w=(1,2,4,6)\).

Finally, the entries in the last \(4\times4\) block depend only on the
difference of the twists. The section counts above give
\[
\chi\bigl(\OO_{\XX}(aH),\OO_{\XX}(bH)\bigr)
=
h^0\bigl(\OO_{\mathcal P(1,1,3)}(b-a)\bigr)
\qquad(1\leq a\leq b\leq4),
\]
which yields the displayed matrix \(A\). Since
\(\mathcal C_s\) is strong,
\[
\chi(L_i,L_j)=\dim\Hom(L_i,L_j)
\qquad(i\leq j),
\]
while the lower-triangular entries vanish by exceptionality. This proves
the block formula for \(S_s\).

For \(t\in\Z\), set
\[
\mathcal C_s(t):=\mathcal C_s\otimes\OO_{\XX}(tH).
\]
Equivalently,
\[
\begin{aligned}
\mathcal C_s(t)=\bigl\langle
&\OO_{\XX}(tH),
\OO_{\XX}(tH+\mathcal D_1),\ldots,
\OO_{\XX}(tH+\mathcal D_s),\\
&\OO_{\XX}((t+1)H),\ldots,
\OO_{\XX}((t+4)H)
\bigr\rangle.
\end{aligned}
\]
Tensoring by a line bundle is an autoequivalence, so every
\(\mathcal C_s(t)\) is again a full strong exceptional collection. More
generally, for every line bundle \(L\),
\(\chi(E\otimes L,F\otimes L)=\chi(E,F)\), so any common twist preserves
the Euler matrix. In particular,
\[
S_{\mathcal C_s(t)}=S_s
\qquad(t\in\Z).\qedhere
\]
\end{proof}

We now pass from the strong collections in the one-point cascade to the Euler
matrices of the Ishii--Ueda collections for surfaces with several cyclic
quotient singularities. The resolution collection determines one block, while
the residual-gerbe objects contribute the local blocks at the singular points.
When the resolution collection consists of line bundles, the mixed pairings
are determined by their intersections with the exceptional curves.

\begin{proposition}[Euler matrix for a canonical stack with singularities of type \(\frac1{n_i}(1,1)\)]
\label{prop:mixed-basket-euler}
Let \(X\) be a normal projective rational surface with singular points \(p_1,\ldots,p_k\) of
types \(\frac1{n_i}(1,1)\), where \(n_i\geq2\), and let
\(E_i\subset\widetilde X\) be the exceptional curve over \(p_i\).
Fix a full exceptional collection \(\langle G_1,\ldots,G_N\rangle\)
on the minimal resolution. For the associated full exceptional collection on the canonical stack of
\cref{prop:stack-existence}, the Euler matrix is
\[
S=
\begin{pmatrix}
T_{n_1}&0&\cdots&0&B_1\\
0&T_{n_2}&\cdots&0&B_2\\
\vdots&\vdots&\ddots&\vdots&\vdots\\
0&0&\cdots&T_{n_k}&B_k\\
0&0&\cdots&0&R
\end{pmatrix},
\qquad R_{ab}=\chi(G_a,G_b).
\]
For an arbitrary exceptional collection on the resolution, the off-diagonal
blocks \(B_i\) are defined by \((B_i)_{j,a}=\chi(\mathcal E_{p_i,j},\Phi(G_a))\),
where \(2\leq j\leq n_i-1\). The local block is
\begin{equation}\label{eq:Tn}
T_n=(I_{n-2}-J_{n-2})^2,
\qquad (J_{n-2})_{ab}=\delta_{b,a+1}\quad(1\leq a,b\leq n-2),
\end{equation}
where \(J_{n-2}\) is the nilpotent shift with ones immediately above
the diagonal. For \(n=2\), the block is empty.

When the resolution collection consists of line bundles
\(G_a=\OO(A_a)\), these pairings are determined by intersection numbers
alone. Put
\(d_{ia}=A_a\cdot E_i\). For \(n_i\geq4\),
\[
(B_i)_{j,a}=
\begin{cases}
0,&2\leq j\leq n_i-3,\\
d_{ia}+1,&j=n_i-2,\\
-d_{ia},&j=n_i-1.
\end{cases}
\]
For \(n_i=3\), the sole row is \((B_i)_{2,a}=-d_{ia}\), while for
\(n_i=2\), \(B_i\) has no rows. Thus the resolution matrix \(R\)
and the intersection numbers \(A_a\cdot E_i\) determine the entire
Euler matrix for a collection of line bundles on the resolution.
\end{proposition}

\begin{proof}
The equivariant Koszul resolution on
\([\A^2/\boldsymbol\mu_n]\) gives
\[
\chi(\OO_0\otimes\rho_u,\OO_0\otimes\rho_v)
=\delta_{uv}-2\delta_{v,u+1}+\delta_{v,u+2}.
\]
Within \(2\leq u,v\leq n-1\), no cyclic wraparound term occurs,
so this is \(T_n\). Distinct local collections have disjoint support,
which gives the off-diagonal zeros among them. Full faithfulness of
\(\Phi\) gives
\[
\chi(\Phi(G_a),\Phi(G_b))=\chi(G_a,G_b),
\]
so the lower-right block is exactly \(R\). \Cref{thm:ishii-ueda-global}
supplies the lower-left zero blocks.

Let \(\Psi\) be the left adjoint of \(\Phi\). Left adjunction gives
\[
\RHom_{\XX}(\mathcal E_{p_i,j},\Phi(G_a))
\simeq
\RHom_{\widetilde X}(\Psi(\mathcal E_{p_i,j}),G_a).
\]
\Cref{lem:cyclic-left-adjoint}, equivalently
\eqref{eq:cyclic-left-adjoint}, computes
\(\Psi(\mathcal E_{p_i,j})\) locally near \(E_i\). For the indices
belonging to the non-special collection, it gives
\[
\Psi(\mathcal E_{p_i,n_i-2})=\OO_{E_i}(-n_i)[1],
\qquad
\Psi(\mathcal E_{p_i,n_i-1})=\OO_{E_i}(-n_i+1),
\]
with all earlier non-special characters sent to zero. The first
displayed source belongs to the local collection only when \(n_i\geq4\),
and the second only when \(n_i\geq3\).
For the divisor inclusion \(E_i\hookrightarrow\widetilde X\),
duality and \(E_i^2=-n_i\) give
\[
\chi_{\widetilde X}(\OO_{E_i}(t),\OO(A_a))
=-\chi_{\Pp^1}\bigl(\OO(d_{ia}-n_i-t)\bigr).
\]
Together with \(\chi(\OO_{\Pp^1}(q))=q+1\), this gives \(d_{ia}+1\) for the
shifted penultimate image and \(-d_{ia}\) for the final image.
\end{proof}

\begin{remark}[Local checks and orthogonality]
\label{rem:local-euler-check}
Fix the singular point \(p\). For a vector bundle \(V\) on its local
chart \([\A^2/\mu_n]\), gerbe adjunction
and the ring-action convention of \cref{sec:canonical-stack} give
\[
\RHom(\OO_0\otimes\rho_j,V)
\simeq (V_0\otimes\rho_{-j-2})^{\mu_n}[-2].
\]
Its Euler characteristic is the multiplicity of \(\rho_{j+2}\) in
the fiber \(V_0\). In particular,
\(\chi(\mathcal E_{p,j},\mathcal L_a)=\delta_{a,j+2}\), with
indices read modulo \(n\).

For the pairing with the image of \(\OO(f)\), recall from
\cref{rem:ishii-ueda-derived-images} that
\(\phi(P_1)\simeq(x,y)\otimes\rho_{n-1}\). The coordinate maps of the
universal cluster give, for
\(n\geq3\),
\[
0\longrightarrow\phi(P_1)\longrightarrow\mathcal L_{n-1}
\longrightarrow \mathcal E_{p,n-1}\longrightarrow0.
\]
Hence \(\chi(\mathcal E_{p,n-1},\phi(P_1))=-1\).
This verifies the Euler pairing with the image of \(\OO(f)\), including the
\(-1\) in the Ishii--Ueda Euler matrix for \(X_{10}\).

For \(n\geq5\), the objects \(\mathcal E_{p,j}\) with
\(2\leq j\leq n-3\) are orthogonal to the resolution image on both
sides: one direction is semiorthogonality, and the other follows from
\(\psi(\mathcal E_{p,j})=0\). This does not split off an orthogonal
summand of the whole category, since
\(\chi(\mathcal E_{p,j},\mathcal E_{p,j+1})=-2\) and
\(\chi(\mathcal E_{p,j},\mathcal E_{p,j+2})=1\).
\end{remark}

The resolution contribution of a line-bundle collection is determined by
the intersections of its members with the exceptional curve. For a single
\(\frac1n(1,1)\)-point, let \(n\geq2\) and \(s\geq0\), and suppose
a normal projective surface \(X\) has a minimal resolution fitting into a
diagram
\[
\widetilde X\xrightarrow{\ \sigma\ }\F_n,
\qquad
\widetilde X\xrightarrow{\ \nu\ }X,
\]
where \(\sigma\) is a sequence of \(s\) point blow-ups, possibly
infinitely near, whose centers avoid the successive strict transforms of
the negative section \(C_0\subset\F_n\), and \(\nu\) contracts its
strict transform \(E\) to the unique singular point of \(X\), of type
\(\frac1n(1,1)\). Here \(s\) is the number of point blow-ups, as in
Section~\ref{sec:cascade-foundations}. In the one-point Corti--Heuberger
cascade, \(n=3\) and \(s=\frac{25}{3}-d\). Let
\(D_a\subset Y_a\) be the exceptional divisor
created at the \(a\)-th stage, let \(\pi_a:\widetilde X\to Y_a\) be the
natural morphism, and set
\[
\Delta_a:=\pi_a^*D_a,
\qquad
M_a:=\OO_{\widetilde X}(\Delta_a).
\]
Thus \(\Delta_a\) is a total transform, not necessarily the final strict
transform when later centers are infinitely near. We use
\[
C_0^2=-n,
\qquad
C_0\cdot f=1,
\qquad
f^2=0.
\]
For the local objects at the unique singular point, retain the notation
\(\mathcal E_{p,i}\) of \cref{def:local-cyclic-blocks}, together with
the notation \(\XX,\Phi\) of \cref{prop:mixed-basket-euler}.

For \(c\in\Z\), take \(a=c+n+1\) in
\eqref{eq:hirzebruch-family}. The blow-up construction for collections of line bundles and the
Ishii--Ueda decomposition give the ordered full exceptional
collection
\[
\begin{aligned}
\mathcal C_{n,s}(c)=\big\langle
&\mathcal E_{p,2},\ldots,\mathcal E_{p,n-1},\\
&\Phi(\OO_{\widetilde X}),
\Phi(M_s),\ldots,\Phi(M_1),\\
&\Phi(\sigma^*\OO(f)),
\Phi(\sigma^*\OO(C_0+(c+n+1)f)),\\
&\Phi(\sigma^*\OO(C_0+(c+n+2)f))
\big\rangle .
\end{aligned}
\]

Set
\[
u_n(c)=
\begin{pmatrix}2&2c+n+4&2c+n+6\end{pmatrix},
\qquad
v_n(c)=
\begin{pmatrix}1&2c+n+3&2c+n+5\end{pmatrix},
\]
and
\[
K_n(c)=
\begin{pmatrix}
1&2c+n+2&2c+n+4\\
0&1&2\\
0&0&1
\end{pmatrix}.
\]
The Euler matrix of the ordered collection of line bundles on
\(\widetilde X\) is
\begin{equation}\label{eq:Hns}
R_{n,s}(c)=
\begin{pmatrix}
1&\mathbf1_{1\times s}&u_n(c)\\
0&I_s&\mathbf1_{s\times1}v_n(c)\\
0&0&K_n(c)
\end{pmatrix}.
\end{equation}

For the local-to-resolution block, the intersections with \(E\), in the
same order, are
\[
d=\left(\underbrace{0,\ldots,0}_{s+1},1,c+1,c+2\right).
\]
For \(n\geq4\), define
\begin{equation}\label{eq:Bns}
B_{n,s}(c)=
\begin{pmatrix}
0_{(n-4)\times(s+4)}\\[2mm]
\mathbf1_{1\times(s+1)}&2&c+2&c+3\\
0_{1\times(s+1)}&-1&-(c+1)&-(c+2)
\end{pmatrix},
\end{equation}
where the first block row is absent for \(n=4\). For \(n=3\), set
\[
B_{3,s}(c)=
\begin{pmatrix}
0_{1\times(s+1)}&-1&-(c+1)&-(c+2)
\end{pmatrix}.
\]
For \(n=2\), both the local collection and the local block are empty.

\begin{proposition}[Euler matrix in the one-point case]
\label{prop:general-one-point-euler}
For the collection \(\mathcal C_{n,s}(c)\), the
\((n+s+2)\times(n+s+2)\) matrix of the Euler form is
\[
\boxed{
S_{n,s}(c)=
\begin{pmatrix}
T_n&B_{n,s}(c)\\
0&R_{n,s}(c)
\end{pmatrix}
}.
\]
\end{proposition}

\begin{proof}
The blow-up centers avoid the successive strict transforms of \(C_0\),
so \(\Delta_a\cdot E=0\) for every \(a\). Pullback from \(\F_n\)
gives the remaining intersections
\[
\sigma^*f\cdot E=1,
\qquad
\sigma^*(C_0+(c+n+1)f)\cdot E=c+1,
\qquad
\sigma^*(C_0+(c+n+2)f)\cdot E=c+2.
\]
Thus the intersection vector for the resolution collection is precisely
\(d\). Applying \cref{prop:mixed-basket-euler}, the row indexed by
\(\mathcal E_{p,n-2}\) is \(d+\mathbf1\), while the row indexed by
\(\mathcal E_{p,n-1}\) is \(-d\). The earlier rows vanish. This gives
\eqref{eq:Bns} for \(n\geq4\). When \(n=3\), only the row \(-d\)
belongs to the Ishii--Ueda complement, and when \(n=2\) there is no
local block.

It remains to compute the resolution matrix. Let
\(G_a=L\pi_a^*\OO_{D_a}(-1)\) be the exceptional object contributed
by Orlov's blow-up formula. The divisor sequence gives
\[
[M_a]=[\OO_{\widetilde X}]+[G_a]
\qquad\text{in }K_0(\widetilde X).
\]
Orlov semiorthogonality and divisor duality give
\[
\chi(G_a,G_b)=\delta_{ab},
\qquad
\chi(G_a,\sigma^*L)=-1,
\qquad
\chi(\sigma^*L,G_a)=0
\]
for every line bundle \(L\) on \(\F_n\), including when the centers are
infinitely near. Hence
\[
\chi(\OO,M_a)=1,
\qquad
\chi(M_a,\OO)=0,
\qquad
\chi(M_a,M_b)=\delta_{ab},
\]
and
\[
\chi(M_a,\sigma^*L)=\chi(\OO,\sigma^*L)-1,
\qquad
\chi(\sigma^*L,M_a)=0.
\]
These identities give the first two block rows of
\eqref{eq:Hns}. Finally, Riemann--Roch on \(\F_n\), using
\[
K_{\F_n}=-2C_0-(n+2)f,
\qquad
\chi(\OO(A),\OO(B))
=1+\frac12(B-A)\cdot(B-A-K_{\F_n}),
\]
gives the row \(u_n(c)\), the row \(v_n(c)\), and the block
\(K_n(c)\). This proves \eqref{eq:Hns} and the proposition.
\end{proof}

This \(K_0\)-relation gives the unimodular change of basis from the Orlov
collection to the collection of line bundles. The corresponding Euler
matrices are related by congruence.

For \(n=3\), \(s=8\), and \(c=-1\), one has \(T_3=(1)\) and
\[
B_{3,8}(-1)=
\begin{pmatrix}0_{1\times9}&-1&0&-1\end{pmatrix}.
\]
Moreover,
\[
R_{3,8}(-1)=
\begin{pmatrix}
1&\mathbf1_{1\times8}&2&5&7\\
0_{8\times1}&I_8&\mathbf1_{8\times1}&4\mathbf1_{8\times1}&6\mathbf1_{8\times1}\\
0&0_{1\times8}&1&3&5\\
0&0_{1\times8}&0&1&2\\
0&0_{1\times8}&0&0&1
\end{pmatrix}.
\]
For \(X_{10}\), the matrix \(R_{3,8}(-1)\) is the Euler matrix of the
collection of line bundles on the minimal resolution. The matrix
\(S_{3,8}(-1)\) is the Euler matrix after applying the Ishii--Ueda
construction. It is the \(13\times13\) matrix
\[
S_{3,8}(-1)=
\begin{pmatrix}
1&B_{3,8}(-1)\\
0&R_{3,8}(-1)
\end{pmatrix}.
\]
Proposition~\ref{prop:one-point-strong-euler} computes \(S_8\) from the
full strong collection of line bundles directly on the canonical stack.
For \(X_{10}\), one has \(s=8\), so the collection is \(\mathcal C_8\)
and its Euler matrix is \(S_{X_{10}}=S_8\).

To compare the collections themselves, specialize the Ishii--Ueda
collection to \(n=3\), \(s=8\), and \(c=-1\):
\[
\begin{aligned}
\mathcal C_{3,8}(-1)=\bigl\langle
&\mathcal E_p,\,\Phi(\OO_{\widetilde X}),
\Phi(M_8),\ldots,\Phi(M_1),\\
&\Phi(\sigma^*\OO(f)),
\Phi(\sigma^*\OO(C_0+3f)),
\Phi(\sigma^*\OO(C_0+4f))
\bigr\rangle.
\end{aligned}
\]
The exceptional divisors created by the eight smooth blow-ups avoid
the contracted \((-3)\)-curve, so they descend to Cartier divisors on
the singular surface. The compatibility \eqref{eq:ishii-ueda-projection-formula}
and \(\Phi(\OO_Y)\simeq\OO_{\mathcal X}\) then give
\(\Phi(\OO_{\widetilde X})\simeq\OO_{\mathcal X_{10}}\), and, after
matching the exceptional divisors,
\(\Phi(M_i)\simeq\OO_{\mathcal X_{10}}(D_i)\). Moreover,
\[
\tau^*\OO_{\Pp(1,1,3)}(3)
\simeq\OO_{\F_3}(C_0+3f),
\]
so the \(X\)-linearity of \(\Phi\) gives
\[
\Phi(\sigma^*\OO(C_0+3f))
\simeq\OO_{\mathcal X_{10}}(3H).
\]
Since the \(D_i\) are mutually orthogonal, their order can be rearranged.
Thus \(\mathcal C_{3,8}(-1)\) and
\[
\mathcal C_8=\left\langle
\OO_{\mathcal X_{10}},
\OO_{\mathcal X_{10}}(D_1),\ldots,\OO_{\mathcal X_{10}}(D_8),
\OO_{\mathcal X_{10}}(H),\OO_{\mathcal X_{10}}(2H),
\OO_{\mathcal X_{10}}(3H),\OO_{\mathcal X_{10}}(4H)
\right\rangle
\]
share ten of their thirteen objects
\[
\OO_{\mathcal X_{10}},\quad
\OO_{\mathcal X_{10}}(D_1),\ldots,\OO_{\mathcal X_{10}}(D_8),\quad
\OO_{\mathcal X_{10}}(3H).
\]
Their Euler pairings agree on this common subcollection. The remaining
objects are \(\mathcal E_p\), \(\Phi(\sigma^*\OO(f))\), and
\(\Phi(\sigma^*\OO(C_0+4f))\) in the Ishii--Ueda collection, and
\(\OO_{\mathcal X_{10}}(H)\), \(\OO_{\mathcal X_{10}}(2H)\), and
\(\OO_{\mathcal X_{10}}(4H)\) in the strong collection. Understanding
the relation between these two three-object complements gives a concrete
mutation problem on \(X_{10}\).


\appendix
\section{Brauer groups and obstructions in the Corti--Heuberger cascades}
\label{app:brauer}

The proof of \cref{thm:brauer-cascades-intro} combines the cascade description of
\cref{sec:cascade-foundations}, the cokernel criterion of \cref{cor:brauer-cokernel},
and invariance under blow-ups at smooth points (\cref{lem:brauer-cascade-invariance}).

We first describe the lattice mechanism for a normal projective rational surface
\(X\) with singular points \(p_1,\ldots,p_k\) of types
\(\frac1{n_i}(1,1)\), where \(n_i\geq2\). Its minimal resolution
has pairwise disjoint exceptional curves \(E_1,\ldots,E_k\) with
\(E_i^2=-n_i\). Put \(\Lambda=\bigoplus_i\Z E_i\). In the notation of
\cref{rem:brauer-obstruction}, the intersection map is
\[
\theta:\Pic(\widetilde X)\longrightarrow\Lambda^\vee,
\qquad
D\longmapsto(D\cdot E_1,\dots,D\cdot E_k).
\]
Here \(\Lambda^\vee=\Hom(\Lambda,\Z)\) has basis \(e_1,\dots,e_k\)
dual to the \(E_i\). In particular, \(\theta(E_i)=-n_i e_i\).
By \cref{cor:brauer-cokernel}, \(\Br(X)\cong\operatorname{coker}(\theta)\).

Bright's exact sequence applies to arbitrary rational surface singularities
\cite[Proposition~1]{Bright2013}, whereas the classification carried out in
\cite[Section~3]{Bright2013} concerns the classical singular del Pezzo case
obtained by contracting \((-2)\)-curves. In the Corti--Heuberger cascades the
exceptional curves over the singular points have self-intersection \(-3\).

\subsection{Saturation of the exceptional lattice}
\label{app:brauer-glue}

The intersection lattice \(P=\Pic(\widetilde X)\) is free and unimodular.
Indeed, the rational-surface description used in
\cref{cor:rational-surface-fec} reduces this to \(\Pp^2\) and the
Hirzebruch surfaces, whose intersection matrices have determinant \(\pm1\).
Each point blow-up adds an orthogonal rank-one summand of square \(-1\).
In particular, intersection identifies \(P\) with \(P^\vee\).

\begin{definition}\label{def:brauer-glue}
The \emph{saturation} of the exceptional lattice in \(P\) is
\[
\Lambda^{\mathrm{sat}}
=
\bigl(\Lambda\otimes\mathbb Q\bigr)\cap\Pic(\widetilde X).
\]
The lattice is \emph{primitive} if \(P/\Lambda\) is torsion-free,
equivalently if \(\Lambda^{\mathrm{sat}}=\Lambda\). The finite quotient
\(\Lambda^{\mathrm{sat}}/\Lambda\) is its \emph{saturation quotient}.
Using the intersection form to identify \(\Lambda^\vee\) with a lattice
in \(\Lambda\otimes\mathbb Q\), the \emph{discriminant group}
\[
A_\Lambda:=\Lambda^\vee/\Lambda\cong\bigoplus_{i=1}^k\Z/n_i
\]
carries the induced discriminant bilinear form
\[
b_\Lambda:A_\Lambda\times A_\Lambda\longrightarrow\mathbb Q/\Z,
\qquad
b_\Lambda(\bar x,\bar y)=x\cdot y\pmod\Z.
\]
\end{definition}

An \emph{integral overlattice} of \(\Lambda\) is a finite-index extension
inside \(\Lambda\otimes\mathbb Q\) whose pairings are integral. It lies
in \(\Lambda^\vee\), and its quotient by \(\Lambda\) is totally
isotropic for \(b_\Lambda\), meaning that every pairing between elements
of the quotient is zero. Conversely, the inverse image of such a subgroup
is an integral overlattice. This is the bilinear version of the
overlattice correspondence discussed in \cite[Section~1.4]{Nikulin1980}.
In particular, \(H_\Lambda:=\Lambda^{\mathrm{sat}}/\Lambda\subset A_\Lambda\) is totally isotropic,
since its representatives belong to the integral lattice \(P\).

The next lemma identifies the Brauer group canonically with the dual of
the saturation quotient. Nontrivial \(\Br(X)\), equivalently torsion in
\(\Cl(X)\), is the Brauer obstruction to an untwisted adherent KKS
decomposition: by \cref{thm:kks-obstruction}, such a decomposition requires
\(\Cl(X)\) to be torsion-free.

\begin{lemma}[The Brauer group and the saturation quotient]\label{lem:brauer-glue-duality}
There are natural isomorphisms
\[
\begin{aligned}
\Br(X)&\cong\operatorname{coker}(\theta)
\cong\Hom\bigl(\Lambda^{\mathrm{sat}}/\Lambda,\mathbb Q/\Z\bigr)\\
&\cong\Hom\bigl(\Cl(X)_{\mathrm{tors}},\mathbb Q/\Z\bigr).
\end{aligned}
\]
Thus \(\Br(X)=0\) if and only if \(\Lambda\) is primitive in \(P\).
As abstract finite abelian groups, \(\Br(X)\) and
\(\Lambda^{\mathrm{sat}}/\Lambda\) are isomorphic, but this last
identification is not canonical.
\end{lemma}

\begin{proof}
Applying \(\Hom(-,\Z)\) to \(0\to\Lambda\to P\to P/\Lambda\to0\)
gives an exact sequence
\[
P^\vee\longrightarrow\Lambda^\vee
\longrightarrow\Ext^1(P/\Lambda,\Z)\longrightarrow0,
\]
because \(P\) is free. Under the unimodular identification
\(P\cong P^\vee\), the first map is \(\theta\). For any finitely
generated abelian group \(G\), the sequence
\(0\to\Z\to\mathbb Q\to\mathbb Q/\Z\to0\) gives
\(\Ext^1(G,\Z)\cong\Hom(G_{\mathrm{tors}},\mathbb Q/\Z)\).
Pushforward of divisors identifies \(P/\Lambda\) with \(\Cl(X)\):
every prime divisor on \(X\) has a strict transform on the smooth
surface \(\widetilde X\), and the kernel is generated by the contracted
curves. Hence
\((P/\Lambda)_{\mathrm{tors}}=\Lambda^{\mathrm{sat}}/\Lambda
\cong\Cl(X)_{\mathrm{tors}}\).
The result follows from \cref{cor:brauer-cokernel}. A finite abelian
group is noncanonically isomorphic to its character group.
\end{proof}

The Gram matrix \(\operatorname{diag}(-n_1,\ldots,-n_k)\) gives
\(\Lambda^\vee=\bigoplus_i\Z(E_i/n_i)\). Every class in the saturation quotient thus
has a representative
\[
L=\sum_{i=1}^k\frac{a_i}{n_i}E_i\in\Pic(\widetilde X),
\qquad \bar a_i\in\Z/n_i,
\]
where the \(a_i\) are integer lifts of the residue classes. Since
the intersection form is integral and adjunction gives
\(K_{\widetilde X}\cdot E_i=n_i-2\), every such representative
satisfies
\begin{equation}\label{eq:mixed-glue-integrality}
L^2=-\sum_i\frac{a_i^2}{n_i}\in\Z,
\qquad
K_{\widetilde X}\cdot L=
\sum_i\frac{a_i(n_i-2)}{n_i}\in\Z.
\end{equation}
These are necessary numerical constraints on classes in the saturation quotient, not a
sufficiency criterion for their realization in \(\Pic(\widetilde X)\).

For the remainder of this appendix, specialize to \(k\) points all
of type \(\frac13(1,1)\), all \(n_i=3\), and
\(A_\Lambda\cong(\Z/3)^k\). Corti--Heuberger's defect is
\(\delta_{\mathrm{CH}}=\dim_{\mathbb F_3}(\Lambda^{\mathrm{sat}}/\Lambda)\).
They determine it in \cite[Section~2.1, Proposition~18(b), and
Lemma~28]{Corti2016}: \(\delta_{\mathrm{CH}}=1\) in the families
\(X_{6,2}\) and \(X_{6,1}\), and \(\delta_{\mathrm{CH}}=0\) otherwise.
Together with
\cref{lem:brauer-glue-duality}, they already give the group orders
in \cref{thm:brauer-cascades-intro}. The calculations below exhibit
the intersection maps and classes in the saturation quotient realizing these values.
The next lemma specializes \eqref{eq:mixed-glue-integrality} and
proves primitivity for one or two exceptional curves without choosing a blow-up presentation.

\begin{lemma}\label{lem:brauer-glue-vectors}
Every class in
\(\Lambda^{\mathrm{sat}}/\Lambda\) can be represented in the form
\[
L=\frac13\sum_{i=1}^k a_iE_i,
\qquad
a_i\in\{0,\pm1\}.
\]
Such a representative satisfies the congruences
\[
\sum_{i=1}^k a_i^2\equiv0\pmod3,
\qquad
\sum_{i=1}^k a_i\equiv0\pmod3.
\]
Consequently:
\begin{enumerate}[label=\textup{(\roman*)}]
\item if \(k=1\) or \(k=2\), the lattice \(\Lambda\) is primitive.
\item if \(3\leq k\leq5\), then either
\(\Lambda^{\mathrm{sat}}/\Lambda=0\), or there is a single triple
\(\{a,b,c\}\) such that
\[
\Lambda^{\mathrm{sat}}/\Lambda=
\left\{
0,\,
\frac13(E_a+E_b+E_c),\,
-\frac13(E_a+E_b+E_c)
\right\}.
\]
These are classes in the saturation quotient. This does not assert that
every numerically admissible triple is realized.
\item if \(k=6\), classes supported on three or on all six exceptional curves
are numerically possible. On \(Z_6\), the saturation quotient is generated by
the class with alternating coefficients on all six exceptional curves, so no
class supported on three exceptional curves is realized. See
\cref{app:brauer-six}.
\end{enumerate}
\end{lemma}

\begin{proof}
For \(n_i=3\), choose the residue representatives
\(a_i\in\{0,\pm1\}\). Formula \eqref{eq:mixed-glue-integrality}
becomes
\[
-\frac13\sum_i a_i^2\in\Z,
\qquad \frac13\sum_i a_i\in\Z,
\]
which gives the two stated congruences.

The first congruence says that the number of nonzero coefficients is divisible
by \(3\). This proves (i). If \(3\leq k\leq5\), a nonzero class has three
nonzero coefficients, and the second congruence forces them to have the same
sign. By the cited overlattice correspondence, the isotropic subgroup is
totally isotropic for \(b_\Lambda\). Choosing representatives with positive
coefficients on their supports, two classes supported on distinct triples
\(A\) and \(B\) have pairing
\[
-\frac{\#(A\cap B)}{3}\pmod{\Z}.
\]
For two distinct triples in a set of at most five elements, the intersection
has size \(1\) or \(2\), so the two classes cannot be orthogonal. Thus
all nonzero classes, if any, have the same support. Their only possible
values are the two opposite-sign classes on that triple, so they generate
the same cyclic subgroup and give the quotient in (ii). For
\(k=6\), coefficient vectors with three or six nonzero entries satisfy
the first numerical condition, subject also to the stated congruence on the signs. In particular,
the alternating coefficient vector satisfies both congruences.
Thus the numerical conditions permit the coefficient pattern
\(\frac13(E_1-E_2+E_3-E_4+E_5-E_6)\), as asserted in (iii).
\end{proof}

\begin{corollary}\label{cor:brauer-one-two}
Let \(X\) be a log del Pezzo surface with at most two singular points,
all of type \(\frac13(1,1)\). Its Brauer group is trivial. In particular,
\(\Br(Z_1)=\Br(\Pp(1,1,3))=0\) and \(\Br(Z_2)=0\).
\end{corollary}

The bound of two singularities in \cref{cor:brauer-one-two} follows from
the \((-3)\)-lattice congruences. For \(3\leq k\leq5\),
\cref{lem:brauer-glue-duality,lem:brauer-glue-vectors} leave only
\(0\) and \(\Z/3\) as possibilities. The intersection maps computed
below distinguish between them.

\subsection{Behavior under cascade blow-ups}
\label{app:brauer-invariance}

The morphisms in a Corti--Heuberger cascade are blow-ups at smooth points by
\cite[Corollary~8]{Corti2016}. Their effect on the intersection map
is elementary.

\begin{lemma}\label{lem:brauer-cascade-invariance}
Let \(X\) be a normal projective surface with quotient singularities and
minimal resolution \(Y\to X\), where \(Y\) is a rational surface, and let \(X'\to X\) be the blow-up of a
smooth point. This blow-up preserves the Brauer group:
\[
\Br(X')\cong\Br(X).
\]
The same conclusion holds for a finite sequence of blow-ups at smooth points,
including infinitely near centers.
\end{lemma}

\begin{proof}
On minimal resolutions, the induced morphism
\[
\beta:\widetilde X'\longrightarrow\widetilde X
\]
is the blow-up of a point disjoint from the exceptional curves
\(E_1,\dots,E_k\) of \(\widetilde X\to X\), since the resolution is an
isomorphism over the smooth locus. The strict
transforms of these curves give natural identifications
\[
\Lambda_{X'}\cong\Lambda_X,
\qquad
\Lambda_{X'}^\vee\cong\Lambda_X^\vee.
\]
If \(D\) is the new exceptional curve, the Picard group decomposes as
\[
\Pic(\widetilde X')
=
\beta^*\Pic(\widetilde X)\oplus\Z D,
\qquad
D\cdot E_i=0.
\]
This decomposition implies that the intersection numbers satisfy
\[
(\beta^*L+nD)\cdot E_i=L\cdot E_i
\]
for every \(L\in\Pic(\widetilde X)\) and every \(i\). Under the displayed identification of the dual exceptional
lattices, the image and cokernel of \(\theta\) are therefore unchanged.
The Brauer groups are these cokernels by \cref{cor:brauer-cokernel}.
Naturality of Bright's sequence identifies this isomorphism with pullback.
The assertion for a sequence follows step by step, even when a center
lies on a curve introduced by an earlier blow-up: that curve lies over
the smooth locus, not over a singular point of the current surface.
\end{proof}

Combining this lemma with the marked cascade presentations supplied by
\cite[Corollary~8]{Corti2016} gives, for every \(X\in X_{k,d}\),
\begin{equation}\label{eq:brauer-cascade}
\Br(X)\cong\Br(Z_k).
\end{equation}
It is therefore enough to compute the cokernel on the rigid base
\(Z_k\), or on a surface connected to it by blow-ups at smooth points.

\subsection{The three-, four-, and five-point cases}
\label{app:brauer-unobstructed}

We compute the remaining Brauer groups from explicit models of the
exceptional intersection map. The three-point base is toric, and the
four- and five-point bases come from the polygon construction. For four
points, we compare the polygon model with a one-point blow-up of \(Z_4\).

\paragraph{Three points.}
The three-point base admits a direct toric calculation.
\phantomsection\label{app:brauer-three}
Corti--Heuberger exhibit polygon \(20\) as a toric \(\mathbb Q\)-Gorenstein degeneration in the
family \(X_{3,5}\)
\cite[Section~7, Table~4 and Figure~11]{Corti2016}. We check its
singularities directly before using uniqueness of the base. Let
\(N\cong\Z^2\), and let
\(v_1,\dots,v_5\in N\) be the primitive ray generators of its fan, in
counterclockwise order, with \(V_i\) the corresponding torus-invariant
divisors:
\[
v_1=(1,0),\quad
v_2=(-1,3),\quad
v_3=(-2,3),\quad
v_4=(-1,0),\quad
v_5=(0,-1).
\]
The first three two-dimensional cones have determinant \(3\), while the
remaining two are smooth. The primitive rays inserted to resolve the three
singular cones are
\[
u_{12}=(0,1),\qquad
u_{23}=(-1,2),\qquad
u_{34}=(-1,1).
\]
Let \(E_1,E_2,E_3\) be the corresponding exceptional torus-invariant
curves. Each inserted ray is one third of the sum of its original
neighbors, and the determinants of all adjacent rays in the subdivided
fan equal one. Thus the subdivision is the minimal toric resolution,
obtained by the ray-subdivision procedure of
\cite[Chapters~10--11]{CoxLittleSchenck2011}. Computing the boundary
self-intersections from the fan relations as in
\cite[Theorem~10.4.4]{CoxLittleSchenck2011} gives the following table. The
last column records \(w_{i-1}+w_{i+1}=b_iw_i\), with square
\(-b_i\) for the corresponding curve:
\[
\begin{array}{c|r|l}
\text{curve}&\text{square}&\text{ray relation}\\ \hline
V_1&0&v_5+u_{12}=0\\
E_1&-3&v_1+v_2=3u_{12}\\
V_2&-1&u_{12}+u_{23}=v_2\\
E_2&-3&v_2+v_3=3u_{23}\\
V_3&-1&u_{23}+u_{34}=v_3\\
E_3&-3&v_3+v_4=3u_{34}\\
V_4&-1&u_{34}+v_5=v_4\\
V_5&0&v_4+v_1=0
\end{array}
\]
The inserted curves are therefore precisely the exceptional \((-3)\)-curves,
so the surface has exactly three \(\frac13(1,1)\)-points and no other
singularities.
After the unimodular change of coordinates
\((x,y)\mapsto(x,x+y)\), these are the rays of the polygon cited above.
It belongs to \(X_{3,5}\) (no family of higher Fano index has three singular
points) and hence is isomorphic to \(Z_3\) by the
uniqueness statement recalled in \cref{sec:cascade-foundations}.

The boundary intersections give
\[
\theta(V_1)=e_1,\qquad
\theta(V_2)=e_1+e_2,\qquad
\theta(V_4)=e_3.
\]
Consequently, the dual basis vectors satisfy
\[
e_1=\theta(V_1),\qquad
e_2=\theta(V_2)-\theta(V_1),\qquad
e_3=\theta(V_4),
\]
so \(\theta\) is surjective and the Brauer group vanishes:
\[
\Br(Z_3)=0.
\]
Equivalently, the divisor classes
\[
\OO(-V_1),\qquad
\OO(V_1-V_2),\qquad
\OO(-V_4)
\]
have intersection \(-\delta_{ij}\) with \(E_j\).

\begin{corollary}[KKS decomposition in the three-point cascade]
\label{rem:three-point-toric-descent}
Let \(X\in X_{3,d}\) and \(s=5-d\). Then
\[
\Db(\coh X)=\left\langle
\Db(K(3,1)\text{-mod}),
\Db(K(3,1)\text{-mod}),
\Db(K(3,1)\text{-mod}),F_1,\ldots,F_{2+s}
\right\rangle,
\]
where every \(F_i\) is exceptional.
\end{corollary}

\begin{proof}
The toric base \(Z_3\) has torsion-free class group. Of its five
fixed points, the first three in the displayed fan are quotient
singularities and the last two are smooth. KKS
\cite[Proposition~5.6, Remark~5.8, and Corollary~5.10]{KKS2020}
therefore gives three untwisted adherent subcategories on the resolution
and two components generated by single exceptional objects. Since the
last fixed point is smooth, \(M=\OO\) in \cite[Remark~5.8]{KKS2020}.
By \cite[Corollary~8(3)]{Corti2016}, \(X\) is obtained by
\(s\) blow-ups at smooth points of \(Z_3\), at distinct centers by
\eqref{eq:cascade-distinct}. On resolutions these avoid all three
\((-3)\)-curves. Lemma~\ref{lem:kks-blowups} preserves the three
adherent subcategories under pullback, and Orlov's decomposition with
\(\OO_D\) on the right adds \(s\) exceptional objects. KKS descent and
the identification of the descended adherent components with
\(\Db(K(3,1)\text{-mod})\)
\cite[Theorems~2.12(i) and~3.16]{KKS2020} give the stated decomposition.
\end{proof}

\begin{remark}[Toric surfaces and toric degenerations]
\label{rem:toric-bases-versus-degenerations}
Among nonsmooth del Pezzo surfaces with only \(\frac13(1,1)\)-points,
the toric isomorphism classes are represented by \(S_{1,25/3}\) and by
toric members of the families
\[
X_{1,22/3},\quad X_{1,19/3},\quad
X_{3,5},\quad X_{3,4},\quad X_{6,2}.
\]
These are the toric cases in the Corti--Heuberger classification
\cite[Tables~1, 2, and~4]{Corti2016}. See also
\cite[Table~1]{OnetoPetracci2018}. In particular, \(Z_k\) is toric
precisely for \(k\in\{1,3,6\}\).

The toric degenerations of \(Z_2,Z_4,Z_5\) in
\cite[Table~4, polygons~23, 11, and~7]{Corti2016} have additional
\(T\)-singularities: respectively an \(A_1\)-point, an \(A_2\)-point,
and two \(\frac14(1,1)\)-points. Thus these toric special fibers
are not the rigid bases themselves. The semiorthogonal decompositions for
toric surfaces in \cite[Section~5]{KKS2020} describe the derived categories
of these special fibers.
\end{remark}

\paragraph{The \(k\)-gon construction.}
For the four- and five-point calculations we use the \(k\)-gon construction.
It starts with a smooth del Pezzo surface of degree \(k\) containing a cyclic
configuration \(C_1+\cdots+C_k\) of \((-1)\)-curves, blows up the vertices
\(C_i\cap C_{i+1}\), and contracts the strict transforms of the \(C_i\).
Corti--Heuberger show that the resulting surface has \(k\) singularities of
type \(\frac13(1,1)\), degree \(k/3\), and belongs to the family
\(X_{k,k/3}\) \cite[Section~3.5, ``\(k\)-gons'']{Corti2016}. Thus the
surfaces used below are respectively members of \(X_{4,4/3}\) and
\(X_{5,5/3}\). Let \(D_i\) be the exceptional curve over the
\(i\)-th vertex and \(E_i\) the strict transform of \(C_i\), with cyclic
indices. Direct intersection calculations give
\(E_i^2=-3\), since each \(C_i\) passes through two vertices.
The curves \(E_i\) are pairwise disjoint, and each \(D_i\) meets just
\(E_i\) and \(E_{i+1}\), transversally once. Consequently,
\[
\theta(D_i)=e_i+e_{i+1},
\qquad
\theta(E_i)=-3e_i.
\]
The degree also follows from the discrepancy computation in
the proof of \cref{prop:stack-existence}: the \(k\) blow-ups lower
\(K^2\) from \(k\) to zero, and contraction of the \(k\) disjoint
\((-3)\)-curves raises it to \(k/3\).

\begin{figure}[H]
\centering
\includegraphics[width=\textwidth]{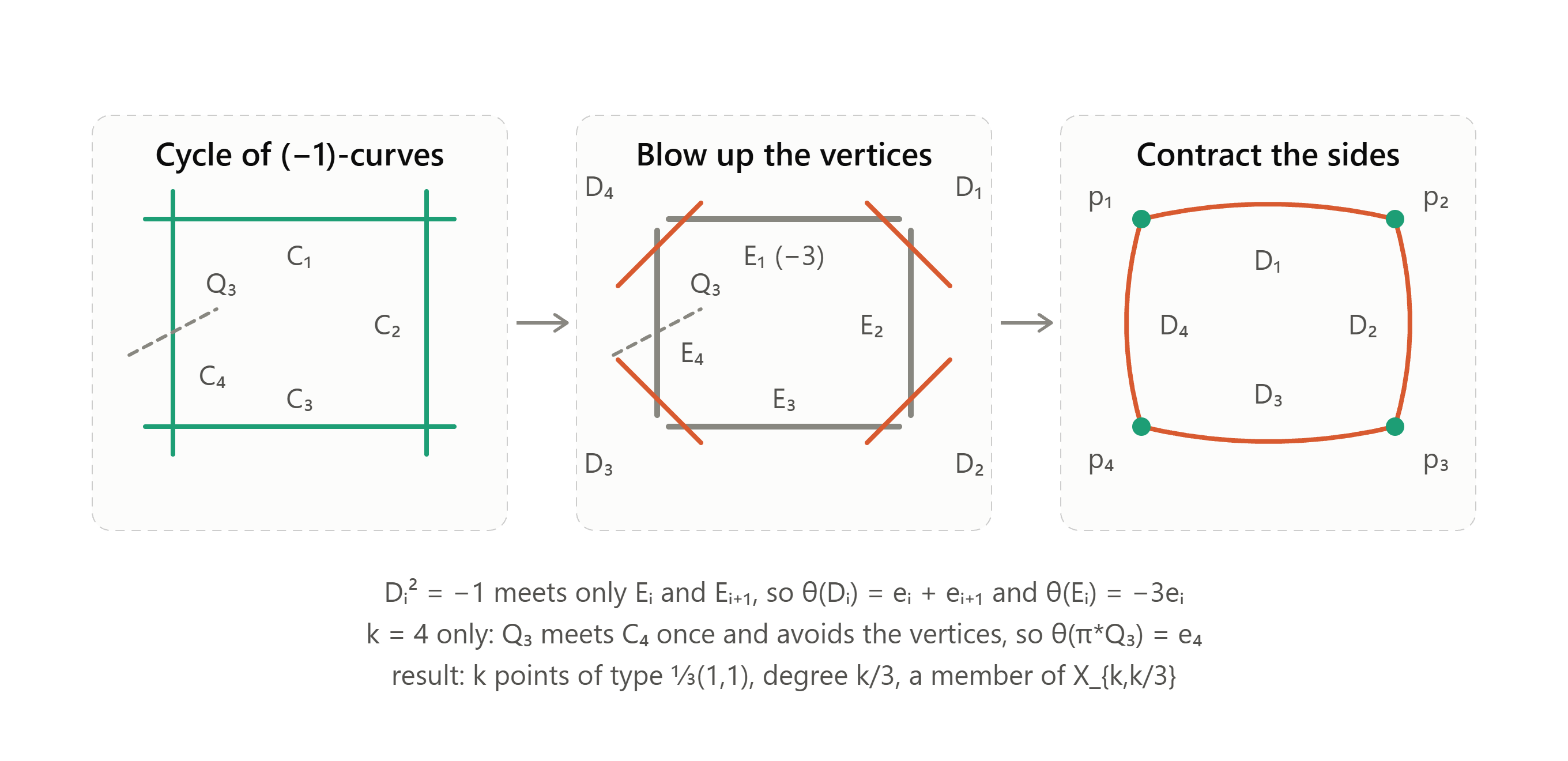}
\caption{The $k$-gon construction for four and five points. The auxiliary curve $Q_3$ is used only in the four-point calculation.}
\label{fig:kgon-four-point-case}
\end{figure}

\paragraph{Five points.}
For the five-point base, consider the degree-five del Pezzo surface
\(\Bl_{p_1,\dots,p_4}\Pp^2\), with exceptional classes \(Q_i\) and line
class \(H\). The curves
\[
Q_1,\quad H-Q_1-Q_2,\quad Q_2,\quad H-Q_2-Q_3,\quad H-Q_1-Q_4
\]
have adjacent intersections one and nonadjacent intersections zero, so
they form a cyclic pentagon of \((-1)\)-curves. Blowing up its vertices
and contracting the strict transforms of its sides gives a surface in
\(X_{5,5/3}\), which is isomorphic to the rigid base \(Z_5\) by
uniqueness. The five vectors
\(e_i+e_{i+1}\) are the columns of
\[
A_5=I+P_5,
\]
where \(P_5\) is the cyclic permutation matrix. A direct calculation gives
\[
\det(A_5)=\prod_{\zeta^5=1}(1+\zeta)=2.
\]
The quotient \(\Z^5/A_5\Z^5\) has order \(2\). Since the image of \(\theta\) also
contains \(3\Z^5\), and multiplication by \(3\) is invertible on this
order-two quotient, we have
\[
A_5\Z^5+3\Z^5=\Z^5.
\]
Hence \(\theta\) is surjective and the Brauer group vanishes:
\[
\Br(Z_5)=0.
\]

\paragraph{Four points.}
Unlike the five-point calculation above, the four-point calculation uses
a member of \(X_{4,4/3}\), obtained from the rigid base \(Z_4\), which
represents the unique isomorphism class in \(X_{4,7/3}\), by one blow-up
at a smooth point
\cite[Corollary~8]{Corti2016}. By \cref{lem:brauer-cascade-invariance},
it has the same cokernel as \(Z_4\). The vectors \(e_i+e_{i+1}\) span a
rank-three subgroup, so one additional divisor class is needed. To avoid
confusing the exceptional classes on a degree-four del Pezzo surface with the
\((-3)\)-curves above, write
\[
S=\Bl_{p_1,\dots,p_5}\Pp^2
\]
with basis \(H,Q_1,\dots,Q_5\), where \(H^2=1\), \(Q_i^2=-1\), and all
other pairings between basis elements vanish. Consider
\[
\begin{aligned}
C_1&=Q_1,\\
C_2&=H-Q_1-Q_2,\\
C_3&=Q_2,\\
C_4&=2H-Q_1-Q_2-Q_3-Q_4-Q_5.
\end{aligned}
\]
Their self-intersections are all \(-1\). The adjacent intersections are
\[
C_1\cdot C_2=C_2\cdot C_3=C_3\cdot C_4=C_4\cdot C_1=1,
\]
and the nonadjacent intersections are
\[
C_1\cdot C_3=0,
\qquad
C_2\cdot C_4=0.
\]
Their intersection graph is therefore a \(4\)-cycle. In addition,
\[
Q_3\cdot C_1=Q_3\cdot C_2=Q_3\cdot C_3=0,
\qquad
Q_3\cdot C_4=1.
\]
For general \(p_1,\dots,p_5\), these intersections are transverse. In
particular, \(Q_3\) avoids the four vertices of the cycle. If \(\pi\) is their
blow-up, the projection formula gives
\[
\theta(\pi^*Q_3)=e_4.
\]
Using this as the starting dual basis vector, the remaining basis vectors are
obtained successively:
\[
\begin{aligned}
e_1&=\theta(D_4)-e_4,\\
e_2&=\theta(D_1)-e_1,\\
e_3&=\theta(D_2)-e_2.
\end{aligned}
\]
These identities make \(\theta\) surjective. Since the \(4\)-gon representative
is a one-point blow-up of \(Z_4\),
\cref{lem:brauer-cascade-invariance} gives
\[
\Br(Z_4)=0.
\]

\subsection{The six-point Brauer obstruction}
\label{app:brauer-six}

Unlike the five smaller cascades, the six-point cascade carries a
nonzero Brauer obstruction. We first exhibit a relation among the toric
boundary divisors that produces an order-three class in the saturation quotient
\(\Lambda^{\mathrm{sat}}/\Lambda\subset A_\Lambda\). This proves
that \(\Br(Z_6)\ne0\). We then compute
the full image of the intersection map: a single alternating-sum
congruence modulo three describes it, showing that the Brauer group is
exactly \(\Z/3\).

For \(Z_6\), the rigid base representing the unique isomorphism class in
\(X_{6,2}\), Corti--Heuberger describe the fan in
\cite[Section~2.2.5]{Corti2016}. We compute the intersection map by first
describing the minimal toric resolution and then determining the lattice
generated by the images of its boundary divisors. Let \(N\cong\Z^2\), and let
\(v_1,\dots,v_6\in N\) be the primitive ray generators of its fan, with
\(V_i\) the corresponding torus-invariant divisors:
\[
\begin{aligned}
v_1&=(1,0),&
v_2&=(-1,3),&
v_3&=(-2,3),\\
v_4&=(-1,0),&
v_5&=(1,-3),&
v_6&=(2,-3),
\end{aligned}
\]
in counterclockwise order. The unimodular change
\((x,y)\mapsto(x,x+y)\) sends these vertices to those of
\cite[Figure~1]{Corti2016}. Indices are cyclic modulo \(6\) throughout. Every
adjacent cone has determinant \(3\), and hence defines a cyclic quotient
singularity of order \(3\). The primitive rays inserted in the
minimal toric resolution are
\[
u_i=\frac{v_i+v_{i+1}}3,
\qquad
i\in\Z/6\Z.
\]
The inserted rays are explicitly
\[
\begin{aligned}
u_1&=(0,1),&
u_2&=(-1,2),&
u_3&=(-1,1),\\
u_4&=(0,-1),&
u_5&=(1,-2),&
u_6&=(1,-1).
\end{aligned}
\]
Let \(E_i\) be the exceptional torus-invariant curve corresponding to
\(u_i\). Since
\[
\det(v_i,u_i)=\det(u_i,v_{i+1})=1,
\]
the subdivision is smooth and gives the minimal toric resolution
\cite[Chapters~10--11]{CoxLittleSchenck2011}. The resolved boundary is the
alternating cycle
\[
V_1,E_1,V_2,E_2,\dots,V_6,E_6,
\]
and the ray relations are
\[
u_{i-1}+u_i=v_i,
\qquad
v_i+v_{i+1}=3u_i.
\]
Therefore, \cite[Theorem~10.4.4]{CoxLittleSchenck2011} gives
\[
V_i^2=-1,\qquad E_i^2=-3.
\]
Consequently, the minimal resolution of each singularity consists of one
\((-3)\)-curve, so all six singularities are of type \(\frac13(1,1)\).

\paragraph{A nonzero class in the saturation quotient.}
An order-three class in the saturation quotient comes from a divisibility relation. The principal-divisor
formula
\(\operatorname{div}(\chi^m)=\sum_\rho\langle m,v_\rho\rangle D_\rho\)
of \cite[Theorem~4.1.3]{CoxLittleSchenck2011}, applied to the character
\(m=(0,1)\), gives
\[
\begin{aligned}
0={}&E_1+3V_2+2E_2+3V_3+E_3-E_4\\
&\hspace{3em}-3V_5-2E_5-3V_6-E_6.
\end{aligned}
\]
Rearranging the relation, set
\[
L=-V_2-V_3+V_5+V_6-E_2+E_5
\in\Pic(\widetilde Z_6).
\]
Then
\begin{equation}\label{eq:brauer-six-glue}
E_1-E_2+E_3-E_4+E_5-E_6=3L.
\end{equation}

\begin{figure}[H]
\centering
\includegraphics[width=\textwidth]{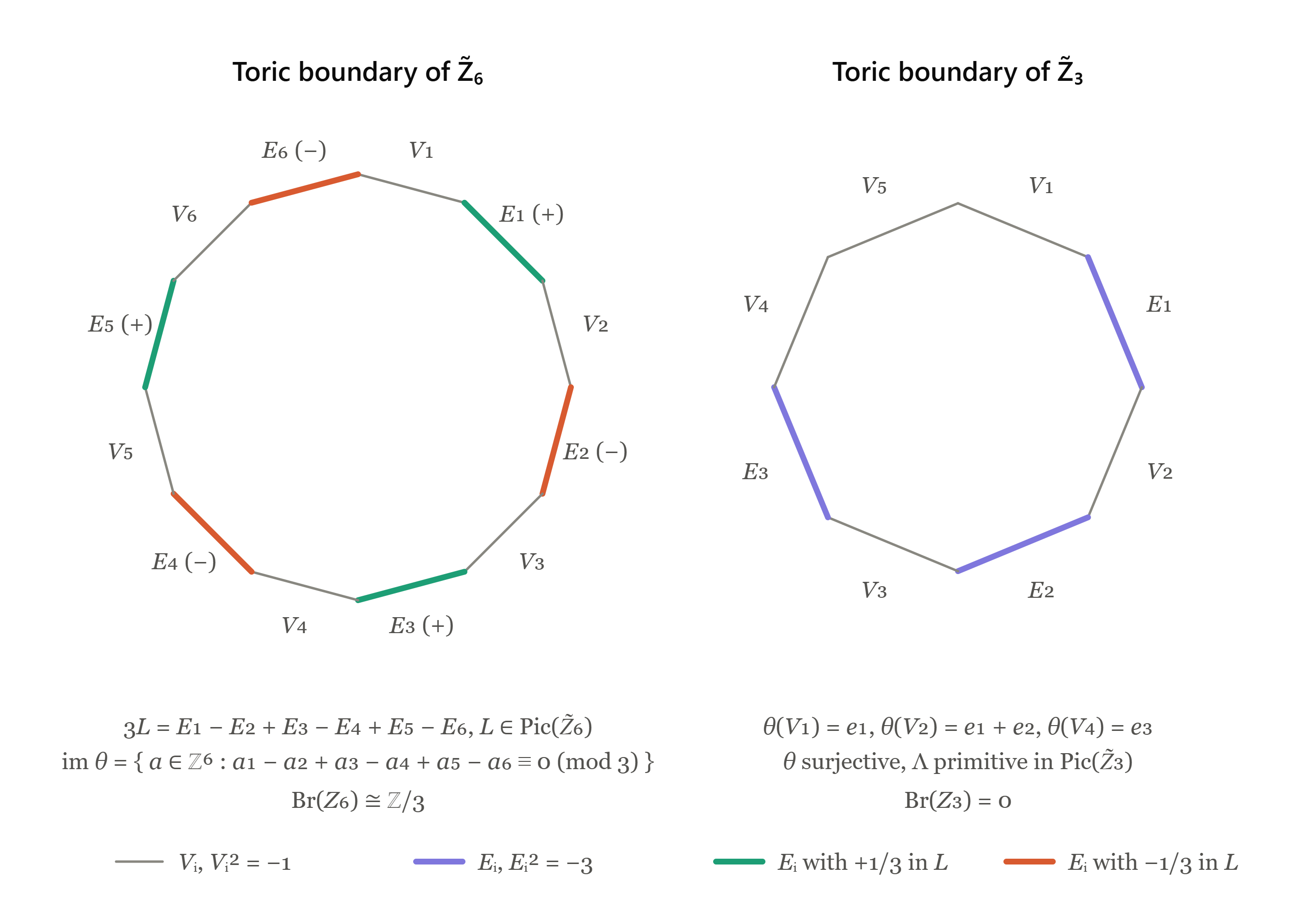}
\caption{Toric boundaries of $\widetilde Z_6$ and $\widetilde Z_3$. The right-hand panel recalls the three-point calculation in Section~\ref{app:brauer-unobstructed}.}
\label{fig:brauer-toric-boundaries-z6-z3}
\end{figure}

Thus \(L\) lies in \(\Lambda^{\mathrm{sat}}\) but not in \(\Lambda\):
the \(E_i\) are linearly independent, and its coefficients in this basis
are \(\pm1/3\). Its class has order three and realizes the alternating
pattern allowed by \cref{lem:brauer-glue-vectors}.
Moreover, \(L^2=-2\), confirming the isotropy
of its class for the discriminant pairing in
\cref{def:brauer-glue}.
In particular, \cref{lem:brauer-glue-duality} already implies that
\(\Br(Z_6)\ne0\). We next determine the whole group.

\paragraph{The cokernel.}
For a complete toric variety, the torus-invariant prime divisors generate
\(\Cl(\widetilde Z_6)\), and smoothness gives
\(\Cl(\widetilde Z_6)=\Pic(\widetilde Z_6)\)
\cite[Theorem~4.1.3]{CoxLittleSchenck2011}. Hence the toric boundary generates
\(\Pic(\widetilde Z_6)\), and its image under \(\theta\) is generated by
\[
\theta(E_i)=-3e_i,
\qquad
\theta(V_j)=e_{j-1}+e_j,
\]
with cyclic indices. Put \(c_j=e_{j-1}+e_j\), where \(e_0=e_6\), and let
\(A\subset\Z^6\) be the subgroup generated by \(c_2,\dots,c_6\).
These five vectors, together with \(e_6\), form a basis of \(\Z^6\):
the matrix of their columns is lower triangular with diagonal entries one.
Thus \(A\) is primitive of rank five. Moreover,
\[
c_1=c_2-c_3+c_4-c_5+c_6,
\]
so \(A\) is also the span of all six images \(\theta(V_j)\). Consider the
surjective homomorphism
\[
\lambda:\Z^6\longrightarrow\Z,
\qquad
(a_1,\dots,a_6)\longmapsto
a_1-a_2+a_3-a_4+a_5-a_6.
\]
Every \(c_j\) lies in \(\ker(\lambda)\). Since both
\(A\) and \(\ker(\lambda)\) have rank five and \(A\) is primitive,
this inclusion is an equality. The image of the full toric boundary is
therefore
\[
\operatorname{im}(\theta)
=
A+3\Z^6
=
\ker(\lambda)+3\Z^6.
\]
This sum is the kernel of the reduction of \(\lambda\) modulo \(3\). Indeed,
if \(\lambda(a)=3q\), then \(a-3q e_1\in\ker(\lambda)\), while the reverse
inclusion is immediate. Consequently,
\[
\operatorname{im}(\theta)
=
\left\{
(a_1,\dots,a_6)\in\Z^6:
a_1-a_2+a_3-a_4+a_5-a_6\equiv0\pmod3
\right\}.
\]
The homomorphism \(\lambda\bmod3:\Z^6\to\Z/3\) is surjective, so the
cokernel and Brauer group are
\[
\operatorname{coker}(\theta)\cong\Z/3,
\qquad
\Br(Z_6)\cong\Z/3.
\]
By \cref{lem:brauer-glue-duality}, the saturation quotient also has order three.
It is therefore generated by the class of \(L\) in
\eqref{eq:brauer-six-glue}. No class supported on three exceptional curves
occurs in the saturation quotient on \(Z_6\).
Indeed, its two nonzero elements are \([L]\) and \(-[L]\), whose
coefficient vectors modulo three are
\(\pm(1,-1,1,-1,1,-1)\). Both have all six entries nonzero.
Here the global toric divisor relations select the class with alternating
coefficients on all six exceptional curves, whereas the computations for
\(k=3,4,5\) exclude even the numerically permitted classes supported on three
exceptional curves.

For the adherence convention we use classes with
\(L_i\cdot E_j=-\delta_{ij}\). This differs from the dual-basis convention
of \cref{cor:brauer-cokernel} only by replacing each \(L_i\) by \(-L_i\).
We keep this sign convention in the twisted six-point construction below.

\paragraph{The obstruction to untwisted adherence.}
In particular, there cannot exist divisor classes \(L_i\) satisfying
\[
L_i\cdot E_j=-\delta_{ij},
\]
since such classes would make \(\theta\) surjective.

By \cref{thm:kks-obstruction}, an untwisted adherent decomposition would
force \(\Cl(Z_6)\) to be torsion-free, equivalently
\(\Br(Z_6)=0\). The computed nonzero class rules this out.
The KKS result for toric surfaces instead gives a semiorthogonal decomposition
of a Brauer-twisted category on \(Z_6\). Under the cascade blow-ups, the
Brauer class is pulled back and the corresponding twisted adherent
decomposition is preserved, as proved below.

\begin{corollary}[KKS decomposition of the twisted category in the six-point cascade]
\label{cor:six-point-twisted-descent}
Let \(X\in X_{6,d}\), let \(\eta:X\to Z_6\) be a marked cascade
presentation, and put \(s=2-d\). There is a generator
\(\alpha\in\Br(Z_6)\) such that, with
\(\alpha_X=\eta^*\alpha\),
\[
\Db(\coh(X,\alpha_X))=
\left\langle
\underbrace{\Db(K(3,1)\text{-mod}),\ldots,
\Db(K(3,1)\text{-mod})}_{6\ \text{components}},
F_1,\ldots,F_s
\right\rangle.
\]
The \(F_i\) are exceptional objects of the twisted derived category.
Their list is empty when \(s=0\). The class \(\alpha_X\) is nonzero.
\end{corollary}

\begin{proof}
The proof has three steps: we first identify the twisted adherence data on the
resolution of \(Z_6\), then pull these data through the cascade blow-ups using
\cref{lem:kks-blowups}, and finally identify the descended twisted components
via the KKS twisted pushforward.

Write \(\pi:\widetilde Z_6\to Z_6\) for the resolution and use
the counterclockwise boundary order \(V_1,E_1,\ldots,V_6,E_6\)
above. In KKS notation, \(E_{i,0}=V_i\), \(E_{i,1}=E_i\),
\(m_i=1\), and \(d_{i,1}=-E_i^2=3\). The last fixed point
corresponds to the cone between \(v_6\) and \(v_1\).
The construction in
\cite[Proposition~5.6, equations~(5.4)--(5.6), p.~510]{KKS2020}
gives the component
\(\langle\mathcal L_{i,0},\mathcal L_{i,0}(E_i)\rangle\), where
\[
\mathcal L_{i,0}=\OO\Bigl(\sum_{j<i}(V_j+E_j)+V_i\Bigr),
\qquad
\deg(\mathcal L_{i,0}|_{E_i})=1+\delta_{i6}.
\]
Indeed, \(V_i\) meets \(E_i\) once, and the sum contributes only
\(V_1\cdot E_6=1\) when \(i=6\). Our adherence convention
\(\deg(\mathcal L_{i,0}|_{E_i})=3+b_i-1\) therefore gives
\(b_i=-1+\delta_{i6}\), exactly
\cite[Proposition~5.6, equation~(5.7)]{KKS2020}.
Adjunction gives \(K_{\widetilde Z_6}\cdot E_i=1\), so tensoring
with \(\OO(K_{\widetilde Z_6})\) adds one to every adherence parameter:
\[
b=(-1,-1,-1,-1,-1,0),\qquad
b+\theta(K_{\widetilde Z_6})=e_6.
\]
Write \(B:\Lambda^\vee\to\Br(Z_6)\) for the connecting homomorphism in
the exact sequence of \cref{rem:brauer-obstruction}. By exactness,
\(B\circ\theta=0\). The displayed equality gives \(\alpha=B(b)=B(e_6)\),
the class used in the twisted descent of
\cite[equation~(5.8) and Theorem~5.9]{KKS2020}. Since
\(\lambda(e_6)=-1\), its image in \(\Br(Z_6)\cong\Z/3\) is nonzero and
hence is a generator.

All cyclic shifts, and the opposite cyclic ordering, give the same Brauer
class in this case. Indeed, KKS \cite[Lemma~5.16]{KKS2020} gives
\(\beta_{i+1}=a_{i+1}\beta_i\) and \(\beta_i=a_i\beta_i^{\prime}\). Since every
\(a_i=1\), these classes all equal \(\alpha\). By KKS
\cite[Theorem~5.9]{KKS2020}, the corresponding orderings yield twisted
decompositions.

There are no smooth fixed points, so the decomposition of
\(\Db(\coh(Z_6,\alpha))\) has six \(K(3,1)\)-components and no components
generated by single exceptional objects. Under the blow-ups
\(\widetilde X\to\widetilde Z_6\),
the distinct smooth cascade centers lift away from the six contracted
\((-3)\)-curves and their successive strict transforms. The new Orlov
exceptional divisors are disjoint from these curves, so their restrictions
vanish. Hence
Lemma~\ref{lem:kks-blowups} preserves the adherent components and their
adherence parameters. Orlov's decomposition with \(\OO_D\) on the right adds
\(s\) components generated by single exceptional objects. The intersection-map description in
\cref{lem:brauer-cascade-invariance} identifies the resulting
Brauer class with \(\eta^*\alpha\), still a generator.
More explicitly, a vector bundle on \(\widetilde Z_6\) realizing
the adherence parameters, as in \cite[Corollary~4.10(1)]{KKS2020}, pulls back to
one with the same restrictions on the six exceptional curves.

Applying the twisted pushforward of \cite[equation~(4.14)]{KKS2020}
and then \cite[Theorem~4.19]{KKS2020} to each adherent component
identifies its image with \(\Db(K(3,1)\text{-mod})\).
For the additional exceptional objects contributed by Orlov's blow-up
formula, the restrictions to the six
exceptional curves are zero, so \cite[Lemma~4.14]{KKS2020}
identifies each with the twisted pullback of its pushforward.
Full faithfulness of twisted pullback preserves their exceptionality, and
adjunction preserves their semiorthogonality to the preceding components.
Fullness follows from \cite[Lemma~4.13]{KKS2020} and the full
resolution decomposition. This is twisted descent, not
ordinary pushforward under \cref{thm:kks-descent}.
For the other generator \( -\alpha_X\), let \(\mathcal A\) be an Azumaya
algebra representing \(\alpha_X\), and let \(\omega_X^\bullet\) be a
dualizing complex of \(X\). Grothendieck duality
\(\RHom_{\mathcal A}(-,\mathcal A\otimes\omega_X^\bullet)\) gives a
contravariant equivalence to \(\Db(\coh(X,-\alpha_X))\), so it reverses
the decomposition. Since \(K(3,1)\) is commutative, vector-space duality
identifies its bounded derived category of finite-dimensional modules
with its opposite. Duality preserves exceptionality, so the reversed
decomposition has the same six local components and exceptional objects.
\end{proof}

Together with \cref{cor:one-point-kks-sod,rem:three-point-toric-descent},
this proves \cref{thm:descent-cascades-intro}.

\subsection{\texorpdfstring{Proof of \cref{thm:brauer-cascades-intro}}{Proof of the Brauer group theorem}}
\label{app:brauer-classification}

For \(k=1,2\), \cref{cor:brauer-one-two} gives \(\Br(X)=0\) for every
individual surface \(X\) in the corresponding cascades, including the
higher-index families \(B_{1,16/3}\) and \(B_{2,8/3}\). In particular,
\(\Br(Z_1)=\Br(Z_2)=0\).
For \(k=3,5\), \cref{app:brauer-unobstructed} computes the Brauer groups
on a toric surface and a pentagon surface in the respective base families.
Each is the corresponding \(Z_k\), by uniqueness of the starting surface
in \cref{sec:cascade-foundations}. For \(k=4\), the same subsection uses a quadrilateral
surface in \(X_{4,4/3}\), not in the base family. Its contraction
to a smooth point of \(Z_4\) and \cref{lem:brauer-cascade-invariance} give
\(\Br(Z_4)=0\). For \(k=6\), the toric computation in
\cref{app:brauer-six} gives \(\Br(Z_6)\cong\Z/3\). Thus
\[
\Br(Z_k)
\cong
\begin{cases}
0, & 1\leq k\leq5,\\
\Z/3, & k=6.
\end{cases}
\]

For every \(X\in X_{k,d}\) with \(k\geq3\), equation
\eqref{eq:brauer-cascade} gives \(\Br(X)\cong\Br(Z_k)\). Together with
\cref{cor:brauer-one-two} for the one- and two-point cascades, this proves
the stated formula for every surface in the six Corti--Heuberger cascades
and the additional higher-index families \(B_{1,16/3}\) and \(B_{2,8/3}\).

This completes the proof of \Cref{thm:brauer-cascades-intro}. \qed


\section{Explicit exceptional collections on the rigid cascade bases}
\label{app:explicit-bases}

This appendix constructs marked blow-up presentations of the minimal
resolutions \(\widetilde Z_k\) and applies the line-bundle form of Orlov's
formula to obtain their exceptional collections. The singular surfaces are
described by Corti--Heuberger and Appendix~\ref{app:brauer}. Here the marked
presentations specify the ordered centers needed for the blow-up formula.
The following subsections treat the six rigid bases, using toric contractions
for three and six points, an elementary transformation for four points, and
the pentagon construction for five. Further subsections treat the two
higher-index families and extend the collections along the cascades to prove
\cref{thm:explicit-cascades-main}.

The following guide records the centers and geometric steps for the resolutions
in Table~\ref{tab:rigid-base-invariants}. All modifications are ordinary
blow-ups at smooth points. The phrase ``infinitely near'' describes the location of a
center relative to the starting surface \(Y_0\) listed there.
\begin{center}
\small
\renewcommand{\arraystretch}{1.2}
\begin{tabular}{@{}c p{0.46\linewidth} p{0.42\linewidth}@{}}
\(k\)&Centers&Key geometric step\\ \hline
1&None&Use the negative section.\\
2&One point and two infinitely near points&Create a second \((-3)\)-curve.\\
3&One torus-fixed point and three infinitely near torus-fixed points&Reverse
four boundary contractions.\\
4&Five distinct points and two infinitely near points&Perform an elementary
transformation.\\
5&Four plane points, then five vertices (four infinitely near)&Separate the
pentagon curves.\\
6&Four torus-fixed points and four infinitely near torus-fixed points&Contract to
a hexagon, then a square.
\end{tabular}
\end{center}
We use the following notation throughout. For a sequence of ordinary
point blow-ups
\[
Y_m\xrightarrow{\beta_m}Y_{m-1}\longrightarrow\cdots
\xrightarrow{\beta_1}Y_0,
\qquad \sigma=\beta_1\circ\cdots\circ\beta_m,
\]
let \(D_i\subset Y_i\) be the exceptional curve created by \(\beta_i\),
and let \(\pi_i:Y_m\to Y_i\) be the natural morphism. Set
\begin{equation}\label{eq:explicit-total-transform-line-bundles}
\Delta_i:=\pi_i^*D_i,
\qquad
M_i:=\OO_{Y_m}(\Delta_i).
\end{equation}
Thus \(\Delta_i\) is the total transform on \(Y_m\) of the exceptional
divisor at its creation stage. When later centers are infinitely near,
it may differ from the final strict transform of \(D_i\). If
\(\langle\OO_{Y_0},L_2,\ldots,L_t\rangle\) is a full exceptional
collection of line bundles on \(Y_0\),
repeated application of \cref{lem:point-blow-up-collection} gives
\begin{equation}\label{eq:explicit-line-bundle-order}
\left\langle
\OO_{Y_m},M_m,\ldots,M_1,
\sigma^*L_2,\ldots,\sigma^*L_t
\right\rangle.
\end{equation}
This is a full exceptional collection of line bundles. The standard full
exceptional collections on \(Y_0=\Pp^2\) and
\(Y_0=\F_n\) are recalled in \cref{sec:orlov}. Here \(H\) is a line
class, and \(C_0,f\) are the section and fiber classes with \(C_0^2=-n\).
We denote the collection on
\(\widetilde Z_k\) by \(\mathscr A_k\), the canonical stack of
\(Z_k\) by \(\mathcal Z_k\), and its Ishii--Ueda embedding by
\(\Phi_k\). Every \(\mathscr A_k\) displayed below is a full
exceptional collection of line bundles. In each case,
\cref{prop:stack-existence} turns
\(\mathscr A_k\) into a full exceptional collection on \(\mathcal Z_k\) by applying \(\Phi_k\)
and prepending the \(k\) local objects. The lengths are recorded
in Table~\ref{tab:rigid-base-invariants}, so we list only the collections
on the resolutions below.

\subsection{The one-point base}
\label{app:explicit-base-one}

The minimal resolution of the rigid starting surface
\(Z_1=S_{1,25/3}=\Pp(1,1,3)\) is \(\F_3\), with
exceptional curve \(C_0\). No point blow-ups are needed, and we take
\[
\mathscr A_1=
\left\langle\OO,\OO(f),\OO(C_0+3f),\OO(C_0+4f)\right\rangle.
\]

\subsection{The two-point base}
\label{app:explicit-base-two}

The directed minimal model program of \cite[Section~6]{Corti2016} contains
a two-point branch starting at the chosen rigid starting surface \(Z_2\)
associated with the index-one family \(X_{2,17/3}\), and ending at \(\Pp(1,1,3)\):
\[
Z_2\longrightarrow T\longrightarrow\Pp(1,1,3),
\qquad \Sing(T)=\left\{\tfrac13(1,1),A_1\right\}.
\]
For the exceptional collection we need a marked blow-up presentation of the
minimal resolution, which we now construct directly.

\begin{proposition}\label{prop:explicit-base-two}
Choose \(q_1\in\F_3\setminus C_0\), and let \(f_0\) be its ruling
fiber. Blow up \(q_1\), with exceptional curve \(D_1\), and then choose
\(q_2\in D_1\) away from the strict transform of \(f_0\). Blow up
\(q_2\), with exceptional curve \(D_2\). Finally, choose \(q_3\) on
the strict transform of \(D_1\), away from both \(D_2\) and the strict
transform of \(f_0\), and blow it up. Then
\[
\widetilde Z_2\simeq
Y_3=\Bl_{q_3}\Bl_{q_2}\Bl_{q_1}\F_3.
\]
The two curves contracted to the quotient singularities are the untouched
\(C_0\) and the final strict transform \(A\) of \(D_1\).
\end{proposition}

\begin{proof}
Write \(Y_i\) for the surface after the first \(i\) blow-ups and
\(T_i\) for the total transform of the class of \(D_i\) on \(Y_3\).
Then \(T_iT_j=-\delta_{ij}\), and pullbacks from \(\F_3\) are
orthogonal to all \(T_i\). On \(Y_2\) the strict transform
\(A_2\) of \(D_1\) has square \(-2\), since the second blow-up is centered
at \(q_2\in D_1\). The last center lies on
\(A_2\), but on neither of its two neighboring curves. Consequently
\[
A=T_1-T_2-T_3,\qquad A^2=-3,\qquad C_0A=0.
\]
The strict transform \(F'=f-T_1\) of \(f_0\) has square \(-1\), and
it meets \(C_0\) and \(A\) once each. The curves \(D_2,D_3\)
remain \((-1)\)-curves, each meeting \(A\) once and disjoint from
\(C_0\) and \(F'\). The dual graph records these transverse
intersections (each edge has intersection multiplicity one):
\[
\begin{array}{ccccc}
&&&&D_2\,(-1)\\[-2pt]
&&&&|\\[-2pt]
C_0\,(-3)&\text{---}&F'\,(-1)&\text{---}&A\,(-3)\\[-2pt]
&&&&|\\[-2pt]
&&&&D_3\,(-1)
\end{array}
\]

To check that contracting \(C_0\) and \(A\) gives the required del
Pezzo surface, consider the rational divisor
\[
P=-K_{Y_3}-\tfrac13(C_0+A)
 =\tfrac53 C_0+5f-\tfrac43T_1-\tfrac23T_2-\tfrac23T_3.
\]
The intersection form in the orthogonal total-transform basis gives
\[
P^2=\tfrac{17}{3},\qquad
PC_0=PA=0,\qquad PF'=\tfrac13,\qquad PD_2=PD_3=\tfrac23.
\]
Every other irreducible curve is the strict transform of a fiber not
through \(q_1\), or of a curve \(B\sim aC_0+bf\) with
\(a\geq1\): an irreducible curve on \(\F_3\) with \(a=0\) is a
fiber, and curves exceptional over \(\F_3\) were already listed.
Since \(B\ne C_0\), nonnegative intersection of
distinct irreducible curves gives \(BC_0=b-3a\geq0\), hence
\(b\geq3a\). Let \(m_1,m_2,m_3\) be its multiplicities at the
successive centers.
Intersection with \(f_0\) gives \(m_1\leq a\), while nonnegative
intersection of its final strict transform with
\(A=T_1-T_2-T_3\) gives
\(\widetilde B A=m_1-m_2-m_3\geq0\), hence
\(m_2+m_3\leq m_1\). Consequently,
\[
P\cdot\widetilde B
 =\tfrac53b-\tfrac43m_1-\tfrac23(m_2+m_3)
 \geq\tfrac53b-2a\geq3a>0.
\]
On a fiber not through \(q_1\), the intersection is \(5/3\).
Thus \(P\) is nef and big, and its only null curves are \(C_0,A\).
The pair \((Y_3,(C_0+A)/3)\) is klt, so the basepoint-free theorem
\cite[Section~3.2]{Kollar1998} makes a sufficiently divisible multiple
of \(P\) semiample. Its associated contraction has precisely these
two disjoint \((-3)\)-curves as exceptional locus and gives a normal
surface \(Z\) with two \(\frac13(1,1)\)-points. The discrepancy
formula identifies \(P\) with the pullback of the ample divisor
\(-K_Z\).

There is also no ambiguity about the Fano index. The divisor-class group
of \(Z\) is the quotient of \(\Pic(Y_3)\) by \(C_0,A\), so it is
free with basis \(f,T_2,T_3\). In this quotient
\[
-K_Z=5f-2T_2-2T_3
\]
is primitive. The surface therefore has Fano index one, and
\cite[Theorem~4 and Remark~7]{Corti2016} identifies it with \(Z_2\).
Since its exceptional curves have square \(-3\), \(Y_3\) is its
minimal resolution.
\end{proof}

With \(\sigma_2:Y_3\to\F_3\) the composite, the seven-object collection is
\[
\mathscr A_2=\left\langle \OO,M_3,M_2,M_1,
\sigma_2^*\OO(f),
\sigma_2^*\OO(C_0+3f),\sigma_2^*\OO(C_0+4f)\right\rangle.
\]

\subsection{The three-point base}
\label{app:explicit-base-three}

For the three-point rigid starting surface \(Z_3\) associated with the
index-one family \(X_{3,5}\), use the fan and boundary notation of
\cref{app:brauer-three}. Its minimal resolution has boundary
\(V_1,E_1,V_2,E_2,V_3,E_3,V_4,V_5\), with self-intersections
\[
(0,-3,-1,-3,-1,-3,-1,0).
\]
To obtain a marked blow-up presentation over \(\F_3\), we contract, in this
order, the successive images of \(V_4,V_3,E_3,E_2\). The following table lists
the neighboring rays at the actual stage of each contraction:
\[
\begin{array}{c|c|c}
\text{curve}&\text{neighboring rays}&\text{ray relation}\\ \hline
V_4&u_{34},v_5&v_4=u_{34}+v_5\\
V_3&u_{23},u_{34}&v_3=u_{23}+u_{34}\\
E_3&u_{23},v_5&u_{34}=u_{23}+v_5\\
E_2&v_2,v_5&u_{23}=v_2+v_5
\end{array}
\]
In every row the determinant of the two neighbors is one. The ray-sum
criterion therefore verifies that the curve being removed is a
\((-1)\)-curve and that its contraction is a smooth toric blow-down
\cite[Theorem~10.4.4]{CoxLittleSchenck2011}. The remaining fan has rays
\(v_1,u_{12},v_2,v_5\), with boundary squares \((0,-3,0,3)\). It
is the fan of \(\F_3\), with negative section corresponding to
\(u_{12}\).

Reversing these four contractions prescribes the centers: insert
\(u_{23}\) in the cone \(\langle v_2,v_5\rangle\), then
\(u_{34}\) in \(\langle u_{23},v_5\rangle\), then \(v_3\) in
\(\langle u_{23},u_{34}\rangle\), and finally \(v_4\) in
\(\langle u_{34},v_5\rangle\). Each center is the intersection of
the two indicated boundary curves. This gives
\(\sigma_3:\widetilde Z_3\to\F_3\) and the full exceptional collection
\[
\mathscr A_3=\left\langle \OO,M_4,M_3,M_2,M_1,
\sigma_3^*\OO(f),
\sigma_3^*\OO(C_0+3f),\sigma_3^*\OO(C_0+4f)\right\rangle.
\]
The creation stages for \(M_1,M_2,M_3,M_4\) correspond respectively
to the rays \(u_{23},u_{34},v_3,v_4\). In particular, \(M_1\) and \(M_2\)
are defined by the total transforms of the exceptional divisors at those
stages, not by their final strict transforms.

\subsection{The four-point base}
\label{app:explicit-base-four}

The quadrilateral construction in \cref{app:brauer-unobstructed} gives a surface
in \(X_{4,4/3}\), not in the index-one family \(X_{4,7/3}\). The chosen
rigid starting surface associated with \(X_{4,7/3}\) is \(Z_4\). The following
proposition contracts
an additional \((-1)\)-curve and performs an elementary transformation of
\(\F_1\) to give a marked blow-up presentation of \(\widetilde Z_4\)
over \(\F_2\).

\begin{proposition}\label{prop:explicit-base-four}
Let \(S=\Bl_{p_1,\ldots,p_5}\Pp^2\), with the points general, and
let \(H,Q_1,\ldots,Q_5\) be its standard divisor classes. Use the
quadrilateral \(C_1,\ldots,C_4\) of \cref{app:brauer-unobstructed}.
Blow up the vertices \(r_i=C_i\cap C_{i+1}\), cyclically, with
exceptional curves \(R_i\), obtaining \(Y\). The strict transform
of the line tangent at \(p_1\) to the conic through the five points
is a \((-1)\)-curve
\[
D=H-Q_1-R_4.
\]
Its contraction \(Y\to Y'\) identifies \(Y'\) with
\(\widetilde Z_4\). Moreover, an elementary transformation at \(r_4\)
gives an explicit seven-point blow-up presentation
\(\sigma_4:\widetilde Z_4\to\F_2\). Its centers, in order, are
the transported points \(p_2,p_3,p_4,p_5,r_1,r_2,r_3\): the first
five are distinct on \(\F_2\), and the last two are infinitely near
\(p_2\).
\end{proposition}

\begin{proof}
All divisor classes on \(Y\) in the calculations below are total
pullbacks, except where strict transforms are explicitly specified.
Write \(E_i\) for the strict transforms of the quadrilateral sides.
The intersection calculations in \cref{app:brauer-unobstructed} show
that these are pairwise disjoint \((-3)\)-curves.
The contraction \(Y\to X_{4,4/3}\) is the \(4\)-gon construction
of \cite[Section~3.5]{Corti2016}.

The conic through five general points is smooth. Its tangent line at
\(p_1\) contains none of \(p_2,\ldots,p_5\), by B\'ezout, and
meets \(Q_1\) at \(r_4\) after the first blow-up. This point is
different from \(r_1\), the direction of the secant through \(p_2\).
Thus \(D\) is the strict transform of an irreducible smooth
rational curve. In the orthogonal
blow-up basis we have
\[
D^2=-1,\qquad (-K_Y)D=1,\qquad DE_i=0\quad(1\leq i\leq4).
\]
These distinct irreducible curves are therefore disjoint. The image of
\(D\) on \(X_{4,4/3}\) lies in the smooth locus and is a floating
\((-1)\)-curve. Its contraction is a log del Pezzo surface with four
unchanged \(\frac13(1,1)\)-points and degree \(7/3\), hence is
\(Z_4\) by the classification in \cref{sec:cascade-foundations}. Contracting
\(D\) on \(Y\) leaves the four \(E_i\) unchanged and gives its
minimal resolution \(Y'\).

To identify the starting ruled surface, first blow up only \(p_1\).
This gives \(\F_1\), with negative section \(Q_1\) and ruling
class \(H-Q_1\). The tangent direction of the conic specifies
\(r_4\) on \(Q_1\). Put
\(W=\Bl_{r_4}\F_1\). The strict transform of that ruling fiber
is \(D\). Contracting it is the elementary transformation of the
ruled surface at a point on its negative section. On \(\F_1\), we have
\(Q_1^2=-1\). Since \(r_4\) lies on \(Q_1\), its strict transform on
\(W\) has square \(-2\). It is disjoint from \(D\), so contracting \(D\)
does not change this self-intersection. The descended ruling therefore has
a section of square \(-2\), and the classification of ruled surfaces
identifies the result with \(\F_2\)
\cite[Chapter~V, Section~2]{Hartshorne1977}.

It remains to account for the marked centers. On \(W\), perform the blow-ups at \(p_2,p_3,p_4,p_5\),
then at \(r_1,r_2,r_3\). All these centers, including the two
directions \(r_2,r_3\) on the exceptional curve over \(p_2\),
lie away from the strict transform of \(D\): the tangent line
avoids the four points, and \(r_1\ne r_4\). The isomorphism
\[
W\setminus D\ \simeq\ \F_2\setminus\{t\},
\]
where \(t\) is the image of \(D\), transports these centers and
the points infinitely near to \(p_2\) to the corresponding blow-up of
\(\F_2\). Blow up the images of \(p_2,p_3,p_4,p_5\), followed by the
transported \(r_1,r_2,r_3\). The transported \(r_1\) lies on the negative
section of \(\F_2\). Thus the transported sequence gives the asserted
morphism \(\sigma_4:\widetilde Z_4\to\F_2\).
\end{proof}

Using precisely this order of seven centers in
\eqref{eq:explicit-total-transform-line-bundles}, we obtain
\[
\mathscr A_4=\left\langle \OO,M_7,\ldots,M_1,
\sigma_4^*\OO(f),
\sigma_4^*\OO(C_0+2f),\sigma_4^*\OO(C_0+3f)\right\rangle.
\]
Here \(C_0\) is the negative section on the new \(\F_2\).

\subsection{The five-point base}
\label{app:explicit-base-five}

For the five-point rigid starting surface \(Z_5\) associated with the
index-one family \(X_{5,5/3}\), the four plane blow-ups followed by
the five vertex blow-ups of the pentagon give a marked presentation of
\(\widetilde Z_5\) and hence the collection obtained by taking total transforms of line bundles.
On the degree-five del Pezzo surface
\(S_5=\Bl_{p_1,\ldots,p_4}\Pp^2\), start with the pentagon and its
intersection data from \cref{app:brauer-unobstructed}, and blow up its five distinct vertices
\(r_i=C_i\cap C_{i+1}\), with exceptional curves \(R_i\).
The strict transforms are the disjoint \((-3)\)-curves described in
\cref{app:brauer-unobstructed}.
By the \(5\)-gon construction of \cite[Section~3.5]{Corti2016},
their contraction is a member of \(X_{5,5/3}\). By the
Corti--Heuberger classification, it is isomorphic to \(Z_5\), and consequently
\[
\widetilde Z_5\simeq
\Bl_{r_1,\ldots,r_5}\Bl_{p_1,\ldots,p_4}\Pp^2.
\]
Here the vertex blow-ups are on \(S_5\), not at nine distinct points
of the plane: \(r_1,r_5\) lie on \(Q_1\) over \(p_1\),
\(r_2,r_3\) lie on \(Q_2\) over \(p_2\), and \(r_4\) maps to a
fifth distinct plane point. After each vertex blow-up, the remaining
vertices are taken on the strict transforms of the indicated curves.
Order the centers as
\(p_1,\ldots,p_4,r_1,\ldots,r_5\), and write \(\sigma_5\)
for their composite. The collection of twelve line bundles is
\[
\mathscr A_5=\left\langle \OO,M_9,\ldots,M_1,
\sigma_5^*\OO(H),\sigma_5^*\OO(2H)\right\rangle.
\]

\subsection{The six-point base}
\label{app:explicit-base-six}

For the six-point rigid starting surface \(Z_6\), use the fan and notation
of \cref{app:brauer-six}. We give the ordered toric contractions that
produce a marked blow-up presentation of \(\widetilde Z_6\) over \(\F_0\).
We first contract
\(V_1,\ldots,V_6\), in that order. At every stage, the neighbors
of the ray \(v_i\) being removed are still \(u_{i-1},u_i\), and
\[
u_{i-1}+u_i=v_i,\qquad \det(u_{i-1},u_i)=1.
\]
Thus every contraction is a smooth toric blow-down. After any initial set \(I\) of
these contractions, the self-intersection of the image of \(E_j\) is
\[
-3+\#\bigl(I\cap\{j,j+1\}\bigr),
\]
with cyclic indices. After all six, the fan consists of the \(u_i\)
and every boundary curve has square \(-1\). This is the smooth toric
del Pezzo surface of degree six: its anticanonical divisor intersects
each invariant curve positively, and its square is six.

Next contract the opposite boundary curves with rays \(u_1\) and
\(u_4\). They are disjoint, and their ray relations are
\[
u_1=u_6+u_2,\qquad u_4=u_3+u_5,
\qquad \det(u_6,u_2)=\det(u_3,u_5)=1.
\]
Both are therefore \((-1)\)-curves at the moment of contraction.
The remaining rays are \(u_2,u_3,-u_2,-u_3\), and
\(\det(u_2,u_3)=1\). In this lattice basis the fan is the square
fan of \(\F_0=\Pp^1\times\Pp^1\).

Reversing the sequence prescribes eight toric point blow-ups: first
insert \(u_4\) between \(u_3,u_5\), then \(u_1\) between
\(u_6,u_2\), and then insert \(v_6,v_5,\ldots,v_1\), each
between \(u_{i-1},u_i\). The insertions of \(u_4,u_1,v_6,v_3\)
blow up the four original fixed points of \(\F_0\). The remaining
four blow-ups are at points infinitely near to those producing \(u_4\)
and \(u_1\). The insertions of \(v_5,v_4\) have centers on the successive
strict transforms of the curve with ray \(u_4\), and those of \(v_2,v_1\)
on the successive strict transforms of the curve with ray \(u_1\).
Let
\(\sigma_6:\widetilde Z_6\to\F_0\) be their composite. With
\(f_1,f_2\) the two ruling classes on \(\F_0\), the resulting
twelve-object collection is
\[
\mathscr A_6=\left\langle \OO,M_8,\ldots,M_1,
\sigma_6^*\OO(f_1),\sigma_6^*\OO(f_2),
\sigma_6^*\OO(f_1+f_2)\right\rangle.
\]

\subsection{The higher-index families}
\label{app:higher-fano-index}

The two higher-index families lie outside the index-one cascades, so we treat
their marked blow-up presentations separately. Their resolutions start from
\(\mathbb P^2\) and \(\mathbb F_3\), respectively.

A quasismooth member \(X_4\in B_{1,16/3}\) is a weighted hypersurface
of degree four:
\[
X_4\subset\mathbb P(1,1,1,3)
\]
\cite{Corti2016}. Write the weighted coordinates as
\([x:y:z:w]\), with \(\deg(x,y,z,w)=(1,1,1,3)\). A weighted
degree-four equation has the form
\[
wL(x,y,z)+Q_4(x,y,z)=0,
\]
where \(L\) is linear and \(Q_4\) is quartic. Projection from the
stacky point \([0:0:0:1]\) defines a birational map
\[
X_4\dashrightarrow\mathbb P^2.
\]
Away from the line \(L=0\), its inverse is given by
\(w=-Q_4/L\). The inverse has base locus \(L=Q_4=0\), which for a
quasismooth member consists of four distinct points \(p_1,\ldots,p_4\)
on the line \(L=0\). Blowing them up resolves the birational map and
gives
\[
\widetilde X_4\simeq
\operatorname{Bl}_{p_1,\ldots,p_4}\mathbb P^2.
\]

Let \(H\) be the pullback of a line and let \(D_1,\ldots,D_4\) be the
exceptional curves. The strict transform of \(L=0\) is
\[
E=H-D_1-\cdots-D_4,
\qquad E^2=-3.
\]
The curve \(E\) is contracted to the unique \(\frac13(1,1)\)-point.

The line-bundle form of Orlov's blow-up formula in
\cref{lem:point-blow-up-collection} gives the collection
\[
\mathcal A_{B_1}=
\left\langle
\mathcal O,\mathcal O(D_1),\ldots,\mathcal O(D_4),
\mathcal O(H),\mathcal O(2H)
\right\rangle
\]
which is a full exceptional collection on \(\widetilde X_4\) of length \(7\). Let
\(\mathcal B_{1,16/3}\) be the canonical stack of \(X_4\), and let \(\Phi\) be the
Ishii--Ueda embedding. Then
\[
D^b(\operatorname{coh}\mathcal B_{1,16/3})
=\left\langle\mathcal E_p,\Phi(\mathcal A_{B_1})\right\rangle
\]
is a full exceptional collection of length \(8\), in agreement with
\cref{cor:length}.

A quasismooth member \(X_6\in B_{2,8/3}\) is a weighted hypersurface
of degree six \cite{Corti2016}:
\[
X_6\subset\mathbb P(1,1,3,3).
\]
Write the coordinates as \([x:y:z:w]\), with
\(\deg(x,y,z,w)=(1,1,3,3)\). A weighted degree-six equation has the
form
\[
q_2(z,w)+z f_3(x,y)+w g_3(x,y)+h_6(x,y)=0,
\]
where \(q_2\) is quadratic in \(z,w\), \(f_3,g_3\) are cubic, and
\(h_6\) is sextic in \(x,y\). The singular line \(x=y=0\) meets
\(X_6\) at the two zeros of \(q_2(z,w)=0\) on that line. For a
quasismooth member these zeros are distinct, and they give the two
\(\frac13(1,1)\)-points.

Consider the rational map
\[
(x:y):X_6\dashrightarrow\mathbb P^1.
\]
After a linear change of \(z,w\), take \(q_2=zw\). Over a point
\([a:b]\in\mathbb P^1\), writing \(x=at\), \(y=bt\), and \(T=t^3\),
the fiber is the conic
\[
zw+T\bigl(zf_3(a,b)+wg_3(a,b)\bigr)+T^2h_6(a,b)=0
\]
in the coarse projective plane with coordinates \([T:z:w]\). Its
discriminant is a binary sextic in \(a,b\). Quasismoothness forces this
discriminant to have six simple zeros: a multiple zero would make the
total surface singular at the node of that fiber. Resolving the two
quotient singularities extends the map to a conic bundle
\(\widetilde X_6\to\mathbb P^1\), whose two exceptional \((-3)\)-curves
are disjoint sections. Since \(zw=0\) is the intersection of each fiber
with \(T=0\), the two components of every reducible fiber meet different
sections. They meet transversely, and their sum is a fiber, so each is a
\((-1)\)-curve. Contracting, in each of the six fibers, the component
disjoint from one chosen section yields a smooth \(\mathbb P^1\)-bundle. The
chosen section retains self-intersection \(-3\), so this bundle is
\(\mathbb F_3\), with that section as \(C_0\). The other section meets
each contracted component once, so its image has self-intersection \(3\)
and belongs to \(|C_0+3f|\). Reversing these contractions yields
\[
\widetilde X_6\simeq
\operatorname{Bl}_{q_1,\ldots,q_6}\mathbb F_3,
\]
where the six centers lie in distinct fibers on the section
\(C_\infty\in|C_0+3f|\). Its strict transform satisfies
\(\widetilde C_\infty^2=-3\), since \(C_\infty^2=3\) and the six
centers lie on it. The curves \(C_0\) and \(\widetilde C_\infty\) are
the two curves contracted to the quotient singularities.

Let \(D_1,\ldots,D_6\) be the exceptional curves of the six blow-ups,
and denote the pullbacks of \(C_0\) and \(f\) from \(\mathbb F_3\) by
the same symbols. Starting with the standard full exceptional collection
on \(\mathbb F_3\),
\[
\left\langle
\mathcal O,
\mathcal O(f),
\mathcal O(C_0+3f),
\mathcal O(C_0+4f)
\right\rangle,
\]
the blow-up construction for collections of line bundles gives the full exceptional collection
on \(\widetilde X_6\)
\[
\mathcal A_{B_2}=
\left\langle
\mathcal O,\mathcal O(D_1),\ldots,\mathcal O(D_6),
\mathcal O(f),\mathcal O(C_0+3f),\mathcal O(C_0+4f)
\right\rangle
\]
of length \(10\). Let
\(\mathcal B_{2,8/3}\) be the canonical stack of \(X_6\). The Ishii--Ueda
embedding gives
\[
D^b(\operatorname{coh}\mathcal B_{2,8/3})
=\left\langle
\mathcal E_{p_1},\mathcal E_{p_2},\Phi(\mathcal A_{B_2})
\right\rangle,
\]
a full exceptional collection of length \(12\), in agreement with
\cref{cor:length}.

The numerical data of the two higher-index families are summarized
in Table~\ref{tab:higher-index-families}.

\begin{table}[htbp]
\centering
\small
\renewcommand{\arraystretch}{1.1}
\begin{tabular}{@{}lccccc@{}}
Family & \(K^2\) & \(Y_0\) & \(\#\text{ blow-ups}\) & \(N\) & \(\ell\) \\ \hline
\(B_{1,16/3}\) & \(\frac{16}{3}\) & \(\mathbb P^2\) & 4 & 7 & 8 \\
\(B_{2,8/3}\) & \(\frac{8}{3}\) & \(\mathbb F_3\) & 6 & 10 & 12
\end{tabular}
\caption{Higher-index families, marked blow-up presentations, and collection lengths.}
\label{tab:higher-index-families}
\end{table}

\begin{remark}
Members of \(B_{1,16/3}\) and \(X_{1,16/3}\) have the same numerical
invariants \(k=1\) and \(K_X^2=16/3\), but belong to different deformation types. For a
member \(X\in B_{1,16/3}\), let \(\nu:Y\to X\) contract
\(E=H-\sum_{i=1}^4D_i\) from \(Y=\operatorname{Bl}_4\mathbb P^2\).
The discrepancy relation \(K_Y=\nu^*K_X-\frac13E\) yields
\[
3\nu^*(-K_X)
=3(-K_Y)-E
=2\left(4H-\sum_{i=1}^4D_i\right),
\]
reflecting Fano index \(2\). For a member \(X\in X_{1,16/3}\) of
Fano index one, let \(Y\) be the three-point blow-up of \(\mathbb F_3\)
away from \(C_0\). Then
\[
3(-K_Y)-C_0
=5C_0+15f-3\sum_{i=1}^3D_i.
\]
This class is not divisible by \(2\) in \(\operatorname{Pic}(Y)\).
Thus \((k,K^2)\) does not determine the deformation family, while the
position of the exceptional curve relative to the canonical class in the
Picard lattice retains additional information.
\end{remark}

\subsection{\texorpdfstring{Extension along the cascades and proof of
\cref{thm:explicit-cascades-main}}{Extension along the cascades and proof of the main theorem}}
\label{app:explicit-cascade-extension}

The marked blow-up presentations in
\cref{app:explicit-base-one,app:explicit-base-two,app:explicit-base-three,app:explicit-base-four,app:explicit-base-five,app:explicit-base-six}
construct the collections of line bundles \(\mathscr A_k\) on
\(\widetilde Z_k\). We extend them along the cascades by applying the
blow-up formula at the \(s\) distinct smooth points of \(Z_k\) described in
\cref{sec:cascade-foundations}. The following corollary proves
\cref{thm:explicit-cascades-main}. We use \(A^{(k)}_j\) for the entries of
\(\mathscr A_k\).

\begin{corollary}\label{cor:explicit-cascade-collections}
Let \(X\in X_{k,d}\), and put
\(s=K_{Z_k}^2-K_X^2\). The case \(s=0\) is allowed when
\(2\leq k\leq6\). For \(k=1\), one has \(s\geq1\).
Choose a marked cascade presentation by
\(s\) distinct smooth points of \(Z_k\), and let
\(\gamma:\widetilde X\to\widetilde Z_k\) be the induced map of
minimal resolutions. Write the exceptional collection on \(\widetilde Z_k\) as
\[
\mathscr A_k=\left\langle
\OO_{\widetilde Z_k},A^{(k)}_2,\ldots,A^{(k)}_{N_k}
\right\rangle,
\]
with \(N_k\) as in Table~\ref{tab:rigid-base-invariants}, and let
\(D_1,\ldots,D_s\) be the disjoint exceptional curves of \(\gamma\).
The distinctness of the centers follows from
\eqref{eq:cascade-distinct}. Then
\[
\Db(\coh\widetilde X)=
\left\langle
\OO_{\widetilde X},
\OO_{\widetilde X}(D_1),\ldots,
\OO_{\widetilde X}(D_s),
\gamma^*A^{(k)}_2,\ldots,
\gamma^*A^{(k)}_{N_k}
\right\rangle
\]
is a full exceptional collection of line bundles. The line bundles
\(\OO_{\widetilde X}(D_i)\) are mutually orthogonal, so their internal
ordering is immaterial. If \(\XX\) is the canonical stack
and \(\Phi\) its Ishii--Ueda embedding, the ordered collection
\[
\begin{aligned}
\Db(\coh\XX)=\bigl\langle{}
&\mathcal E_{p_1},\ldots,\mathcal E_{p_k},
\Phi(\OO_{\widetilde X}),\\
&\Phi(\OO_{\widetilde X}(D_1)),\ldots,
\Phi(\OO_{\widetilde X}(D_s)),\\
&\Phi(\gamma^*A^{(k)}_2),\ldots,
\Phi(\gamma^*A^{(k)}_{N_k})
\bigr\rangle
\end{aligned}
\]
is full exceptional. Its length is given by \cref{cor:length}.
\end{corollary}

\begin{proof}
The centers are distinct by \eqref{eq:cascade-distinct}, so the curves
\(D_1,\ldots,D_s\) are pairwise disjoint exceptional
\((-1)\)-curves. Applying \cref{lem:point-blow-up-collection} successively gives
the displayed full exceptional collection of line bundles on \(\widetilde X\).
Applying \cref{prop:stack-existence} with \(n_i=3\) for \(1\leq i\leq k\), and
adjoining the \(k\) residual-gerbe objects, gives the stated full
exceptional collection on the canonical stack. This proves
\cref{thm:explicit-cascades-main}.
\end{proof}


\section*{Acknowledgements}
I am grateful to my advisor, Sergey Galkin, for his guidance, support, and many valuable discussions. I thank Sergei Burkin for his encouragement and support, and for introducing me to the modern language and techniques of higher algebra and higher category theory. I thank Nicolau Saldanha and Carlos Tomei for supporting my travel to the RGAS School 2026, where part of this work was further developed. I also thank Alexander Kuznetsov for a helpful conversation during the school and for pointing me to the work of Karmazyn--Kuznetsov--Shinder, which motivated the comparison with the singular surface in Section~\ref{sec:kks}. I am grateful to Alexey Elagin for useful discussions at IMPA about derived categories of singular varieties. Finally, I thank Alessio Corti for helpful correspondence on the Corti--Heuberger cascades, and in particular for pointing out the relevance of the Reid--Suzuki construction to \(X_{10}\).

\section*{Funding}
The author was supported by the Funda\c{c}\~ao Carlos Chagas Filho de Amparo \`a Pesquisa do Estado do Rio de Janeiro (FAPERJ) through the Nota 10 Excellence Scholarship.

\section*{Declaration of generative AI and AI-assisted technologies in the writing process}
During the preparation of this work, the author used Prism, ChatGPT (OpenAI), and Claude (Anthropic) to assist with language editing, organization, and manuscript formatting. After using these tools, the author reviewed and edited the content as needed and takes full responsibility for the content of the publication.

\printbibliography

\end{document}